\documentclass[a4paper]{amsart}
\pdfoutput=1
\usepackage{amsmath,amssymb}

\usepackage{subfiles}

\DeclareSymbolFont{bbold}{U}{bbold}{m}{n}
\DeclareSymbolFontAlphabet{\mathbbold}{bbold}
\newcommand{\bbone}{\mathbbold{1}}

\DeclareSymbolFontAlphabet{\amsbb}{AMSb}
\usepackage{mbboard}
\renewcommand{\mathbb}[1]{\amsbb{#1}}

\usepackage[shortlabels]{enumitem}
  \setitemize[1]{leftmargin=2em}
  \setenumerate[1]{leftmargin=2em}
\usepackage{tikz-cd}
\usepackage[unicode,colorlinks=true,linktocpage=true,citecolor=blue]{hyperref}
\usepackage{graphicx,color}
\usepackage[capitalise]{cleveref}
\crefformat{equation}{(#2#1#3)}

\newtheorem{theorem}{Theorem}
\numberwithin{theorem}{subsection}
\newtheorem{thm}[theorem]{Theorem}
\newtheorem{proposition}[theorem]{Proposition}
\newtheorem{propn}[theorem]{Proposition}

\newtheorem{corollary}[theorem]{Corollary}
\newtheorem{cor}[theorem]{Corollary}
\newtheorem{lemma}[theorem]{Lemma}
\theoremstyle{definition}
\newtheorem{definition}[theorem]{Definition}
\newtheorem{defn}[theorem]{Definition}
\newtheorem{example}[theorem]{Example}
\newtheorem{examples}[theorem]{Examples}
\newtheorem{notation}[theorem]{Notation}
\newtheorem{remark}[theorem]{Remark}
\newtheorem{observation}[theorem]{Observation}
\newtheorem{construction}[theorem]{Construction}
\newtheorem{warning}[theorem]{Warning}
\newtheorem{variant}[theorem]{Variant}

\providecommand{\op}{\mathrm{op}}

\providecommand{\xint}{\mathrm{int}}

\newcommand{\xHom}{\operatorname{Hom}}

\newcommand{\xFun}{\operatorname{Fun}}

\newcommand{\act}{\name{act}}
\newcommand{\Act}{\name{Act}}
\newcommand{\el}{\name{el}}

\DeclareMathOperator{\colimP}{colim}
\DeclareMathOperator{\Ar}{Ar}
\DeclareMathOperator{\ev}{ev}
\DeclareMathOperator{\Emb}{Emb}
\newcommand{\colim}{\mathop{\colimP}}

\newcommand{\PSh}{\operatorname{P}}
\newcommand{\xMap}{\operatorname{Map}}
\newcommand{\Map}{\operatorname{Map}}

\newcommand{\xE}{\mathcal{E}}

\newcommand{\xS}{\mathcal{S}}

\newcommand{\xV}{\mathcal{V}}

\newcommand{\id}{\operatorname{id}}

\newcommand{\xF}{\mathbb{F}}
\newcommand{\xFs}{\xF_{*}}
\newcommand{\xFn}{\xFs^{\natural}}

\newcommand*\cocolon{%
  \nobreak
  \mskip6mu plus1mu
  \mathpunct{}%
  \nonscript
  \mkern-\thinmuskip
  {:}%
  \mskip2mu
  \relax
}
\newcommand{\icat}{$\infty$-category}
\newcommand{\icatl}{$\infty$-categorical}
\newcommand{\igpd}{$\infty$-groupoid}
\newcommand{\igpds}{$\infty$-groupoids}
\newcommand{\icats}{$\infty$-categories}
\newcommand{\iopd}{$\infty$-operad}
\newcommand{\iopds}{$\infty$-operads}
\newcommand{\isoto}{\xrightarrow{\sim}}

\newcommand{\xto}[1]{\xrightarrow{#1}}

\newcommand{\xfrom}[1]{\xleftarrow{#1}}
\newcommand{\csquare}[8]{ %
\[ %
\begin{tikzpicture} %
\matrix (m) [matrix of math nodes,row sep=3em,column sep=2.5em,text height=1.5ex,text depth=0.25ex] %
{ #1 \pgfmatrixnextcell #2 \\ %
  #3 \pgfmatrixnextcell #4 \\ }; %
\path[->,font=\footnotesize] %
(m-1-1) edge node[auto] {$#5$} (m-1-2)%
(m-1-1) edge node[left] {$#6$} (m-2-1)%
(m-1-2) edge node[auto] {$#7$} (m-2-2)%
(m-2-1) edge node[below] {$#8$} (m-2-2);%
\end{tikzpicture}%
\]%
}

\makeatletter
\def\@tocline#1#2#3#4#5#6#7{\relax
  \ifnum #1>\c@tocdepth 
  \else
    \par \addpenalty\@secpenalty\addvspace{#2}%
    \begingroup \hyphenpenalty\@M
    \@ifempty{#4}{%
      \@tempdima\csname r@tocindent\number#1\endcsname\relax
    }{%
      \@tempdima#4\relax
    }%
    \parindent\z@ \leftskip#3\relax \advance\leftskip\@tempdima\relax
    \rightskip\@pnumwidth plus4em \parfillskip-\@pnumwidth
    #5\leavevmode\hskip-\@tempdima
      \ifcase #1
       \or \hskip -1em \or \hskip 1em \or \hskip 3em \else \hskip 5em \fi%
      #6\nobreak\relax
    \hfill\hbox to\@pnumwidth{\@tocpagenum{#7}}
      \par
    \nobreak
    \endgroup
  \fi}
\makeatother

\newcommand{\name}[1]{\ensuremath{\text{\textup{#1}}}}

\newcommand{\simp}{\bbDelta}
\newcommand{\Dop}{\simp^{\op}}
\newcommand{\Dopn}{\simp^{\op,\natural}}
\newcommand{\Dnop}{\simp^{n,\op}}
\newcommand{\Dnopn}{\simp^{n,\op,\natural}}
\newcommand{\Tq}[1]{\bbTheta_{#1}}
\newcommand{\TqX}[2]{\bbTheta_{#1,#2}}
\newcommand{\Tqop}[1]{\Tq{#1}^{\op}}
\newcommand{\TqXop}[2]{\TqX{#1}{#2}^{\op}}
\newcommand{\Tqopn}[1]{\Tq{#1}^{\op,\natural}}
\newcommand{\Tn}{\Tq{n}}

\newcommand{\Tnopn}{\Tqopn{n}}

\newcommand{\bbO}{\bbOmega}
\newcommand{\bbOop}{\bbO^{\op}}

\newcommand{\Seg}{\name{Seg}}
\newcommand{\Mon}{\name{Mon}}

\newcommand{\Algd}{\name{Algd}}

\newcommand{\Fun}{\name{Fun}}

\newcommand{\blank}{\text{\textendash}}

\newcommand{\Cat}{\name{Cat}}
\newcommand{\CatI}{\Cat_{\infty}}
\newcommand{\LCatI}{\widehat{\Cat}_{\infty}}
\newcommand{\IFF}{if and only if}
\newcommand{\AlgPatt}{\name{AlgPatt}}

\newcommand{\Alg}{\name{Alg}}

\newcommand{\ie}{i.e.\@}

\newcommand{\cf}{cf.\@}

\newcommand{\angled}[1]{\langle #1 \rangle}
\usepackage{eucal}

\newcommand{\bX}{\mathbb{X}}
\newcommand{\bU}{\mathbb{U}}
\newcommand{\bO}{\mathbb{O}}
\newcommand{\bUop}{\bU^{\op}}
\newcommand{\bUopn}{\bU^{\op,\natural}}
\newcommand{\ObX}{\mathcal{O}_{\bX}}
\newcommand{\bY}{\mathbb{Y}}

\DeclareMathOperator{\Span}{Span}

\newcommand{\RFib}{\name{RFib}}
\newcommand{\LFib}{\name{LFib}}
\newcommand{\Day}{\name{Day}}

\usepackage[utf8]{inputenc}

\newcommand{\actto}{\rightsquigarrow}
\newcommand{\intto}{\rightarrowtail}

\tikzcdset{
   active/.style={rightsquigarrow},
   inert/.style={tail}
 }

\usepackage{mathtools}

\makeatletter
\MHInternalSyntaxOn
\def\MT_leftarrow_fill:{%
  \arrowfill@\leftarrow\relbar\relbar}
\def\MT_rightarrow_fill:{%
  \arrowfill@\relbar\relbar\rightarrow}
\newcommand{\xrightleftarrows}[2][]{\mathrel{%
  \raise.55ex\hbox{%
    $\ext@arrow 0359\MT_rightarrow_fill:{\phantom{#1}}{#2}$}%
  \setbox0=\hbox{%
    $\ext@arrow 3095\MT_leftarrow_fill:{#1}{\phantom{#2}}$}%
  \kern-\wd0 \lower.55ex\box0}}
\MHInternalSyntaxOff
\makeatother

\newcommand{\Cart}{\name{Cart}}

\newcommand{\pbcorner}{\arrow[phantom]{dr}[at start,description]{\lrcorner}}

\title{Enriched Homotopy-Coherent Structures}
\author{Hongyi Chu}
\address{Max Planck Institute for Mathematics, Bonn, Germany}

\author{Rune Haugseng}
\address{Norwegian University of Science and Technology (NTNU), Trondheim, Norway}
\date{\today}
\begin{document}
\begin{abstract}
  We introduce a general notion of enrichment for homotopy-coherent
  algebraic structures described by Segal conditions, using the
  framework of ``algebraic patterns'' developed in our previous work.
  This recovers several known examples of enriched structures,
  including enriched \icats{}, enriched \iopds{}, and enriched
  $\infty$-properads. As new examples we discuss enriched modular
  \iopds{}, enriched $n$-fold \icats{}, and a non-iterative definition of enriched $(\infty,n)$-categories.
\end{abstract}

\maketitle
\tableofcontents

\section{Introduction}

Many homotopy-coherent algebraic and categorical structures can be
conveniently described in terms of \emph{Segal conditions}, and in
previous work \cite{patterns,freealg} we introduced a general
framework for such structures, called \emph{algebraic patterns}: An
algebraic pattern consists of an \icat{} $\mathcal{O}$ equipped with
an ``inert--active'' factorization system and a collection of
``elementary'' objects, and a \emph{Segal $\mathcal{O}$-object} in a complete \icat{}
$\mathcal{C}$ is a functor $F \colon \mathcal{O} \to \mathcal{C}$
satisfying the Segal-type limit condition that for every $O \in
\mathcal{O}$ we have
\[ F(O) \simeq \lim_{E \in \mathcal{O}^{\el}_{X/}} F(X),\]
where $\mathcal{O}^{\el}_{X/}$ is the \icat{} of inert maps from $X$
to elementary objects. Structures that can be described as Segal $\mathcal{O}$-spaces (meaning Segal $\mathcal{O}$-objects in the \icat{} of spaces) include \icats{}, \iopds{}, $\infty$-properads, $(\infty,n)$-categories, and cyclic or modular \iopds{}.

However, for all of these structures it is often interesting to
consider variants that are \emph{enriched} in some symmetric monoidal
\icat{} (such as vector spaces, spectra, or chain complexes). Our goal
in the present paper is to set up a general notion of enriched
algebraic structures controlled by an algebraic pattern. More
precisely, we will define an \emph{enrichable pattern} to consist of
an algebraic pattern $\mathcal{O}$ together with a suitable map
$\mathcal{O} \to \xFs$, where $\xFs$ is the category of finite pointed
sets. Here the fibre $\mathcal{O}_{0}$ over the one-point set inherits
an algebraic pattern structure, and our enriched Segal
$\mathcal{O}$-objects, which following \cite{Macpherson} we will call
\emph{$\mathcal{O}$-algebroids}, will have an underlying Segal
$\mathcal{O}_{0}$-space. (In the simplest cases, such as for enriched
\icats{} and $\infty$-operads, this is just a space --- the underlying
space of objects.)  We then define $\mathcal{O}$-algebroids in a
symmetric monoidal, or more generally $\mathcal{O}$-monoidal, \icat{}
$\mathcal{V}$ as algebras in $\mathcal{V}$ for a family of patterns
$\mathcal{O}_{\bX}$ parametrized by a Segal $\mathcal{O}_{0}$-space
$\bX$. (These are \emph{cartesian} patterns in the sense of
\cite{freealg}, meaning that the limits in their Segal conditions are
indexed by finite sets, and so it makes sense to talk about algebras
for them in $\mathcal{V}$.)

In \S\ref{sec:general} we introduce enrichable patterns and their
algebroids, and develop some general theory. The main results are as
follows:
\begin{itemize}
\item $\mathcal{O}$-algebroids in spaces are equivalent to Segal
  $\mathcal{O}$-spaces (\cref{thm:algdspc}).
\item Algebroids in a presheaf \icat{} with respect to Day convolution
  can also be described as Segal spaces for a certain pattern (\cref{AlgdDayisSeg}), which
  allows us to show that the \icat{} of $\mathcal{O}$-algebroids in a presentably
  $\mathcal{O}$-monoidal \icat{} is presentable (\cref{cor:algdpres}).
\item If $\mathcal{O}$ is extendable (see \cite[\S 8]{patterns}), we
  give a description of free $\mathcal{O}$-algebroids
  (\cref{cor:freealg}).
\end{itemize}
We also discuss functoriality of algebroids in
\S\ref{subsec:algdfun}. Many of the proofs are generalizations of
arguments from \cite{enriched,enrcomp,ChuHaugseng} for particular
cases of enrichable patterns.

We then describe a number of examples of enrichable patterns and
algebroids in \S\ref{sec:ex}. This includes the previously studied
cases of enriched \icats{} \cite{enriched}, enriched \iopds{}
\cite{ChuHaugseng} and enriched $\infty$-properads \cite{iprpd}, but
we also look at some new examples:
\begin{itemize}
\item In \S\ref{subsec:cartprod} we show that cartesian products of
  enrichable patterns are again enrichable in a natural way, and use
  this to define enriched $n$-fold \icats{}.
\item In \S\ref{subsec:wreath} we define \emph{wreath products} of
  enrichable patterns, which give a (non-iterative) definition of
  enriched $(\infty,n)$-categories.
\item In \S\ref{subsec:modular} we show that the category of
  undirected graphs used by Hackney, Robertson, and Yau
  \cite{HRYModular} to define modular \iopds{} is enrichable, which
  allows us to define enriched modular \iopds{}. (Similar arguments
  also apply to other categories of trees and graphs defined by the
  same authors, which gives enriched versions of for instance cyclic
  \iopds{} and wheeled $\infty$-properads.)
\end{itemize}
We note that in this paper we consider all of these examples as merely
\emph{algebraic} structures. That is, we do not invert some class of
``fully faithful and essentially surjective'' morphisms, for instance
by restricting to an appropriate full subcategory of ``complete''
objects; indeed, we expect such completeness results can only be proved on a case
by case basis.

\subsubsection*{Acknowledgments}
The first author was supported by Labex CEMPI (ANR-11-LABX-0007-01) in
the early stages of work on this project.  Parts of the paper were
written while the second author was in residence at the Matematical
Sciences Research Institute in Berkeley, California, during (the
pre-COVID part of) the Spring 2020 semester, and it is thereby partially
supported by the National Science Foundation under grant DMS-1440140.

We thank Shaul Barkan and Jan Steinebrunner for helpful comments on a
draft of the paper, in particular in connection with \S\ref{sec:lke}.

\section{General theory}\label{sec:general}

\subsection{Review of algebraic patterns and Segal objects}
In this section we will briefly review the notions of algebraic
patterns and their Segal objects, as introduced in \cite{patterns}, as
well as the special case of \emph{cartesian patterns} introduced in
\cite{freealg}.

\begin{defn}
  An \emph{algebraic pattern} consists of an \icat{} $\mathcal{O}$
  equipped with a factorization system, whereby each map factors as an
  \emph{inert} map followed by an \emph{active} map, together with a
  full subcategory $\mathcal{O}^{\el} \subseteq \mathcal{O}^{\xint}$
  consisting of \emph{elementary} objects, where $\mathcal{O}^{\xint}$
  is the subcategory of $\mathcal{O}$ containing only the inert
  maps. A morphism of algebraic patterns from $\mathcal{O}$ to
  $\mathcal{P}$ is a functor $f \colon \mathcal{O} \to \mathcal{P}$
  that preserves inert and active morphisms and elementary objects.
\end{defn}

\begin{notation}
  We write $\AlgPatt$ for the \icat{} of algebraic patterns and
  morphisms thereof; see \cite[\S 5]{patterns} for more
  details.
\end{notation}

\begin{notation}
  If $\mathcal{O}$ is an algebraic pattern we will indicate an inert
  map between objects $O,O'$ of $\mathcal{O}$ as $O \intto O'$ and an
  active map as $O \actto O'$. These symbols are not meant to suggest
  any intuition about the nature of inert and active maps.
\end{notation}

\begin{example}
  Let $\xF_{*}$ denote a skeleton of the category of finite pointed
  sets with objects $\angled{n} := (\{0,1,\ldots,n\},0)$. This has a
  factorization system where an inert map is a morphism that is an
  isomorphism away from the base point, and an active map is a
  morphism that doesn't send anything except the base point to the
  base point. We can give $\xF_{*}$ the two following pattern
  structures using this factorization system, both of which will play
  an important role in this paper:
  \begin{itemize}
  \item We write $\xF_{*}^{\flat}$ for the
    algebraic pattern where
    $\angled{1}$ is the only elementary object.
  \item We write $\xF_{*}^{\natural}$ for the algebraic pattern where
    $\angled{0}$ and $\angled{1}$ are both elementary objects.
  \end{itemize}
\end{example}

\begin{notation}
  If $\mathcal{O}$ is an algebraic pattern and $O$ is an object of
  $\mathcal{O}$, we write
  \[\mathcal{O}^{\el}_{O/} :=
    \mathcal{O}^{\el}\times_{\mathcal{O}^{\xint}}
    \mathcal{O}^{\xint}_{O/}\]
  for the \icat{} of inert maps from $O$ to elementary objects, and
  inert maps between them.
\end{notation}

\begin{defn}
  A \emph{cartesian pattern} is an algebraic pattern $\mathcal{O}$
  equipped with a morphism of algebraic patterns
  $|\blank|_{\mathcal{O}} \colon \mathcal{O} \to \xF_{*}^{\flat}$ such that
  for every object $O \in \mathcal{O}$ the induced map
  \[ \mathcal{O}^{\el}_{O/} \to
    \xF_{*,|O|_{\mathcal{O}}/}^{\flat,\el} \] is an equivalence. (When
  the pattern $\mathcal{O}$ is clear from context, we will usually
  abbreviate $|O|_{\mathcal{O}}$ to $|O|$.) A \emph{morphism of
    cartesian patterns} is a morphism of algebraic patterns over
  $\xF_{*}^{\flat}$.
\end{defn}

\begin{notation}
  We write $\rho_{i} \colon \angled{n} \to \angled{1}$, $i = 1,\ldots,n$, for the inert
  map given by
  \[ \rho_{i}(j) =
    \begin{cases}
      0, & j \neq i, \\
      1, & j = i.
    \end{cases}
  \]
  Then $\xF^{\flat, \el}_{*,\angled{n}/}$ is equivalent to the
  discrete set $\{\rho_{1},\ldots,\rho_{n}\}$. On the other hand,
  $\xF^{\natural, \el}_{*,\angled{n}/}$ is the category
  $\{\rho_{1},\ldots,\rho_{n}\}^{\triangleright}$, with the additional
  terminal object corresponding to the unique (inert) map $\angled{n}
  \to \angled{0}$.
\end{notation}

\begin{notation}
  If $\mathcal{O}$ is a cartesian pattern and $O$ is an object of
  $\mathcal{O}$ such that $|O| \cong \angled{n}$, then the \icat{}
  $\mathcal{O}^{\el}_{O/}$ is equivalent to a discrete set consisting
  of $n$ inert morphisms with source $O$, with an essentially unique
  such morphism lying over
  $\rho_{i} \colon \angled{n} \to \angled{1}$ for $i = 1,\ldots,n$.
  We will often denote these inert morphisms by $\rho_{i}^{O}\colon O \to O_{i}$.
\end{notation}

\begin{defn}
  Let $\mathcal{O}$ be an algebraic pattern. A Segal
  $\mathcal{O}$-object in an \icat{} $\mathcal{C}$ is a functor $F \colon
  \mathcal{O} \to \mathcal{C}$ such that for every $O \in \mathcal{O}$
  the induced functor
  \[ (\mathcal{O}^{\el}_{X/})^{\triangleleft} \to \mathcal{O} \to
    \mathcal{C} \]
  is a limit diagram.  We write $\Seg_{\mathcal{O}}(\mathcal{C})$ for
  the full subcategory of $\Fun(\mathcal{O},\mathcal{C})$ spanned by
  the Segal $\mathcal{O}$-objects. We refer to Segal
  $\mathcal{O}$-objects in the \icat{} $\mathcal{S}$ of spaces as
  \emph{Segal $\mathcal{O}$-spaces}.
\end{defn}

\begin{defn}
  If $\mathcal{O}$ is a cartesian pattern, we
  will refer to an $\mathcal{O}$-Segal object as an
  \emph{$\mathcal{O}$-monoid}. We write
  \[\Mon_{\mathcal{O}}(\mathcal{C}) :=
    \Seg_{\mathcal{O}}(\mathcal{C})\] for the \icat{} of
  $\mathcal{O}$-monoids in $\mathcal{C}$.
\end{defn}

\begin{observation}
  Note that the limit condition for an $\mathcal{O}$-monoid $M$ is a
  finite product: for $O \in \mathcal{O}$ over $\angled{n}$ in
  $\xF_{*}$, we want the natural map
  \[ M(O) \to \prod_{i} M(O_{i}),\]
  induced by the inert maps $\rho_{i}^{O} \colon O \to O_{i}$, to be
  an equivalence.
\end{observation}

\begin{defn}
  A morphism of algebraic patterns $f \colon \mathcal{O} \to
  \mathcal{P}$ is a \emph{Segal morphism} if the functor $f^{*}\colon
  \Fun(\mathcal{P}, \mathcal{S}) \to \Fun(\mathcal{O}, \mathcal{S})$
  given by composition with $f$ restricts to a functor
  \[ f^{*} \colon \Seg_{\mathcal{P}}(\mathcal{S}) \to
    \Seg_{\mathcal{O}}(\mathcal{S}), \]
  or equivalently if the analogue holds for Segal objects in every
  \icat{} (\cite[Lemma 4.5]{patterns}). This is in particular the case if for every $O \in \mathcal{O}$, the
  induced functor
  \[ \mathcal{O}^{\el}_{O/} \to \mathcal{P}^{\el}_{f(O)/} \] is
  coinitial, in which case we say that $f$ is a \emph{strong Segal
    morphism}. We will further say that $f$ is an \emph{iso-Segal
    morphism} if these functors are all equivalences.
\end{defn}

\begin{observation}
  Suppose $f \colon \mathcal{O} \to \mathcal{P}$ is a morphism of
  cartesian patterns. Then for any $O \in \mathcal{O}$ we have a
  commutative triangle
  \[
    \begin{tikzcd}
      \mathcal{O}^{\el}_{O/} \arrow{rr} \arrow{dr}[swap]{\sim} & &
      \mathcal{P}^{\el}_{f(O)/} \arrow{dl}{\sim} \\
       & \xF^{\flat}_{*,|O|/},
    \end{tikzcd}
  \]
  so the horizontal morphism is also an equivalence. In particular,
  $f$ is an iso-Segal morphism.
\end{observation}

\begin{defn}\label{def SegalOfib}
  Let $\mathcal{O}$ be an algebraic pattern. A \emph{Segal
    $\mathcal{O}$-fibration} is a cocartesian fibration $\pi \colon \mathcal{E}
  \to \mathcal{O}$ whose associated functor $\mathcal{O} \to \CatI$ is
  a Segal $\mathcal{O}$-object. By \cite[Lemma 9.4]{patterns} the
  \icat{} $\mathcal{E}$ inherits a canonical pattern structure, where
  \begin{itemize}
  \item the inert morphisms are the cocartesian morphisms that lie over inert
    morphisms in $\mathcal{O}$,
  \item the active morphisms are all the morphisms the lie over active
    morphisms in $\mathcal{O}$,
  \item the elementary objects are the objects that lie over
    elementary objects in $\mathcal{O}$.
  \end{itemize}
  With this structure, the projection
  $\pi \colon \mathcal{E} \to \mathcal{O}$ is a strong Segal
  morphism --- indeed, the functor
  \[ \mathcal{E}^{\el}_{X/} \to \mathcal{O}^{\el}_{\pi(X)/} \]
  is an equivalence, so $\pi$ is even an iso-Segal morphism.
\end{defn}

\begin{defn}
  If $\mathcal{O}$ is a cartesian pattern, we refer to a
  Segal $\mathcal{O}$-fibration $\xE\to \mathcal{O}$ as an
  \emph{$\mathcal{O}$-monoidal \icat{}}.
\end{defn}

\begin{defn}\label{def alg}
  Let $\mathcal{O}$ be a cartesian pattern and
  $\mathcal{C}^{\otimes} \to \mathcal{O}$ an $\mathcal{O}$-monoidal
  \icat{}. An \emph{$\mathcal{O}$-algebra} in $\mathcal{C}^{\otimes}$
  is a section
  \[
    \begin{tikzcd}
      \mathcal{O} \arrow[equals]{dr} \arrow{rr}{A} & &
      \mathcal{C}^{\otimes} \arrow{dl} \\
       & \mathcal{O},
    \end{tikzcd}
  \]
  where $A$ takes inert morphisms in $\mathcal{O}$ to cocartesian
  morphisms in $\mathcal{C}^{\otimes}$.
  More generally, for a morphism $f \colon \mathcal{P} \to
  \mathcal{O}$ of cartesian patterns we define a \emph{$\mathcal{P}$-algebra 
  in $\mathcal{C}^{\otimes}$ over $\mathcal{O}$} to be a commutative triangle
  \[
  \begin{tikzcd}
  \mathcal{P} \arrow{rr}{A} \arrow{dr}[swap]{f} & &
  \mathcal{C}^{\otimes} \arrow{dl} \\
  & \mathcal{O}
  \end{tikzcd}
  \]
  such that $A$ takes inert morphisms in $\mathcal{P}$ to cocartesian
  morphisms in $\mathcal{C}^{\otimes}$, which are inert by
  definition. (This condition is equivalent to $A$ being a morphism of
  algebraic patterns over $\mathcal{O}$, as is clear from the
  definition of the pattern structure on $\mathcal{C}^{\otimes}$.) We
  write $\Alg_\mathcal{O}(\mathcal{C}^\otimes)$ and
  $\Alg_{\mathcal{P}/\mathcal{O}}(\mathcal{C}^\otimes)$ for the full subcategories of
  $\xFun_{/\mathcal{O}}(\mathcal{O}, \mathcal{C}^\otimes)$ and
  $\xFun_{/\mathcal{O}}(\mathcal{P}, \mathcal{C}^\otimes)$ spanned by $\mathcal{O}$-algebras and
  $\mathcal{P}$-algebras over $\mathcal{O}$, respectively.
\end{defn}

\begin{observation}\label{rem pullback_algebras}
  The pullback of an $\mathcal{O}$-monoidal \icat{} $\mathcal{C}^\otimes$ along a
  morphism $f \colon \mathcal{P} \to \mathcal{O}$ of cartesian
  patterns yields a $\mathcal{P}$-monoidal \icat{} $f^*\mathcal{C}^\otimes$. Hence,
  base change along $f$ induces a natural equivalence
  \[\Alg_{\mathcal{P}/\mathcal{O}}(\mathcal{C}^\otimes)\simeq \Alg_{\mathcal{P}}(f^*\mathcal{C}^\otimes).\]
\end{observation}

\subsection{Enrichable patterns and algebroids}
We are now ready to introduce the data we need to define enriched
structures, which we call an \emph{enrichable pattern}.

\begin{notation}
  Given a morphism of algebraic patterns $\mathcal{O}\to
  \xF_{*}^{\natural}$, we write $\mathcal{O}^{\flat}$ for the pullback
  of patterns $\mathcal{O}\times_{\xF_{*}^{\natural}}
  \xF_{*}^{\flat}$; this exists by \cite[Corollary 5.5]{patterns} and
  is simply given by the \icat{} $\mathcal{O}$ equipped with the same
  factorization system as before, but with $\mathcal{O}^{\flat,\el}$
  consisting only of the elementary objects in $\mathcal{O}$ that lie
  over $\angled{1}$.
\end{notation}

\begin{defn}
  An \emph{enrichable pattern} is an algebraic pattern $\mathcal{O}$
  equipped with a morphism of patterns $|\blank|_{\mathcal{O}} \colon \mathcal{O}
  \to \xF_{*}^{\natural}$, such that $\mathcal{O}^{\flat}$ is a
  cartesian pattern.
\end{defn}

\begin{remark}
  Instead of working over $\xFn$ we could also consider the pattern
  $\Span(\xF)^{\natural}$, consisting of the (2,1)-category $\Span(\xF)$ whose objects are finite sets and whose morphisms from $X$ to $Y$ are spans
  \[
    \begin{tikzcd}
      {} & S \arrow{dl} \arrow{dr} \\
      X & & Y,
    \end{tikzcd}
  \]
  with composition given by pullbacks, together with the factorization
  system where the inert and active maps are those whose forwards and
  backwards components are isomorphisms, respectively, and with the
  one-element set as the only elementary object. This gives a somewhat
  more general notion of enrichable patterns, since we can regard
  $\xF_{*}$ as the subcategory of $\Span(\xF)$ containing the spans
  where the backwards maps are injective. Since the examples we are
  interested in all live over $\xFn$, however, we will stick with the
  version that is directly supported by our previous work in
  \cite{freealg}.
\end{remark}

\begin{remark}
  We postpone a discussion of specific examples until \S\ref{sec:ex},
  but we note here that the \emph{hypermoment categories} of Berger
  \cite{BergerMoment} are a special case of our enrichable patterns,
  which includes all the examples we discuss. Any \emph{perfect
    operator category} in the sense of Barwick \cite{bar} also gives
  rise to an enrichable pattern; see \cref{var:opcat}.
\end{remark}

\begin{notation}
  Let $\mathcal{O}$ be an enrichable pattern. We will often denote the
  specified map to $\xFs$ as simply $|\blank|$ when it is clear from
  context that this refers to the pattern $\mathcal{O}$. Moreover, we
  will refer to an $\mathcal{O}^{\flat}$-monoid as an
  \emph{$\mathcal{O}$-monoid}, and write
  $\Mon_{\mathcal{O}}(\mathcal{C})$ for the \icat{}
  $\Mon_{\mathcal{O}^{\flat}}(\mathcal{C})$ of these in
  $\mathcal{C}$. In the same way we will speak about
  $\mathcal{O}$-monoidal \icats{}, $\mathcal{O}$-algebras, and so
  forth.
\end{notation}

\begin{lemma}\label{lem enrpatt}
  If $\mathcal{O}$ is an enrichable pattern and $\pi \colon
  \mathcal{E} \to \mathcal{O}$ is a Segal $\mathcal{O}$-fibration,
  then $\mathcal{E}$ is also an enrichable pattern via the composite
  $\mathcal{E} \xto{\pi} \mathcal{O} \xto{|\blank|} \xF_{*}^{\natural}$.
\end{lemma}
\begin{proof}
  For $X \in \mathcal{E}$, consider the commutative diagram
  \[
    \begin{tikzcd}
      \xE^{\flat,\el}_{X/} \ar[r] \ar[d] &
      \mathcal{O}^{\flat,\el}_{\pi(X)/} \ar[r, "\sim"] \ar[d] & 
      \xF^{\flat,\el}_{*,|\pi(X)|/} \ar[d]\\
      \xE^\el_{X/} \ar[r, "\sim"] & \mathcal{O}^\el_{\pi(X)/} \ar[r] & \xF^{\natural,\el}_{*,|\pi(X)|/},
    \end{tikzcd}
  \]
  where both squares are cartesian. Here the bottom left horizontal
  morphism is an equivalence since $\pi$ is a Segal
  $\mathcal{O}$-fibration, while the top right horizontal morphism is
  one since $\mathcal{O}^{\flat}$ is cartesian. Hence the composite in
  the top row is an equivalence, which says precisely that
  $\mathcal{E}^{\flat}$ is cartesian, as required.
\end{proof}

\begin{notation}\label{no O_0}
  Let $\mathcal{O}$ be an enrichable pattern. We write
  $\mathcal{O}_{0}$ for the full subcategory of $\mathcal{O}$ spanned
  by objects that lie over $\angled{0} \in \xF_{*}$, or equivalently
  the fibre product $\mathcal{O} \times_{\xF_{*}} \{\angled{0}\}$. We
  can regard this as a fibre product
  $\mathcal{O} \times_{\xF_{*}^{\natural}} \{\angled{0}\}$ in the
  \icat{} of algebraic patterns (where $\angled{0}$ is elementary in
  $\{\angled{0}\}$), so that $\mathcal{O}_{0}$ inherits a pattern
  structure.  Indeed, since the functor from $\mathcal{O}$ to
  $\xF_{*}$ is compatible with the inert--active factorization system,
  it follows that the inert--active factorization of a morphism in
  $\mathcal{O}_{0}$ again lies in $\mathcal{O}_{0}$, so we get a
  restricted factorization system on $\mathcal{O}_{0}$, and we take
  the elementary objects to be the elementary objects of $\mathcal{O}$
  that lie over $\angled{0}$. We denote the inclusion
  $\mathcal{O}_{0} \hookrightarrow \mathcal{O}$ by
  $i_{0}^{\mathcal{O}}$ (or just $i_{0}$ if $\mathcal{O}$ is clear
  from the context); this is a morphism of algebraic patterns. 
\end{notation}

\begin{observation}\label{rem O_0 int}
  Let $\mathcal{O}$ be an enrichable pattern, and suppose $O$ is an
  object in $\mathcal{O}_{0}$. If we have an inert morphism $O \intto
  O'$ in $\mathcal{O}$, then the object $O'$ must also lie in
  $\mathcal{O}_{0}$ (since there are no inert morphisms $\angled{0}
  \to \angled{n}$ in $\xF_{*}$ unless $n = 0$). This means that we
  have an equivalence of \icats{}
  \[ \mathcal{O}^{\xint}_{0,O/} \isoto \mathcal{O}^{\xint}_{O/},\]
  which restricts to an equivalence
  \[ \mathcal{O}^{\el}_{0,O/} \isoto \mathcal{O}^{\el}_{O/}.\]
  This shows that $i_{0}^{\mathcal{O}}$ is an iso-Segal morphism.
\end{observation}

\begin{observation}\label{rem O_0 act}
  Suppose $\mathcal{O}$ is an
  enrichable pattern. If $O$ is an object of $\mathcal{O}_{0}$ and
  $O' \actto O$ is an active morphism, then $O'$ must also lie in
  $\mathcal{O}_{0}$ (since there are no active morphisms $\angled{n}
  \to \angled{0}$ in $\xF_{*}$ unless $n = 0$). We thus have an
  equivalence of \icats{}
  \[\mathcal{O}^{\act}_{0,/O} \isoto
  \mathcal{O}_{/O}^{\act}\] for all $O \in \mathcal{O}_{0}$, which
  means precisely that $i_{0}^{\mathcal{O}}$ has \emph{unique lifting
    of active morphisms} in the sense of \cite[\S 6]{patterns}.
\end{observation}
Applying \cite[Proposition 6.3]{patterns}, we then get the following:
\begin{lemma}\label{lem i_0 Seg}
  Let $\mathcal{C}$ be a complete \icat{}. If $\mathcal{O}$ is an
  enrichable pattern, then right Kan extension
  along the inclusion $i^{\mathcal{O}}_{0} \colon \mathcal{O}_{0} \hookrightarrow \mathcal{O}$
  takes Segal $\mathcal{O}_{0}$-objects in $\mathcal{C}$ to Segal $\mathcal{O}$-objects. \qed
\end{lemma}

\begin{observation}\label{enrpattmonSeg}
  If $\mathcal{O}$ is an enrichable pattern, then it follows from
  \cite[Example 6.6]{patterns} that any
  $\mathcal{O}$-monoid in an \icat{} $\mathcal{C}$ is also a Segal
  $\mathcal{O}$-object. We thus have a fully faithful inclusion
  \[ \Mon_{\mathcal{O}}(\mathcal{C}) \hookrightarrow
    \Seg_{\mathcal{O}}(\mathcal{C}) \]
  of full subcategories of $\Fun(\mathcal{O}, \mathcal{C})$.
  In particular, any $\mathcal{O}$-monoidal \icat{} is a Segal
  $\mathcal{O}$-fibration, and so is an enrichable pattern by
  \cref{lem enrpatt}. We can give a more precise description of which Segal $\mathcal{O}$-objects are $\mathcal{O}$-monoids:
\end{observation}

\begin{lemma}\label{lem:monoidrestrpt}
  Let $\mathcal{O}$ be an enrichable pattern and $\mathcal{C}$ a
  complete \icat{}. A Segal
  $\mathcal{O}$-object in $\mathcal{C}$ is an $\mathcal{O}$-monoid
  \IFF{} its restriction to $\mathcal{O}_{0}$ is a constant functor
  with value the terminal object of $\mathcal{C}$.
\end{lemma}
\begin{proof}
  For any $O \in \mathcal{O}_{0}$ we must have
  $\mathcal{O}^{\flat,\el}_{O/} \simeq \emptyset$, since this lives
  over $\xF_{*,\angled{0}/}^{\flat,\el}$, which is empty (since there
  are no inert maps $\angled{0} \to \angled{1}$ in
  $\xF_{*}$). Therefore, if 
  $M \colon \mathcal{O} \to \mathcal{C}$ is an
  $\mathcal{O}$-monoid, the object $M(O)$ is the terminal object of
  $\mathcal{C}$, so that $M|_{\mathcal{O}_{0}}$ is indeed constant
  with value the terminal object.

  Conversely, suppose $F \colon \mathcal{O} \to \mathcal{C}$ is a
  Segal $\mathcal{O}$-object such that $F|_{\mathcal{O}_{0}}$ is
  constant with value the terminal object. We claim that this means
  that the induced functor $F' \colon \mathcal{O}^{\el}_{O/} \to \mathcal{C}$ is
  right Kan extended along the inclusion $j \colon
  \mathcal{O}^{\flat,\el}_{O/} \to \mathcal{O}^{\el}_{O/}$ for every 
  $O \in \mathcal{O}$. Indeed, for $\alpha \colon O \intto E$ in
  $\mathcal{O}^{\el}_{O/}$ we have that
  $(\mathcal{O}^{\flat,\el}_{O/})_{\alpha/}$ is $\{\alpha\}$ if $E$
  lies over $\angled{1}$, and empty if $E$ lies over $\angled{0}$. For
  a functor $\phi \colon \mathcal{O}^{\flat,\el}_{O/} \to \mathcal{C}$
  we thus get
  \[ j_{*}\phi(\alpha) \simeq
    \begin{cases}
      \phi(\alpha), & \alpha \in \mathcal{O}^{\flat,\el}_{O/} \\
      *, & \alpha \notin \mathcal{O}^{\flat,\el}_{O/}.
    \end{cases}
  \]
  The map $F' \to j_{*}j^{*}F'$ is therefore an equivalence, hence so
  is the induced map on limits. The Segal map $F(O) \to \prod_{i}
  F(O_{i})$ decomposes as
  \[ F(O) \isoto \lim_{E \in \mathcal{O}^{\el}_{O/}} F(E) \isoto \prod_{i}
    F(O_{i})\]
  where the first map is an equivalence since $F$ is a Segal
  $\mathcal{O}$-object and the second is the map on limits we just saw
  was an equivalence. Thus $F$ is indeed an $\mathcal{O}$-monoid.
\end{proof}

We are now ready to define the algebraic structures associated to an
enrichable pattern $\mathcal{O}$ that we think of as \emph{enriched}
versions of Segal $\mathcal{O}$-spaces.

\begin{defn}\label{def OX}
  Let $\mathcal{O}$ be an enrichable pattern. Given a functor
  $\bX \colon \mathcal{O}_{0} \to \mathcal{S}$, we write
  $\ObX \to \mathcal{O}$ for the left fibration corresponding to the
  functor $i_{0,*}\bX \colon \mathcal{O} \to \mathcal{S}$ obtained as
  the right Kan extension of $\bX$ along $i_{0}$. If $\bX$ is a Segal
  $\mathcal{O}_{0}$-space, then $i_{0,*}\bX$ is a Segal
  $\mathcal{O}$-space by \cref{lem i_0 Seg}. Thus
  $\ObX \to \mathcal{O}$ is a Segal $\mathcal{O}$-fibration; we write
  $\ObX^{\natural}$ for the pattern structure lifted from
  $\mathcal{O}$, which is again an enrichable pattern by \cref{lem
    enrpatt}. Combining the right Kan extension functor $i_{0,*}$ with
  the straightening equivalence, we obtain in particular a functor
  \[ \mathcal{O}_{(\blank)} \colon \Seg_{\mathcal{O}_{0}}(\mathcal{S})
    \to \AlgPatt_{/\mathcal{O}}. \]
\end{defn}

\begin{observation}\label{rem O_0X}
  Let $\bX$ be a functor $\mathcal{O}_{0} \to \mathcal{S}$. If we pull
  back the left fibration $\mathcal{O}_\bX\to \mathcal{O}$ to
  $\mathcal{O}_{0}$, we get the left fibration
  $\mathcal{O}_{0,\bX}\to \mathcal{O}_0$ that corresponds to the
  functor $i_0^*i_{0,*}(\bX)\colon \mathcal{O}_0\to
  \mathcal{S}$. Since $i_0$ is fully faithful, the counit
  $i_0^*i_{0,*}\to \id$ is a natural equivalence. Hence,
  $\mathcal{O}_{0,\bX}\to \mathcal{O}_0$ is the left fibration for the
  functor $\bX$.
\end{observation}

\begin{defn}
  Let $\mathcal{O}$ be an enrichable pattern, and suppose
  $\mathcal{C}$ is an $\mathcal{O}$-monoidal \icat{}.
  For $\bX \in \Seg_{\mathcal{O}_{0}}(\mathcal{S})$, we define an 
  \emph{$\mathcal{O}$-algebroid} in $\mathcal{C}$ with
  underlying Segal $\mathcal{O}_{0}$-space $\bX$ to be an $\ObX$-algebra
  in $\mathcal{C}$. In other words, it is
  a commutative triangle
  \[
    \begin{tikzcd}
      \mathcal{O}_{\bX} \arrow{rr} \arrow{dr} & & \mathcal{C}^{\otimes}
      \arrow{dl} \\
      & \mathcal{O}
    \end{tikzcd}
  \]
  where the horizontal functor takes inert morphisms in
  $\mathcal{O}_{\bX}$ to cocartesian morphisms in
  $\mathcal{C}^{\otimes}$. We write
  \[ \Algd_{\mathcal{O}}(\mathcal{C}) \to
  \Seg_{\mathcal{O}_{0}}(\mathcal{S})\] for the cartesian fibration
  for the functor $\Seg_{\mathcal{O}_{0}}(\mathcal{S})^{\op} \to
  \CatI$ that takes $\bX \in \Seg_{\mathcal{O}_{0}}(\mathcal{S})$ to
  $\Alg_{\mathcal{O}_{\bX}/\mathcal{O}}(\mathcal{C})$; here we call
  $\Algd_{\mathcal{O}}(\mathcal{C})$ the \emph{\icat{} of
  $\mathcal{O}$-algebroids in $\mathcal{C}$}. 
\end{defn}

\begin{defn}
  More generally, given enrichable patterns $\mathcal{O}$ and
  $\mathcal{P}$, a morphism $f\colon \mathcal{P}\to \mathcal{O}$ of
  algebraic patterns over $\xF_{*}^{\natural}$, and an
  $\mathcal{O}$-monoidal \icat{} $\mathcal{C}$, we define a
  \emph{$\mathcal{P}$-algebroid over $\mathcal{O}$} in $\mathcal{C}$
  with underlying Segal $\mathcal{P}_{0}$-space $\bX$ to be a
  $\mathcal{P}_{\bX}$-algebra in $\mathcal{C}$ over
  $\mathcal{O}$. This is equivalently a commutative square of \icats{}
  \[
    \begin{tikzcd}
      \mathcal{P}_\bX \arrow[r] \arrow[d] & \mathcal{C}^\otimes \arrow[d]\\
      \mathcal{P} \arrow[r, "f"] & \mathcal{O}
    \end{tikzcd}
  \]
  such that the upper horizontal map preserves inert morphisms. We
  then take
  \[ \Algd_{\mathcal{P}/\mathcal{O}}(\mathcal{C})\to
  \Seg_{\mathcal{P}_0}(\xS)\] to be the cartesian fibration
  corresponding to the functor
  $\Seg_{\mathcal{P}_0}(\xS)^\op\to \CatI$ taking $\bX$ to
  $\Alg_{\mathcal{P}_\bX/\mathcal{O}}(\mathcal{C}^\otimes)$;  we
  call $\Algd_{\mathcal{P}/\mathcal{O}}(\mathcal{C})$ the
  \emph{\icat{} of $\mathcal{P}$-algebroids over $\mathcal{O}$ in
    $\mathcal{C}$}. We will denote
  $\Algd_{\mathcal{P}/\mathcal{O}}(\mathcal{C})$ as 
  just $\Algd_{\mathcal{P}}(\mathcal{C})$
  when the slicing over $\mathcal{O}$ is clear from the context.
\end{defn}

\subsection{Enrichment in spaces}\label{subsec:enrspc}
In this section we fix an enrichable pattern $\mathcal{O}$ and perform a sanity check: 
We will show that an $\mathcal{O}$-algebroid in spaces is the same thing as a Segal
$\mathcal{O}$-space. More precisely, we will show:
\begin{thm}\label{thm:algdspc}
  There is a canonical equivalence between $\Algd_{\mathcal{O}}(\mathcal{S})$ and
$\Seg_{\mathcal{O}}(\mathcal{S})$, compatible with the projections to $\Seg_{\mathcal{O}_{0}}(\mathcal{S})$.
\end{thm}

We start by noting an alternative description of algebroids in any
cartesian monoidal \icat{}:

\begin{definition}
  Let $\mathcal{C}$ be an \icat{} with finite products. We write
  $\mathcal{M}_{\mathcal{O}}(\mathcal{C}) \to 
  \Seg_{\mathcal{O}_{0}}(\mathcal{S})$ for the cartesian fibration
  corresponding to the functor
  \[ \bX \mapsto \Mon_{\ObX}(\mathcal{C}),\]
  using the functoriality of $\mathcal{O}_{(\blank)}$ noted in
  \cref{def OX}.
\end{definition}

\begin{lemma}\label{lem AlgM}
  Let $\mathcal{C}$ be an \icat{} with finite products. Then there is
  a natural equivalence
  \[ \Algd_{\mathcal{O}/\xF_{*}}(\mathcal{C}^{\times}) \simeq
    \mathcal{M}_{\mathcal{O}}(\mathcal{C})\]
  over $\Seg_{\mathcal{O}_{0}}(\mathcal{S})$.
\end{lemma}
\begin{proof}
  From \cite[Proposition 5.1]{freealg} we have for any $\bX \in
  \Seg_{\mathcal{O}_{0}}(\mathcal{S})$ an equivalence
  \[ \Alg_{\ObX/\xF_{*}}(\mathcal{C}^{\times}) \simeq 
    \Mon_{\ObX}(\mathcal{C}),\]
  and unwinding the definition of this it is clear that it is natural
  in $\bX$.
\end{proof}

Next, we check that $\Seg_{\mathcal{O}}(\mathcal{C})$ is a cartesian
fibration over $\Seg_{\mathcal{O}_{0}}(\mathcal{C})$:

\begin{propn}\label{propn:SegCartFib}
  Let $\mathcal{C}$ be a complete \icat{}. Restriction along the inclusion $i_{0} \colon \mathcal{O}_{0}
  \hookrightarrow \mathcal{O}$ gives a cartesian fibration
  \[ i_0^{*} \colon \Seg_{\mathcal{O}}(\mathcal{C}) \to
    \Seg_{\mathcal{O}_0}(\mathcal{C}).\] For
  $F \in \Seg_{\mathcal{O}}(\mathcal{C})$, the cartesian lift of a
  morphism $\phi \colon \bX \to i_0^{*}F$ in
  $\Seg_{\mathcal{O}_0}(\mathcal{C})$ is given by the pullback of
  $i_{0, *}\bX \to i_{0, *}i_0^{*}F$ along the unit map
  $F \to i_{0, *}i_0^{*}F$.
\end{propn}

\begin{proof}
  We know from \cref{lem i_0 Seg} that right Kan extension along
  $i_{0}$ gives a right adjoint to $i_{0}^{*}$ on Segal objects.  To
  see that $i_{0}^{*}$ is a cartesian fibration we can therefore apply
  the criterion of \cite[Corollary 4.52]{nmorita}, which requires us
  to check that for every $F \in
   \Seg_{\mathcal{O}}(\mathcal{C})$ and every morphism $\bX \to i_0^{*}F$
   in $\Seg_{\mathcal{O}_0}(\mathcal{C})$, the object $\bar{F}$ defined by the pullback square 
   \[
   \begin{tikzcd}
   \bar{F} \arrow{r} \arrow{d} & F \arrow{d} \\
   i_{0, *}\bX \arrow{r} & i_{0, *}i_0^{*}F,
   \end{tikzcd}
   \]
   induces a equivalence
   $i_0^{*}\bar{F} \to i_0^{*}i_{0, *}\bX \to \bX$. 

  Since $i_{0}$ is fully faithful, the counit $i_0^{*}i_{0, *} \to
  \id$ is a natural equivalence. As $i_0^{*}$ preserves limits we then
  get a pullback square
  \[
  \begin{tikzcd}
  i_0^{*}\bar{F} \arrow{r} \arrow{d} & i_0^{*}F \arrow{d}{\id} \\
  \bX \arrow{r} & i_0^{*}F,
  \end{tikzcd}
  \]
  so that the left vertical map is indeed an equivalence. Moreover, by
  \cite[Proposition 4.51]{nmorita} the morphism $\bar{F}\to F$ is here
  the cartesian lift of $\bX\to i_0^*F$ at $F$, as required.
\end{proof}

\begin{observation}\label{obs: seg over rke ftr}
  In the situation above, from the adjunction
  \[ i_{0}^{*} \colon \Seg_{\mathcal{O}}(\mathcal{C})
    \rightleftarrows \Seg_{\mathcal{O}_{0}}(\mathcal{C}) \cocolon i_{0,*}\]
  we obtain an equivalence between the pullback of
  \[ \Ar(\Seg_{\mathcal{O}}(\mathcal{C})) \xto{\ev_{1}}
    \Seg_{\mathcal{O}}(\mathcal{C}) \xfrom{i_{0,*}}
    \Seg_{\mathcal{O}_{0}}(\mathcal{C})\]
  and that of
  \[ \Seg_{\mathcal{O}}(\mathcal{C}) \xto{i_{0}^{*}}
    \Seg_{\mathcal{O}_{0}}(\mathcal{C}) \xfrom{\ev_{0}}
    \Ar(\Seg_{\mathcal{O}_{0}}(\mathcal{C})),\]
  given by taking a pair $(F, F \to i_{0,*}\bX)$ to the adjoint
  morphism $(F, i_{0}^{*}F \to \bX)$, and vice versa.
  Moreover, the projection from the second pullback to 
  $\Seg_{\mathcal{O}}(\mathcal{C})$ has a fully faithful left adjoint,
  which takes $F$ to the pair $(F, \id_{i_{0}^{*}F})$. We thus obtain
  an equivalence between $\Seg_{\mathcal{O}}(\mathcal{C})$ and the
  full subcategory of the first pullback spanned by pairs $(F, F \to
  i_{0,*}\bX)$ where the adjoint morphism $i_{0}^{*}F \to
  i_{0}^{*}i_{0,*}\bX \simeq \bX$ is an equivalence.
  In particular,
  we get an identification of the fibre of $i_{0}^{*}$ at $\bX$ 
  with the full
  subcategory of
  $\Seg_{\mathcal{O}}(\mathcal{C})_{/i_{0, *}\bX}$ spanned by the
  morphisms $F \to i_{0, *}\bX$ such that the adjoint morphism
  $i_0^{*}F \to \bX$ is an equivalence.
\end{observation}

We now restrict to the case of Segal $\mathcal{O}$-spaces, where we
want to describe the fibre of $\Seg_{\mathcal{O}}(\mathcal{S})$ at
$\bX \in \Seg_{\mathcal{O}_{0}}(\mathcal{S})$. The starting point for
this is the following observation, which already appeared as
\cite[Proposition 3.2.5]{iopdprop} (but we include a proof for
completeness).

\begin{propn}\label{lem:LFibSegslice}
  Let $\mathcal{O}$ be an algebraic pattern, and suppose
  $\pi \colon \mathcal{F} \to \mathcal{O}$ is the left fibration
  corresponding to a Segal $\mathcal{O}$-space
  $F \colon \mathcal{O} \to \mathcal{S}$. Then left Kan extension
  along $\pi$ gives a functor
  \[ \pi_{!} \colon \Seg_{\mathcal{F}}(\mathcal{S}) \to
    \Seg_{\mathcal{O}}(\mathcal{S})\]
  where $\mathcal{F}$ has the lifted pattern structure as in
  \cref{def SegalOfib}, 
  and this induces an equivalence
  \[ \Seg_{\mathcal{F}}(\mathcal{S})  \isoto
    \Seg_{\mathcal{O}}(\mathcal{S})_{/F}.\]
\end{propn}
\begin{proof}
  Left Kan extension along $\pi$ induces an equivalence
  \[ \Fun(\mathcal{F}, \mathcal{S}) \isoto \Fun(\mathcal{O},
    \mathcal{S})_{/F}\] by \cite[Corollary 9.8]{freepres}.
  It therefore suffices to show that the
  full subcategories of Segal objects correspond under this
  equivalence. In other words, we must show that
  $X \colon \mathcal{F} \to \mathcal{S}$ is a Segal
  $\mathcal{F}$-space \IFF{} $\pi_{!}X$ is a Segal
  $\mathcal{O}$-space. To see this, we observe that we have a
  commutative diagram
  \[
    \begin{tikzcd}[column sep=small]
      \colim_{p \in F(P)} X(p) \arrow{d}{\sim} \arrow{r} & \colim_{p \in F(P)}
      \lim_{\alpha \in \mathcal{O}^{\el}_{P/}} X(\alpha_{!}p)
      \arrow{r}{\sim} & \lim_{E \in \mathcal{O}^{\el}_{P/}} \colim_{q \in
        F(E)} X(q) \arrow{d}{\sim} \\
      \pi_{!}X(P) \arrow{rr} & & \lim_{E \in \mathcal{O}^{\el}_{P/}} \pi_{!}X(E).
    \end{tikzcd}
  \]
  Here the vertical morphisms are equivalences since $\pi$ is a left
  fibration (so that the left Kan extension along $\pi$ is given by
  colimits over the fibres) and the top right horizontal morphism is
  an equivalence because for any functor $K \colon \mathcal{I} \to
  \mathcal{S}$, $\mathcal{I}$-limits distribute over $K$-colimits 
  in $\mathcal{S}$ by \cite[Corollary 6.15]{patterns}. Thus the
  bottom horizontal morphism is an equivalence \IFF{} the top left
  horizontal morphism is one. Moreover, the latter is an equivalence
  \IFF{} it is an equivalence on fibres over each point $p \in F(p)$,
  \ie{} \IFF{} the map $X(p) \to \lim_{\alpha \in
    \mathcal{O}^{\el}_{P/}} X(\alpha_{!}p)$ is an equivalence for all
  $p \in F(P)$. This condition holds for all $p \in \mathcal{F}$
  precisely when $X$ is a Segal $\mathcal{F}$-space, so we see that
  this is equivalent to $\pi_{!}X$ being a Segal $\mathcal{O}$-space.
\end{proof}

\begin{observation}\label{obs lke seg lfib nat}
  In the situation above, if $\phi \colon F \to G$ is a morphism in
  $\Seg_{\mathcal{O}}(\mathcal{S})$ corresponding to a functor $\Phi
  \colon \mathcal{F} \to \mathcal{G}$ between the corresponding left
  fibrations, then we have a commutative square
  \[
    \begin{tikzcd}
      \Seg_{\mathcal{G}}(\mathcal{S}) \arrow{r}{\Phi^{*}}
      \arrow{d}[swap]{\sim} & \Seg_{\mathcal{F}}(\mathcal{S})
      \arrow{d}{\sim}
      \\
      \Seg_{\mathcal{O}}(\mathcal{S})_{/G} \arrow{r}{\phi^{*}} &
      \Seg_{\mathcal{O}}(\mathcal{S})_{/F},
    \end{tikzcd}
  \]
  where $\Phi^{*}$ is given by composition with $\Phi$ and $\phi^{*}$
  by pullback along $\phi$. In other words, the equivalence of 
  \cref{lem:LFibSegslice} is natural.
\end{observation}

\begin{cor}\label{prop MonSeg}
  There is an equivalence
  \[ \mathcal{M}_{\mathcal{O}}(\mathcal{S}) \simeq
    \Seg_{\mathcal{O}}(\mathcal{S}) \] over
  $\Seg_{\mathcal{O}_{0}}(\mathcal{S})$, which takes an $\ObX$-monoid
  to its left Kan extension along the projection $\ObX \to \mathcal{O}$.  
\end{cor}
\begin{proof}
  From \cref{obs lke seg lfib nat} we obtain a natural equivalence
  \[ \Seg_{\ObX}(\mathcal{S}) \simeq
    \Seg_{\mathcal{O}}(\mathcal{S})_{/i_{0,*}\bX},\]
  given by left Kan extension along the projection $\ObX \to
  \mathcal{O}$. Let $\mathcal{E}_{\mathcal{O}} \to
  \Seg_{\mathcal{O}_{0}}(\mathcal{S})$ be the cartesian fibration for
  the left-hand functor; then by combining this natural equivalence
  with \cref{obs: seg over rke ftr} we have equivalences
  \begin{equation}
    \label{eq:segsliceeq}
    \mathcal{E}_{\mathcal{O}} \simeq   
    \Ar(\Seg_{\mathcal{O}}(\mathcal{S}))
    \times_{\Seg_{\mathcal{O}}(\mathcal{S})} 
    \Seg_{\mathcal{O}_{0}}(\mathcal{S})
    \simeq
    \Seg_{\mathcal{O}}(\mathcal{S})
    \times_{\Seg_{\mathcal{O}_{0}}(\mathcal{S})}
    \Ar(\Seg_{\mathcal{O}_{0}}(\mathcal{S})),
  \end{equation}
 where we can identify $\Seg_{\mathcal{O}}(\mathcal{S})$ with the full
 subcategory on the right consisting of pairs $(F, i_{0}^{*}F \to
 \bX)$ where the second component is an equivalence. Moreover,
 $\mathcal{M}_{\mathcal{O}}(\mathcal{S})$ is a full subcategory of
 $\mathcal{E}_{\mathcal{O}}$ by \cref{enrpattmonSeg}, so to complete
 the proof we only need to show that this full subcategory corresponds
 precisely to the image of $\Seg_{\mathcal{O}}(\mathcal{S})$ under
 \cref{eq:segsliceeq}. It suffices to check this on fibres over each
 $\bX \in \Seg_{\mathcal{O}_{0}}(\mathcal{S})$, where we want to
 identify $\Mon_{\ObX}(\mathcal{S})$ with Segal $\mathcal{O}$-spaces
 $M$ such that $i_{0}^{*}M \simeq \bX$. From 
 \cref{lem:monoidrestrpt} we know that
 $\Mon_{\ObX}(\mathcal{S})$ is the full subcategory of
 $\Seg_{\mathcal{O}_{\ObX}}(\mathcal{S})$ spanned by those Segal
 $\ObX$-spaces $M$ whose restriction to $\mathcal{O}_{0,\bX}$ is
 constant at $*$. Under the equivalence
 $\Seg_{\mathcal{O}_{\ObX}}(\mathcal{S}) \simeq
 \Seg_{\mathcal{O}}(\mathcal{S})_{/i_{0,*}\bX}$ given by left Kan
 extension along the projection $\pi \colon \ObX \to \mathcal{O}$,
 this condition corresponds to the map $i_{0}^{*}\pi_{!}M \to\bX$
 being an equivalence, since its fibres can be identified with the
 values of $M$ on $\mathcal{O}_{0,\bX}$. This is precisely the
 correspondence we want.
\end{proof}

\begin{proof}[Proof of \cref{thm:algdspc}]
  Combine the equivalences of \cref{lem AlgM} and \cref{prop MonSeg}.
\end{proof}

\subsection{Functoriality of algebroids}\label{subsec:algdfun}
In this section we will show that under certain conditions a Segal morphism
$f \colon \mathcal{O} \to \mathcal{P}$ over $\xF_{*}^{\natural}$
between enrichable patterns induces (under mild assumptions on the
$\mathcal{P}$-monoidal \icat{} $\mathcal{C}$) a commutative square
\[
  \begin{tikzcd}
  \Algd_{\mathcal{P}}(\mathcal{C}) \arrow{r}{f^{*}} \arrow{d}
  & \Algd_{\mathcal{O}}(\mathcal{C})   \arrow{d} \\
  \Seg_{\mathcal{P}_{0}}(\mathcal{S}) \arrow{r}{f_{0}^{*}} & \Seg_{\mathcal{O}_{0}}(\mathcal{S})
  \end{tikzcd}
\]
where $f^{*}$ preserves cartesian morphisms. To see why we need to impose a condition on $f$ for this to happen, we will first look at the case of
algebroids in
spaces, which we just saw are equivalent to Segal objects in
spaces. It is therefore convenient to start with some general
observations on (Segal) morphisms between enrichable patterns.

\begin{defn}
  If $\mathcal{O}$ and $\mathcal{P}$ are enrichable patterns, then a
  \emph{morphism of enrichable patterns} between them is a commutative
  triangle of algebraic patterns
  \[
    \begin{tikzcd}
      \mathcal{O} \arrow{dr} \arrow{rr}{f} & & \mathcal{P} \arrow{dl}
      \\
      & \xF_{*}^{\natural}.
    \end{tikzcd}
  \]
\end{defn}

\begin{observation}
  Any morphism $f \colon \mathcal{O} \to \mathcal{P}$ of
  enrichable patterns restricts to a morphism of algebraic patterns
  \[f_{0} := f \times_{\xF_{*}^{\natural}} \{\angled{0}\} \colon
    \mathcal{O}_{0} \to \mathcal{P}_{0}.\]
  Moreover, if $P$ is an object of $\mathcal{P}_{0}$ and
  $O\actto f(P)$ is an active morphism, then $O$ must lie in
  $\mathcal{O}_{0}$ (since there are no active morphisms $\angled{n}
  \to \angled{0}$ in $\xF_{*}$ unless $n = 0$). We thus have an
  equivalence of \icats{}
  \[ \mathcal{O}^{\act}_{0,/P} \isoto \mathcal{O}_{/P}^{\act}.\]
\end{observation}

\begin{warning}
  A morphism of
  enrichable patterns is \emph{not} necessarily a Segal
  morphism. (In particular, if $\mathcal{O}$ is an enrichable pattern,
  then the canonical morphism to $\xF_{*}^{\natural}$ may not be a Segal morphism.)
  We can give a simple criterion for such a map to be a strong
  Segal morphism, however:
\end{warning}

\begin{lemma}\label{enrpattxF*Segmor}
  A morphism $f \colon \mathcal{O} \to \mathcal{P}$ of enrichable
  patterns is a strong Segal morphism \IFF{} for every
  $O \in \mathcal{O}$ and every inert morphism
  $\phi \colon f(O) \intto E$ with $E \in \mathcal{P}^{\el}_{0}$, the
  \icat{} $(\mathcal{O}^{\el}_{O})_{/\phi}$ is weakly contractible. In
  particular, the morphism $\mathcal{O} \to \xF_{*}^{\natural}$ is a
  strong Segal morphism \IFF{} the \icat{} $\mathcal{O}^{\el}_{O/}$ is
  weakly contractible for all $O \in \mathcal{O}$.
\end{lemma}
\begin{proof}
  By \cite[Theorem 4.1.3.1]{ht}, the functor
  $\mathcal{O}^{\el}_{O/} \to \mathcal{P}^{\el}_{P/}$ is coinitial
  \IFF{} for all $\phi \in \mathcal{P}^{\el}_{P/}$, the \icat{}
  $(\mathcal{O}^{\el}_{O/})_{/\phi}$ is weakly contractible. If $\phi$
  lies over $\rho_{i}$ in $\xF_{*}$ then we can identify this \icat{}
  with
  $\mathcal{O}^{\flat,\el}_{O/} \times_{\mathcal{P}^{\flat,\el}_{P/}}
  (\mathcal{P}^{\flat,\el}_{P/})_{\phi/}$ (as there is no inert map
$\angled{0} \intto \angled{1}$), and this is a contractible \igpd{}
since $\mathcal{O}^{\flat}$ and $\mathcal{P}^{\flat}$ are cartesian.
It therefore suffices to know that the contractibility condition holds when $\phi$
lies over a map to $\angled{0}$, as claimed.
\end{proof}

\begin{lemma}\label{lem:segprescartcond}
  Let $\mathcal{C}$ be a complete \icat{}, and suppose $f \colon
  \mathcal{O} \to \mathcal{P}$ is a Segal morphism of enrichable
  patterns. Then we have a commutative square
  \begin{equation}
    \label{eq:seg0square}
    \begin{tikzcd}
      \Seg_{\mathcal{P}}(\mathcal{C}) \arrow{r}{f^{*}}
      \arrow{d}{i_{0}^{\mathcal{P},*}} &
      \Seg_{\mathcal{O}}(\mathcal{C}) \arrow{d}{i_{0}^{\mathcal{O},*}}
      \\
      \Seg_{\mathcal{P}_{0}}(\mathcal{C}) \arrow{r}{f_{0}^{*}} & \Seg_{\mathcal{O}_{0}}(\mathcal{C}),
    \end{tikzcd}
  \end{equation}
  where the vertical morphisms are cartesian fibrations. Here $f^{*}$
  takes $i_{0}^{\mathcal{P},*}$-cartesian morphisms to
  $i_{0}^{\mathcal{O},*}$-cartesian ones \IFF{} the Beck--Chevalley morphism
  \[ f^{*}i_{0,*}^{\mathcal{P}}F \to
    i_{0,*}^{\mathcal{O}}f_{0}^{*}F \] is an equivalence for every
  $F \in \Seg_{\mathcal{P}_{0}}(\mathcal{C})$. Moreover, this
  condition holds for all complete \icats{} $\mathcal{C}$ if it holds
  in spaces.
\end{lemma}
\begin{proof}
  From \cref{propn:SegCartFib} we know that a morphism $X \to Y$ in $\Seg_{\mathcal{P}}(\mathcal{C})$ is cartesian precisely if the commutative square
  \[
    \begin{tikzcd}
      X \arrow{r} \arrow{d} & Y \arrow{d} \\
      i_{0,*}^{\mathcal{P}}i_{0}^{\mathcal{P},*}X \arrow{r} & i_{0,*}^{\mathcal{P}}i_{0}^{\mathcal{P},*}Y
    \end{tikzcd}
  \]
  is a pullback. Applying $f^{*}$, we get the diagram
  \[
    \begin{tikzcd}
      f^{*}X \arrow{r} \arrow{d} & f^{*}Y \arrow{d} \\
      f^{*}i_{0,*}^{\mathcal{P}}i_{0}^{\mathcal{P},*}X \arrow{r} \arrow{d} & f^{*}i_{0,*}^{\mathcal{P}}i_{0}^{\mathcal{P},*}Y \arrow{d} \\
      i_{0,*}^{\mathcal{O}}i_{0}^{\mathcal{O},*}f^{*}XX \arrow{r} & i_{0,*}^{\mathcal{O}}i_{0}^{\mathcal{O},*}f^{*}Y,
    \end{tikzcd}
  \]
  where we have used the identification
  $f_{0}^{*}i_{0}^{\mathcal{P},*} \simeq
  i_{0}^{\mathcal{O},*}f^{*}$. Here the top square is a pullback
  (since $f^{*}$ preserves limits), so $f^{*}X \to f^{*}Y$ is cartesian if the bottom is also a pullback, which is the case if the natural transformation $f^{*}i_{0,*}^{\mathcal{P}} \to i_{0,*}^{\mathcal{O}}f_{0}^{*}$ is an equivalence at $i_{0}^{\mathcal{P},*}X$ and $i_{0}^{\mathcal{P},*}Y$. Conversely, for $F \in \Seg_{\mathcal{P}_{0}}(\mathcal{C})$, the morphism $i_{0,*}^{\mathcal{P}}F \to *$  to the terminal object is (trivially) a cartesian morphism in $\Seg_{\mathcal{P}}(\mathcal{C})$. If $f^{*}$ preserves cartesian morphisms then we must therefore have a pullback square
  \[
    \begin{tikzcd}
      f^{*}i_{0,*}^{\mathcal{P}}F \arrow{r} \arrow{d} & * \arrow[equals]{d} \\
      i_{0,*}^{\mathcal{O}}f_{0}^{*}F \arrow{r} & *,
    \end{tikzcd}
  \]
  which means that the left vertical map must be an equivalence. If this condition holds when $\mathcal{C}$ is the \icat{} of spaces, then the pointwise formula for right Kan extensions and the Yoneda lemma show that it also holds in $\mathcal{C}$.
\end{proof}

\begin{defn}
  We say a morphism of enrichable patterns
  $f \colon \mathcal{O} \to \mathcal{P}$ is \emph{simple} if it is a
  strong Segal morphism and satisfies the condition of
  \cref{lem:segprescartcond}.
\end{defn}

The pointwise formula for right Kan extensions shows immediately that a morphism of enrichable patterns $f \colon \mathcal{O} \to \mathcal{P}$ is simple if the induced functor
\[ \mathcal{O}_{0,O/} \to \mathcal{P}_{0,fO/}\] is coinitial for all
$O \in \mathcal{O}$, but this is a rather strong condition since it
implies that the natural map
$f^{*}i_{0,*}^{\mathcal{P}}F \to i_{0,*}^{\mathcal{O}}f_{0}^{*}F$ is
an equivalence for \emph{all} functors $F$, rather than just for Segal
$\mathcal{P}_{0}$-objects. Using the Segal condition we can find a weaker criterion: 
\begin{lemma}\label{obs:simplecond}
  Let $f \colon \mathcal{O} \to \mathcal{P}$ be a strong Segal morphism of enrichable patterns. If the functor
  \[ \mathcal{O}^{\el}_{0,E/} \to \mathcal{P}^{\el}_{0,fE/}\]
  induced by $f$ is coinitial for all $E \in \mathcal{O}^{\el}$, then $f$ is simple.
\end{lemma}
\begin{proof}
  We must show that for $F \in \Seg_{\mathcal{P}_{0}}(\mathcal{S})$ and $O \in \mathcal{O}$, the map
  \[ (f^{*}i_{0,*}^{\mathcal{P}}F)(O)  \to (i_{0,*}^{\mathcal{O}}f_{0}^{*}F)(O)\] is an equivalence. By \cref{lem i_0 Seg} the source and target are both Segal $\mathcal{O}$-spaces, so it suffices to show this for $O \in \mathcal{O}^{\el}$. Moreover, since $i_{0}$ has unique lifting of active morphisms, \cite[Lemma 6.2]{patterns} and the pointwise formula for right Kan extensions imply that this holds when
  \[ \lim_{\mathcal{P}^{\xint}_{0,fE/}} F \to \lim_{\mathcal{O}^{\xint}_{0,E/}} f_{0}^{*}F \]
  is an equivalence for $E \in \mathcal{O}^{\el}$ and $F \in \Seg_{\mathcal{P}_{0}}(\mathcal{S})$. Since $F$ is a Segal $\mathcal{P}_{0}$-space, $F|_{\mathcal{P}_{0}^{\xint}}$ is right Kan extended from $\mathcal{P}_{0}^{\el}$. If we write $\mathcal{E}_{E} \to \mathcal{P}^{\xint}_{0,fE/}$ for the cartesian fibration of the functor that takes $fE \intto X$ to $\mathcal{P}_{0,X/}^{\el}$ and $\mathcal{E}_{E} \to \mathcal{P}_{0}^{\el}$ for the projection, we can therefore rewrite the source as
  \[ \lim_{(fE \intto X) \in \mathcal{P}^{\xint}_{0,fE/}} F(X) \simeq \lim_{(fE \intto X) \in \mathcal{P}^{\xint}_{0,fE/}} \lim_{(X \intto E') \in \mathcal{P}^{\el}_{0,X/}} F(E') \simeq \lim_{\mathcal{E}_{E}} p^{*}F. \]
  Now we observe that the functor $p$ factors through a functor $q \colon \mathcal{E}_{E} \to \mathcal{P}^{\el}_{0,fE/}$ that takes 
  $(fE \intto X, X \intto E')$ to the composite $fE \intto E'$. We claim that $q$ is coinitial. To see this it suffices to show that for $\alpha \colon fE \intto E'$ with $E' \in \mathcal{P}_{0}^{\el}$, the \icat{} $(\mathcal{E}_{E})_{/\alpha}$ is weakly contractible, and this is clear since $(\alpha, \id_{E'})$ is a terminal object. It follows that $(f^{*}i_{0,*}^{\mathcal{P}}F)(E)$ can be computed as $\lim_{fE \intto E' \in \mathcal{P}^{\el}_{0,fE/}} F(E')$, and applying the same argument to $(i_{0,*}^{\mathcal{O}}f_{0}^{*}F)(E)$ we see that the map between them will be an equivalence if $\mathcal{O}^{\el}_{0,E/} \to \mathcal{P}^{\el}_{0,fE/}$ is coinitial, as required.
\end{proof}

\begin{observation}\label{obs:simplemonmor}
  If $\mathcal{C}$ is the \icat{} $\mathcal{S}$ of spaces and $f \colon \mathcal{O} \to \mathcal{P}$ is a simple morphism of enrichable patterns, then  the
  square \cref{eq:seg0square} corresponds under the equivalence of
  \cref{prop MonSeg}
  to a natural transformation given
  at $\bX \in \Seg_{\mathcal{P}_{0}}(\mathcal{S})$ by a functor
  \[ f^{*}_{\bX} \colon \Mon_{\mathcal{P}_{\bX}}(\mathcal{S}) \to
    \Mon_{\mathcal{O}_{f_{0}^{*}\bX}}(\mathcal{S}).\]
  This is simply given by composing with the projection
  $f_{\bX} \colon f^{*}\mathcal{P}_{\bX} \to \mathcal{P}_{\bX}$ and using the identification
  \begin{equation}
    \label{eq:simplepb}
  f^{*}\mathcal{P}_{\bX} \simeq \mathcal{O}_{f_{0}^{*}\bX}  
  \end{equation}
  of right fibrations over $\mathcal{O}$, corresponding to the equivalence $f^{*}i_{0,*}^{\mathcal{P}}\bX
  \simeq i_{0,*}^{\mathcal{O}}f_{0}^{*}\bX$. 
\end{observation}

\begin{construction}
  Given a simple morphism of enrichable patterns $f \colon \mathcal{O}
  \to \mathcal{P}$ and an
  $\mathcal{O}$-monoidal \icat{} $\mathcal{V}$, we
  will define an induced functor on algebroids that fits in a
  commutative square
  \[
    \begin{tikzcd}
      \Algd_{\mathcal{P}}(\mathcal{V}) \arrow{r}{f^{*}} \arrow{d} &
      \Algd_{\mathcal{O}}(\mathcal{V}) \arrow{d} \\
      \Seg_{\mathcal{P}_{0}}(\mathcal{C}) \arrow{r}{f_{0}^{*}} & \Seg_{\mathcal{O}_{0}}(\mathcal{C}),
    \end{tikzcd}
  \]
  such that $f^{*}$ preserves cartesian morphisms. On the fibre over
  $\bX \in \Seg_{\mathcal{P}_{0}}(\mathcal{C})$, we define this as the
  restriction functor
  \[  f_{\bX}^{*} \colon \Alg_{\mathcal{P}_{\bX}}(\mathcal{V}) \to 
    \Alg_{\mathcal{O}_{f_{0}^{*}\bX}}(\mathcal{V}),\]
  and then use the obvious naturality of this in $\bX$. Note that this is also clearly compatible with composition of simple morphisms.
\end{construction}

\begin{observation}\label{prop nat}
  Let $g\colon \mathcal{C}\to \mathcal{D}$ be a lax $\mathcal{O}$-monoidal functor. Then composition with $g$ induces a commutative triangle 
  \[
  \begin{tikzcd}[column sep=small]
  \Algd_\mathcal{O}(\mathcal{C}) \arrow[rr, "g_*"] \arrow[rd] & & \Algd_{\mathcal{O}}(\mathcal{D})\arrow[ld]\\
  & \Seg_{\mathcal{O}_0}(\xS), &
  \end{tikzcd}
\]
since $g$ clearly induces a natural transformation of functors $\Alg_{\mathcal{O}_{\blank}/\mathcal{O}}(\mathcal{C})\to \Alg_{\mathcal{O}_{\blank}/\mathcal{O}}(\mathcal{D})$.
\end{observation}

\begin{propn}\label{propn:algdtomonnat}
  The equivalence of \cref{prop MonSeg} is natural in simple morphisms of enrichable patterns.
\end{propn}
\begin{proof}
  For an enrichable pattern $\mathcal{O}$, consider the pullback
  \[
    \begin{tikzcd}
      \mathcal{F}_{\mathcal{O}} \arrow{r} \arrow{d} & \Ar(\LFib(\mathcal{O})) \arrow{d} \\
      \Seg_{\mathcal{O}}(\mathcal{S}) \times \Seg_{\mathcal{O}_{0}}(\mathcal{S}) \arrow{r} & \LFib(\mathcal{O}) \times \LFib(\mathcal{O})
    \end{tikzcd}
  \]
  with the bottom horizontal functor given by straightening over $\mathcal{O}$ in the first variable and by right Kan extending to $\mathcal{O}$ and then straightening in the second.
  We can view $\mathcal{M}_{\mathcal{O}}(\mathcal{S})$ as a full subcategory of
  $\mathcal{F}_{\mathcal{O}}$ by thinking of an $\mathcal{O}_{\bX}$-monoid $M$ in $\mathcal{S}$ as a morphism
  \[
    \begin{tikzcd}
      \mathcal{M} \arrow{rr} \arrow{dr} & & \mathcal{O}_{\bX} \arrow{dl} \\
       & \mathcal{O}
    \end{tikzcd}
  \]
  of left fibrations over $\mathcal{O}$.
  The equivalence to $\Seg_{\mathcal{O}}(\mathcal{S})$ then corresponds to the domain projection.

  If $f \colon \mathcal{P} \to \mathcal{O}$ is a simple morphism, pulling back along $f$ gives a functor $\mathcal{F}_{\mathcal{O}} \to \mathcal{F}_{\mathcal{P}}$ over $f^{*} \times f_{0}^{*}$, which restricts to a functor $\mathcal{M}_{\mathcal{O}}(\mathcal{S}) \to \mathcal{M}_{\mathcal{P}}(\mathcal{S})$ over $f^{*}$. In other words, we have a commutative square
  \[
    \begin{tikzcd}
      \mathcal{M}_{\mathcal{O}}(\mathcal{S}) \arrow{r}{f^{*}} \arrow{d}[swap]{\sim} & \mathcal{M}_{\mathcal{P}}(\mathcal{S}) \arrow{d}{\sim} \\
      \Seg_{\mathcal{O}}(\mathcal{S}) \arrow{r}{f^{*}} & \Seg_{\mathcal{P}}(\mathcal{S}),
    \end{tikzcd}
  \]
  as required.
\end{proof}

\subsection{Enrichment in presheaves and presentability}

If  $\mathcal{C}^{\otimes}$ is a small $\mathcal{O}$-monoidal
\icat{}, then by \cite[\S 6]{freealg} there is a Day convolution
$\mathcal{O}$-monoidal structure
$\PSh_{\mathcal{O}}(\mathcal{C})^{\otimes}$ on presheaves. In this
section we will show that algebroids in
$\PSh_{\mathcal{O}}(\mathcal{C})$ admit a simple description, namely that
\[ \Algd_{\mathcal{O}}(\PSh_{\mathcal{O}}(\mathcal{C})) \simeq
  \Seg_{\mathcal{C}^{\op,\otimes}}(\mathcal{S}). \] We will then see
that this equivalence is compatible with $\mathcal{O}$-monoidal
localizations of $\PSh_{\mathcal{O}}(\mathcal{C})$, and use this to
show that algebroids in any presentably $\mathcal{O}$-monoidal \icat{}
form a presentable \icat{}.

\begin{propn}\label{prop AlgdM}
  Let $\mathcal{O}$ be an enrichable pattern and
  $\mathcal{C}^{\otimes}$ a small $\mathcal{O}$-monoidal \icat{}. Then
  there is a natural equivalence
  \[ \Algd_{\mathcal{O}}(\PSh_{\mathcal{O}}(\mathcal{C})) \simeq
    \mathcal{M}_{\mathcal{C}^{\op,\otimes}}(\mathcal{S}) \]
  over $\Seg_{\mathcal{O}_{0}}(\mathcal{S})$ which is compatible with simple morphisms of enrichable patterns.
\end{propn}
\begin{proof}
  Since $\mathcal{C}^{\otimes}$ is $\mathcal{O}$-monoidal, it is by definition a Segal $\mathcal{O}^\flat$-fibration and therefore $\mathcal{C}^{\otimes}_{0} \simeq \mathcal{O}_{0}$. Hence,
  \cref{obs:simplecond} implies that the projection $\mathcal{C}^{\otimes} \to \mathcal{O}$ is a simple morphism of enrichable patterns, so that 
for $\bX \in \Seg_{\mathcal{O}_{0}}(\mathcal{S})$ we have a natural equivalence
  \[ \mathcal{C}^{\otimes}_{\bX} \simeq \mathcal{C}^{\otimes}
    \times_{\mathcal{O}} \mathcal{O}_{\bX}.\]
From \cite[Corollary 6.21]{freealg} we have for any
morphism of cartesian patterns $f \colon \mathcal{P} \to \mathcal{O}^{\flat}$
 a natural equivalence
\[ \Alg_{\mathcal{P}/\mathcal{O}}(\PSh_{\mathcal{O}}(\mathcal{C}))
  \simeq \Mon_{\mathcal{P} \times_{\mathcal{O}}
    \mathcal{C}^{\op,\otimes}}(\mathcal{S}).\]
Applying this to the cartesian patterns
$\mathcal{O}_{\bX}^{\flat}$, we get equivalences
\[
  \Alg_{\mathcal{O}_{\bX}/\mathcal{O}}(\PSh_{\mathcal{O}}(\mathcal{C}))
  \simeq \Mon_{\mathcal{O}_{\bX} \times_{\mathcal{O}}
    \mathcal{C}^{\op,\otimes}}(\mathcal{S}) \simeq
  \Mon_{\mathcal{C}^{\op,\otimes}_{\bX}}(\mathcal{S}) \simeq
  \mathcal{M}_{\mathcal{C}^{\op,\otimes}}(\mathcal{S})_{\bX},\]
natural in $\bX$.

To show the compatibility with simple morphisms we first observe that for every $\bX\in \Seg_{\mathcal{O}_0}(\xS)$, the pullback along a simple morphism $f\colon 
  \mathcal{P} \to \mathcal{O}$ takes an object $A\in 
  \Alg_{\mathcal{O}_\bX/\mathcal{O}}(\PSh_{\mathcal{O}}(\mathcal{C}))$ given by
  \[
  \begin{tikzcd}
  \mathcal{O}_\bX \arrow[rr, "A"] \arrow[rd]&& \PSh_{\mathcal{O}}(\mathcal{C})^{\otimes}\arrow[ld]\\
  &\mathcal{O}&
  \end{tikzcd}
  \]
  to the commutative triangle 
  \[
  \begin{tikzcd}
  \mathcal{O}_\bX\times_\mathcal{O} \mathcal{P} \arrow[rr, "f^*A"] \arrow[rd]&& \PSh_{\mathcal{O}}(\mathcal{C})\times_\mathcal{O} \mathcal{P}\arrow[ld]\\
  &\mathcal{P}.&
  \end{tikzcd}
  \]
  Here we have
  $\mathcal{O}_\bX\times_{\mathcal{O}} \mathcal{P} \simeq
  \mathcal{P}_{f_{0}^*\bX}$ and
  $\PSh_{\mathcal{O}}(\mathcal{C})\times_\mathcal{O} \mathcal{P}\simeq
  \PSh_{\mathcal{P}}(f^*\mathcal{C})$ by \cref{eq:simplepb} in
  \cref{obs:simplemonmor} and \cite[Lemma 6.20]{freealg},
  respectively. It follows from the construction that $f^*A$ takes
  inert morphisms to cocartesian morphisms, \ie{} $f^*A$ lies in
  $\Alg_{\mathcal{P}_{f_{0}^*\bX}/\mathcal{P}}(\PSh_{\mathcal{P}}(f^*\mathcal{C}))$,  so that
  we have a functor
  $\Alg_{\mathcal{O}_\bX/\mathcal{O}}(\PSh_{\mathcal{O}}(\mathcal{C}))\to
  \Alg_{\mathcal{P}_{f_{0}^*\bX}/\mathcal{P}}(\PSh_{\mathcal{P}}(f^*\mathcal{C}))$
  natural in $\bX$. By the first part of the proof and \cite[\S
  6]{freealg} this functor can be identified with the functor
  $\Mon_{\mathcal{C}^{\op,\otimes}_{\bX}}(\mathcal{S})\to
  \Mon_{(f^*\mathcal{C})^{\op,\otimes}_{f_{0}^{*}\bX}}(\mathcal{S})$ induced by the
  canonical map
  \[(f^*\mathcal{C})^{\op,\otimes}_{f_{0}^{*}\bX} \simeq f^{*}(\mathcal{C}^{\op,\otimes}_{\bX}) \to
  \mathcal{C}^{\op,\otimes}_{\bX}.\] The naturality in $\bX$ then gives a
commutative diagram 
\[
  \begin{tikzcd}[column sep=small]
    \Algd_{\mathcal{O}}(\PSh_\mathcal{O}(\mathcal{C})) \arrow{dd} \arrow{rr}{f^{*}} \arrow[dash]{dr}{\sim} &  & \Algd_{\mathcal{P}}\PSh_\mathcal{P}(f^*\mathcal{C}) \arrow[dash]{dr}{\sim} \arrow{dd} \\
    & \mathcal{M}_{\mathcal{C}^{\op,\otimes}}(\xS) \arrow[crossing over]{rr}[near start]{f^{*}} \arrow{dl} & &    \mathcal{M}_{f^*\mathcal{C}^{\op,\otimes}}(\xS) \arrow{dl} \\
    \Seg_{\mathcal{O}_0}(\xS) \arrow{rr}{f_{0}^{*}} & & \Seg_{\mathcal{P}_0}(\xS),
  \end{tikzcd}
\]
  so that the equivalence
  $\Algd_{\mathcal{O}}(\PSh_{\mathcal{O}}(\mathcal{C})) \simeq
  \mathcal{M}_{\mathcal{C}^{\op,\otimes}}(\mathcal{S})$ is indeed
  compatible with simple morphisms.
\end{proof}

Combining \cref{prop AlgdM} with \cref{prop MonSeg}, we get:
\begin{cor}\label{AlgdDayisSeg}
    Let $\mathcal{O}$ be an enrichable pattern and
  $\mathcal{C}^{\otimes}$ a small $\mathcal{O}$-monoidal \icat{}. Then
  there is an equivalence
  \[ \Algd_{\mathcal{O}}(\PSh_{\mathcal{O}}(\mathcal{C})) \simeq \Seg_{\mathcal{C}^{\op,\otimes}}(\mathcal{S}),\]
  natural in simple morphisms of enrichable patterns. \qed
\end{cor}

\begin{defn}
  Let $\mathcal{O}$ be a cartesian pattern. We say an
  $\mathcal{O}$-monoidal \icat{} $\mathcal{V}^{\otimes}$ is
  \emph{compatible with small colimits} if the \icats{}
  $\mathcal{V}_{E}$ are cocomplete for all $E \in \mathcal{O}^{\el}$,
  and for every active map $\alpha \colon O \actto E$ in $\mathcal{O}$
  with $E \in \mathcal{O}^{\el}$, the cocartesian transport
  functor
  \[\alpha_{!} \colon \prod_{i} \mathcal{V}_{O_{i}} \simeq
    \mathcal{V}^{\otimes}_{O} \to \mathcal{V}_{E} \] preserves small
  colimits in each variable. We say that $\mathcal{V}^{\otimes}$ is
  \emph{presentably $\mathcal{O}$-monoidal} if it is compatible with
  small colimits and in addition the \icats{} $\mathcal{V}_{E}$ for
  $E \in \mathcal{O}^{\el}$ are all presentable.
\end{defn}

\begin{defn}
  Given a commutative triangle
  \[
    \begin{tikzcd}
      \mathcal{C} \arrow{rr}{G} \arrow{dr} & & \mathcal{D}
      \arrow{dl} \\
      & \mathcal{B},
    \end{tikzcd}
  \]
  we say that $G$ has a left adjoint \emph{relative to $\mathcal{B}$}
  if $G$ has a left adjoint $F \colon \mathcal{D} \to \mathcal{C}$ and
  the unit map $d \to GFd$ maps to an equivalence in $\mathcal{B}$ for
  all $d \in \mathcal{D}$. (Then $F$ is automatically a functor over $\mathcal{B}$ by \cite[Remark 7.3]{freealg}.)
\end{defn}

\begin{defn}
  If $G \colon \mathcal{C}^{\otimes} \to \mathcal{D}^{\otimes}$ is a
  lax $\mathcal{O}$-monoidal functor, we say that $G$ has an
  \emph{$\mathcal{O}$-monoidal left adjoint} $F$ if $F$ is a left
  adjoint to $G$ relative to $\mathcal{O}$; then $F$ is necessarily an
  $\mathcal{O}$-monoidal functor by \cite[Remark 7.5]{freealg}. We say
  $F$ is an \emph{$\mathcal{O}$-monoidal localization} if in addition
  $G$ is fully faithful.
\end{defn}

We now want to look at algebroids in localizations of Day
convolutions at classes of maps compatible
with the $\mathcal{O}$-monoidal structure in the following sense:

\begin{definition}
  Let $\mathcal{C}^{\otimes}$ be a small $\mathcal{O}$-monoidal
  \icat{}. A collection $\mathbb{S} = (\mathbb{S}_{E})_{E \in \mathcal{O}^{\el}}$
  of sets of morphisms $\mathbb{S}_{E}$ in $\PSh(\mathcal{C}_{E})$ is
  \emph{compatible with the $\mathcal{O}$-monoidal structure} if for
  every active morphism $\phi \colon O
  \actto E$ in $\mathcal{O}$ with $E \in \mathcal{O}^{\el}$, the functor
  \[\phi_{!}^{\PSh_\mathcal{O}(\mathcal{C})} \colon \prod_{i} \PSh(\mathcal{C}_{O_{i}}) \simeq
  \PSh_{\mathcal{O}}(\mathcal{C})^{\otimes}_{O} \to
  \PSh(\mathcal{C}_{E}) \]
  takes a morphism $(\id,\ldots,\id,s,\id,\ldots,\id)$ with $s
  \in \mathbb{S}_{O_{i}}$ into the strongly saturated
  class $\overline{\mathbb{S}}_{E}$ generated by $\mathbb{S}_{E}$. We
  then write $\PSh_{\mathcal{O},\mathbb{S}}(\mathcal{C})^{\otimes}$
  for the full subcategory generated by the collection of full subcategories
  $\PSh_{\mathbb{S}_{E}}(\mathcal{C}_{E})$ of $\mathbb{S}_{E}$-local
  objects, \ie{} the objects $\Phi$ such that
  $\Map_{\PSh(\mathcal{C}_E)}(\blank, \Phi)$ takes the morphisms in
  $\mathbb{S}_E$ to equivalences.
\end{definition}

\begin{remark}\label{rem monloc}
  Let $\mathcal{C}^\otimes$ and $\mathbb{S}$ be as in the previous definition. Then \cite[Corollary 7.19]{freealg} shows that there is an $\mathcal{O}$-monoidal localization
  \[\PSh_{\mathcal{O}}(\mathcal{C})^\otimes\xto{L_{\mathbb{S}}}
  \PSh_{\mathcal{O},\mathbb{S}}(\mathcal{C})^\otimes\]
  left adjoint to the inclusion
  $\PSh_{\mathcal{O},\mathbb{S}}(\mathcal{C})^{\otimes}
  \hookrightarrow \PSh_{\mathcal{O}}(\mathcal{C})^{\otimes}$.
\end{remark}

\begin{remark}\label{rem Vk}
  Given a presentably
  $\mathcal{O}$-monoidal \icat{} $\mathcal{V}^{\otimes}$, then \cite[Corollary 7.15 and Remark 7.21]{freealg} show that there exists a 
  a regular cardinal $\kappa$ and a small $\mathcal{O}$-monoidal \icat{}
  $\mathcal{V}^{\kappa,\otimes}$ such that 
  \[\mathcal{V}^{\otimes}\simeq\PSh_{\mathcal{O},\mathbb{S}^{\kappa}}(\mathcal{V}^{\kappa})^{\otimes},\]
  where $\mathbb{S}^{\kappa}_{E}$ denotes the collection of maps
  \[ \colim_{\mathcal{I}} y(\phi) \to y(\colim_{\mathcal{I}} \phi) \]
  in $\PSh(\mathcal{V}^{\kappa}_{E})$ where $\phi \colon \mathcal{I}
  \to \mathcal{V}^{\kappa}_{E}$ ranges over a set of representatives
  of $\kappa$-small colimit diagrams. In particular, for any presentably
  $\mathcal{O}$-monoidal \icat{} $\mathcal{V}^{\otimes}$ there exists
  a small $\mathcal{O}$-monoidal \icat{} $\mathcal{C}^{\otimes}$
  and an $\mathcal{O}$-monoidal localization
  \[ \PSh_{\mathcal{O}}(\mathcal{C})^{\otimes} \xto{L}
  \mathcal{V}^{\otimes}.\]
\end{remark}

\begin{definition}
  Let $\mathcal{C}^{\otimes}$ be a small $\mathcal{O}$-monoidal
  \icat{} and let $\mathbb{S}=(\mathbb{S}_E)_{E\in \mathcal{O}^\el}$ be a collection of sets of morphisms
  compatible with the $\mathcal{O}$-monoidal structure on
  $\PSh_{\mathcal{O}}(\mathcal{C})^{\otimes}$. If for $E\in \mathcal{O}^\el$ the map $\iota_E\colon \mathcal{C}_E^\op \to \mathcal{C}^{\op,\otimes}$ denotes the natural inclusion of the fibre at $E$, we write $\Seg^\mathbb{S}_{\mathcal{C}^{\op, \otimes}}(\xS)$ for the full subcategory of $\xFun(\mathcal{C}^{\op, \otimes},\xS)$ spanned by Segal objects $F$ such that $\iota_E^*(F)\in \PSh(\mathcal{C}_E)$ is $\mathbb{S}_E$-local for every $E$.
\end{definition}

\begin{proposition}\label{prop:algdSloc}
  Let $\mathcal{C}^{\otimes}$ be a small $\mathcal{O}$-monoidal
  \icat{} and let $\mathbb{S} = (\mathbb{S}_{E})_{E \in \mathcal{O}^{\el}}$ be
  sets of morphisms which are compatible with the $\mathcal{O}$-monoidal structure on
  $\PSh_{\mathcal{O}}(\mathcal{C})^{\otimes}$. Then the natural equivalence $\Algd_{\mathcal{O}}(\PSh_{\mathcal{O}}(\mathcal{C})) \simeq
  \mathcal{M}_{\mathcal{C}^{\op,\otimes}}(\mathcal{S})$ of \cref{prop AlgdM} restricts to an equivalence 
  \[
  \Algd_{\mathcal{O}}(\PSh_{\mathcal{O},\mathbb{S}}(\mathcal{C}))\simeq \Seg_{\mathcal{C}^{\op, \otimes}}^\mathbb{S}(\xS).
  \]
\end{proposition}
\begin{proof}
  Composition with the lax $\mathcal{O}$-monoidal inclusion
  $\PSh_{\mathcal{O},\mathbb{S}}(\mathcal{C})^{\otimes}
  \hookrightarrow \PSh_{\mathcal{O}}(\mathcal{C})^{\otimes}$ of
  \cref{rem monloc} gives (\cf{} \cref{prop nat}) a
  functor
  \[\Algd_\mathcal{O}(\PSh_{\mathcal{O},\mathbb{S}}(\mathcal{C}))\to
    \Algd_\mathcal{O}(\PSh_{\mathcal{O}}(\mathcal{C})),\] which is
  then clearly fully faithful.
  
  By definition an object $\Phi\in \Alg_{\mathcal{O}_{\bX}/\mathcal{O}}(\PSh_{\mathcal{O}}(\mathcal{C}))$ lies in $\Alg_{\mathcal{O}_{\bX}/\mathcal{O}}(\PSh_{\mathcal{O},\mathbb{S}}(\mathcal{C}))$ if and only if for every $E_X\in \mathcal{O}^\el_\bX$ the object $\Phi(E_X)$ is $\mathbb{S}_E$-local. Therefore, under the equivalence $ \Alg_{\mathcal{O}_{\bX}/\mathcal{O}}(\PSh_{\mathcal{O}}(\mathcal{C}))
  \simeq \Mon_{\mathcal{C}^{\op,\otimes}_{\bX}}(\mathcal{S})$ of \cref{prop AlgdM}, 
  the full subcategory \[\Alg_{\mathcal{O}_{\bX}/\mathcal{O}}(\PSh_{\mathcal{O}, \mathbb{S}}(\mathcal{C}))\subseteq \Alg_{\mathcal{O}_{\bX}/\mathcal{O}}(\PSh_{\mathcal{O}}(\mathcal{C}))\] corresponds to the full subcategory of $\Mon_{\mathcal{C}^{\op,\otimes}_{\bX}}(\xS)$ spanned by objects $\Phi$
  such that for every $E_X\in \mathcal{O}^\el_\bX$, $\Phi$ restricts to an $\mathbb{S}_E$-local object  $\{E_X\}\times_{\mathcal{O}}\mathcal{C}^{\op, \otimes}\simeq \mathcal{C}_E^\op\to \xS$. Finally, the equivalence \[\Mon_{\mathcal{C}^{\op,\otimes}_{\bX}}(\xS)\simeq \Seg_{\mathcal{C}^{\op, \otimes}}(\xS)_\bX\] from \cref{prop MonSeg} shows that $\Alg_{\mathcal{O}_{\bX}/\mathcal{O}}(\PSh_{\mathcal{O}, \mathbb{S}}(\mathcal{C}))$ can be identified with $\Seg_{\mathcal{C}^{\op, \otimes}}^{\mathbb{S}}(\xS)_\bX$. Since this identification is natural in $\bX$, it gives an equivalence \[\Algd_{\mathcal{O}}(\PSh_{\mathcal{O},\mathbb{S}}(\mathcal{C}))\simeq \Seg_{\mathcal{C}^{\op, \otimes}}^\mathbb{S}(\xS) \qedhere\]
\end{proof}
\begin{corollary}\label{cor:algdpres}
  If $\xV^\otimes$ is presentably $\mathcal{O}$-monoidal, then $\Algd_{\mathcal{O}}(\xV)$ is presentable.
\end{corollary}
\begin{proof}
  According to \cref{rem Vk} there exists a small $\mathcal{O}$-monoidal \icat{}
  $\mathcal{C}^{\otimes}$ and a collection
        $\mathbb{S}=(\mathbb{S}_E)_{E\in \mathcal{O}^\el}$ such that $\mathcal{V}^{\otimes}\simeq\PSh_{\mathcal{O},\mathbb{S}}(\mathcal{C})^{\otimes}$. \cref{prop:algdSloc} then gives an equivalence 
  \[  \Algd_{\mathcal{O}}(\xV)\simeq \Seg_{\mathcal{C}^{\op, \otimes}}^{\mathbb{S}}(\xS).\]
  Here $\Seg_{\mathcal{C}^{\op, \otimes}}(\xS)$ is presentable by \cite[Lemma 2.10]{patterns}. The \icat{} $\Seg_{\mathcal{C}^{\op, \otimes}}^{\mathbb{S}}(\xS)$ is an accessible localization of this, and so is again presentable.
\end{proof}

\subsection{Naturality of Day convolution}

Our aim in this section is to show that the equivalence
\[ \Algd_{\mathcal{O}}(\PSh_{\mathcal{O}}(\mathcal{C})) \simeq \Seg_{\mathcal{C}^{\op,\otimes}}(\mathcal{S})\]
is natural in the $\mathcal{O}$-monoidal \icat{} $\mathcal{C}$. The key input for this is the naturality of the equivalence
\[ \Alg_{\mathcal{O}}(\PSh_{\mathcal{O}}(\mathcal{C})) \simeq \Mon_{\mathcal{C}^{\op,\otimes}}(\mathcal{S}), \]
in $\mathcal{C}$, which was not discussed in \cite{freealg}. We will prove this by identifying the (co)cartesian fibrations for the two functors.

\begin{propn}\label{obs:slicefunfib}
  Suppose $p \colon \mathcal{E} \to \mathcal{B}$ is the cocartesian
  fibration for a functor $F \colon \mathcal{B} \to \Cat$. Given a functor
  $\mathcal{E} \to \mathcal{K}$ that takes $p$-cocartesian morphisms
  to equivalences (which corresponds to a natural transformation
  $F \to \mathcal{K}$) and a map $f \colon \mathcal{L} \to \mathcal{K}$, we can identify the cocartesian fibration
  for the functor $\Fun_{/\mathcal{K}}(\mathcal{L}, F(\blank))$ as the pullback
  \[ \Fun_{/\mathcal{K}}(\mathcal{L}, \mathcal{E}) \times_{\Fun_{/\mathcal{K}}(\mathcal{L}, \mathcal{B} \times \mathcal{K})} \mathcal{B} \to \mathcal{B}.\]
\end{propn}
\begin{proof}
  For any \icat{} $\mathcal{L}$ we know from \cite[Proposition 3.1.2.1]{ht} that
  $\mathcal{E}^{\mathcal{L}} \to \mathcal{B}^{\mathcal{L}}$ is also a
  cocartesian fibration, and hence so is its base change
  $\mathcal{E}^{\mathcal{L}} \times_{\mathcal{B}^{\mathcal{L}}}
  \mathcal{B} \to \mathcal{B}$ to the constant functors. The latter
  fibration corresponds to the functor $\Fun(\mathcal{L}, F(\blank))$ ---
  indeed, this construction is evidently right adjoint to the functor
  \[ (\mathcal{E} \to \mathcal{B}) \mapsto (\mathcal{E} \times
    \mathcal{L} \to \mathcal{B}),\] which under straightening corresponds to
  $F \mapsto \mathcal{L} \times F(\blank)$.

  We now observe that the fibration we want to identify fits in a pullback square
  \[
    \begin{tikzcd}
    \Fun_{/\mathcal{K}}(\mathcal{L}, \mathcal{E}) \times_{\Fun_{/\mathcal{K}}(\mathcal{L}, \mathcal{B} \times \mathcal{K})} \mathcal{B} \arrow{r} \arrow{d} \pbcorner & \mathcal{E}^{\mathcal{L}} \times_{\mathcal{B}^{\mathcal{L}}} \mathcal{B} \arrow{d} \\
      \mathcal{B} \times \{\id_{\mathcal{K}}\} \arrow{r} & \mathcal{B} \times \mathcal{K}^{\mathcal{L}}
    \end{tikzcd}
  \]
  of cocartesian fibrations over $\mathcal{B}$. This corresponds on functors to the pullback square
  \[
    \begin{tikzcd}
      \Fun_{/\mathcal{K}}(\mathcal{L}, F(\blank))  \pbcorner \arrow{r} \arrow{d} & \Fun(\mathcal{L}, F(\blank)) \arrow{d} \\
      \{f \} \arrow{r} & \Fun(\mathcal{L}, \mathcal{K}),
    \end{tikzcd}
  \]
  as required.
\end{proof}

\begin{defn}
  Let $\mathcal{O}$ be a cartesian pattern. We define the cocartesian fibration
  \begin{equation}
    \label{eq:daytensfib}
    \Day^{\otimes}_{\mathcal{O}} \to \Mon_{\mathcal{O}}(\CatI) \times \mathcal{O}
 \end{equation}
 as the pullback
  \[
    \begin{tikzcd}
      \Day^{\otimes}_{\mathcal{O}} \pbcorner \arrow{r} \arrow{d} & \RFib^{\times} \arrow{d} \\
      \Mon_{\mathcal{O}}(\CatI) \times \mathcal{O} \arrow{r} & \CatI^{\times},
    \end{tikzcd}
  \]
  where the bottom functor is adjoint to the equivalence
  $\Mon_{\mathcal{O}}(\CatI) \simeq
  \Alg_{\mathcal{O}}(\CatI^{\times})$.
\end{defn}

\begin{observation}\label{obs:dayotimesfunctor}
  The cocartesian fibration
  \[ \Day^{\otimes}_{\mathcal{O}} \to \Mon_{\mathcal{O}}(\CatI)\]
  obtained from \cref{eq:daytensfib} corresponds to
  the (covariant) functor
  $\mathcal{C}^{\otimes} \mapsto
  \PSh_{\mathcal{O}}(\mathcal{C})^{\otimes}$, given on a morphism
  $f \colon \mathcal{C}^{\otimes} \to \mathcal{D}^{\otimes}$ by the $\mathcal{O}$-monoidal functor
  \[ F_{!} \colon \PSh_{\mathcal{O}}(\mathcal{C})^{\otimes} \to
    \PSh_{\mathcal{O}}(\mathcal{D})^{\otimes}\] given over an
  elementary object $E \in \mathcal{O}$ by left Kan extension along
  $F_{E}$. Here $F_{!}$ has a lax $\mathcal{O}$-monoidal right adjoint
  $F^{*}$ (apply \cite[Proposition 7.3.2.6]{ha} as in the proof of
  \cite[Corollary 7.3.2.7]{ha}); this is given over an elementary
  object $E$ by composition with $F_{E}$. In particular, the
  projection from $\Day^{\otimes}_{\mathcal{O}}$ to
  $\Mon_{\mathcal{O}}(\CatI)$ is also a cartesian fibration.
\end{observation}

\begin{cor}\label{lem:algopshfib}
  Let $\phi \colon \mathcal{Q} \to \mathcal{O}$ be a morphism of cartesian patterns. The cocartesian fibration for the functor $\Alg_{\mathcal{Q}/\mathcal{O}}(\PSh_{\mathcal{O}}(\blank))$ is the pullback
    \begin{equation}
    \label{eq:algpshfib}
    \begin{tikzcd}
      \Alg'_{\mathcal{Q}/\mathcal{O}}(\Day_{\mathcal{O}}^{\otimes}) \arrow{r} \arrow{d} \pbcorner & \Alg_{\mathcal{Q}/\mathcal{O}}(\Day^{\otimes}_{\mathcal{O}}) \arrow{d} \\
      \Mon_{\mathcal{O}}(\CatI) \arrow{r} & \Alg_{\mathcal{Q}/\mathcal{O}}(\Mon_{\mathcal{O}}(\CatI) \times \mathcal{O}).
    \end{tikzcd}
  \end{equation}
\end{cor}
\begin{proof}
  By \cref{obs:slicefunfib}, the functor $\Fun_{/\mathcal{O}}(\mathcal{Q}, \PSh_{\mathcal{O}}(\blank))$ corresponds to the fibration
\[    \Fun_{/\mathcal{O}}(\mathcal{Q}, \Day_{\mathcal{O}}^{\otimes}) \times_{\Fun_{/\mathcal{O}}(\mathcal{Q}, \Mon_{\mathcal{O}}(\CatI) \times \mathcal{O})}\Mon_{\mathcal{O}}(\CatI) \to \Mon_{\mathcal{O}}(\CatI).
\]
 Here $\Alg'_{\mathcal{Q}/\mathcal{O}}(\Day_{\mathcal{O}}^{\otimes})$ is the full
 subcategory of the domain given fibrewise by $\mathcal{Q}$-algebras,
 and so classifies the subfunctor
 $\Alg_{\mathcal{Q}/\mathcal{O}}(\PSh_{\mathcal{O}}(\blank))$.
\end{proof}

Next, we observe that the cocartesian fibration \cref{eq:algpshfib} has an alternative description:
\begin{observation}
  We have a commutative diagram 
  \[
    \begin{tikzcd}
      \Alg'_{\mathcal{Q}/\mathcal{O}}(\Day_{\mathcal{O}}^{\otimes}) \arrow{r} \pbcorner \arrow{d} & \Alg_{\mathcal{Q}/\mathcal{O}}(\Day_{\mathcal{O}}^{\otimes}) \arrow{r} \pbcorner \arrow{d} & \Alg_{\mathcal{Q}/\xF_{*}}(\RFib^{\times}) \arrow{d} \\
      \Mon_{\mathcal{O}}(\CatI) \arrow{r} & \Alg_{\mathcal{Q}/\mathcal{O}}(\Mon_{\mathcal{O}}(\CatI) \times \mathcal{O}) \arrow{r} & \Alg_{\mathcal{Q}/\xF_{*}}(\CatI^{\times}).
    \end{tikzcd}
  \]
  Here both squares are cartesian, and the composite in the bottom row can be identified with the map $\phi^{*} \colon \Mon_{\mathcal{O}}(\CatI) \to \Mon_{\mathcal{Q}}(\CatI)$ given by composition with $\phi$. It follows that \cref{eq:algpshfib} is equivalent to the pullback along $\phi^{*}$ of the cocartesian fibration
  \begin{equation}
    \label{eq:algrfib}
    \Mon_{\mathcal{Q}}(\RFib) \to \Mon_{\mathcal{Q}}(\CatI),
  \end{equation}
  given by the target projection. Our next goal is to give an alternative description of this fibration.
\end{observation}

\begin{propn}\label{propn:funtorfib}
  Let $\mathcal{K}$ be an \icat{}.
  \begin{enumerate}[(i)]
  \item There is a natural equivalence $\Fun(\mathcal{K}, \RFib) \simeq \RFib \times_{\CatI} \Cart(\mathcal{K}^{\op})$ over the cartesian straigthening equivalence $\Fun(\mathcal{K}, \CatI) \simeq \Cart(\mathcal{K}^{\op})$, where the pullback is via the target projection on $\RFib$ and the source projection on $\Cart(\mathcal{K}^{\op})$.
  \item The cocartesian fibration $\Fun(\mathcal{K}, \RFib) \to \Fun(\mathcal{K}, \CatI)$ classifies the functor $\Phi \mapsto \PSh(\int^{\name{cart}}_{\mathcal{K}^{\op}} \Phi)$, where $\int^{\name{cart}}_{\mathcal{K}^{\op}} \Phi$ denotes the cartesian fibration that corresponds to $\Phi$.
  \end{enumerate}
\end{propn}
\begin{proof}
  By definition, $\RFib$ is a full subcategory of $\Ar(\CatI)$, which lets us identify $\Fun(\mathcal{K}, \RFib)$ with the full subcategory of
  \[ \Fun(\mathcal{K}, \Ar(\CatI)) \simeq \Ar(\Fun(\mathcal{K}, \Cat)) \simeq \Ar(\Cart(\mathcal{K}^{\op}))\]
  spanned by commutative triangles
  \[
    \begin{tikzcd}
      \mathcal{E} \arrow{rr}{\phi} \arrow{dr}[swap]{p} & & \mathcal{F} \arrow{dl}{q} \\
       & \mathcal{K}^{\op}
    \end{tikzcd}
  \]
  where $p$ and $q$ are cartesian fibrations, $\phi$ preserves
  cartesian morphisms, and on each fibre over $k \in \mathcal{K}^{\op}$ the
  functor $ \phi_{k} \colon \mathcal{E}_{k} \to \mathcal{F}_{k}$ is a
  right fibration. We claim that these conditions are equivalent to:
  $q$ is a cartesian fibration and $\phi$ is a right fibration. That
  the first conditions imply that $\phi$ is a right fibration follows
  from \cite[Lemma A.1.8]{cois} (since condition (4) is vacuous when
  the fibres are right fibrations). On the other hand, if $q$ is a
  cartesian fibration and $\phi$ is a right fibration then the
  composite $p$ is a cartesian fibration, and by \cite[Proposition
  2.4.1.3(3)]{ht} and the uniqueness of cartesian lifts, the
  $p$-cartesian morphisms are precisely the $\phi$-cartesian morphisms
  over $q$-cartesian morphisms, so that $\phi$ preserves cartesian
  morphisms. Moreover, since $\phi$ is a right fibration, this means
  that the $p$-cartesian morphisms are \emph{precisely} the morphisms
  that map to $q$-cartesian morphisms in $\mathcal{F}$.

  This description of the $p$-cartesian morphisms implies that for a square
  \[
    \begin{tikzcd}
      \mathcal{E} \arrow{r}{\phi} \arrow{d}[swap]{e} & \mathcal{F} \arrow{d}{f} \\
      \mathcal{E}' \arrow{r}{\phi'} & \mathcal{F}'
    \end{tikzcd}
  \]
  of cartesian fibrations over $\mathcal{K}^{\op}$, where $\phi$ and $\phi'$ are right fibrations, the map $e$ will automatically preserve cartesian morphisms if $f$ does so.
  We can identify $\Ar(\Cart(\mathcal{K}^{\op}))$ with a subcategory of
  \[ \Ar(\Cat_{\infty/\mathcal{K}^{\op}}) \simeq \Ar(\CatI) \times_{\CatI} \Cat_{\infty/\mathcal{K}^{\op}},\]
  and our description of the objects and morphisms in the subcategory corresponding to $\Fun(\mathcal{K}, \RFib)$ shows that this subcategory is precisely
  \[ \RFib \times_{\CatI} \Cart(\mathcal{K}^{\op}),\]
  as required.

  This shows that the fibration $\Fun(\mathcal{K}, \RFib) \to \Fun(\mathcal{K}, \CatI)$ is equivalent to the projection $\RFib \times_{\CatI} \Cart(\mathcal{K}^{\op}) \to \Cart(\mathcal{K}^{\op})$. It is thus the pullback along the domain functor $\Cart(\mathcal{K}^{\op}) \to \CatI$ of the fibration $\RFib \to \CatI$, which corresponds to the (covariant) presheaf functor. Its pullback then classifies the composite
  \[ \Cart(\mathcal{K}^{\op}) \to \CatI \xto{\PSh_{!}} \LCatI,\]
  which gives (ii).
\end{proof}

\begin{cor}\label{cor:monrfibftr}
  For a cartesian pattern $\mathcal{O}$, the cocartesian fibration $\Mon_{\mathcal{O}}(\RFib) \to \Mon_{\mathcal{O}}(\CatI)$ classifies the functor $\mathcal{C}^{\otimes} \mapsto \Mon_{\mathcal{C}^{\op,\otimes}}(\mathcal{S})$.
\end{cor}
\begin{proof}
  From \cref{propn:funtorfib} we know that the pullback of $\Fun(\mathcal{O}, \RFib) \to \Fun(\mathcal{O}, \CatI)$ to $\Mon_{\mathcal{O}}(\CatI)$ classifies the functor
  \[ \mathcal{C}^{\otimes} \mapsto \PSh(\mathcal{C}_{\otimes}) \simeq
    \Fun(\mathcal{C}^{\op,\otimes}, \mathcal{S}),\] where
  $\mathcal{C}_{\otimes} \to \mathcal{O}^{\op}$ is the cartesian
  fibration for the same functor as the cocartesian fibration
  $\mathcal{C}^{\otimes} \to \mathcal{O}$, and we have
  $\mathcal{C}^{\op,\otimes} := (\mathcal{C}_{\otimes})^{\op}$. It remains to show that the restriction to the full subcategory $\Mon_{\mathcal{O}}(\RFib)$ corresponds to the subfunctor given by monoids. This amounts to the claim that a right fibration $\mathcal{E} \to \mathcal{C}_{\otimes}$ corresponds to a $\mathcal{C}^{\op,\otimes}$-monoid \IFF{} the composite $\mathcal{E} \to \mathcal{O}^{\op}$ corresponds to an $\mathcal{O}$-monoid in $\CatI$; this follows from the proof of \cite[Proposition 6.18]{freealg}.
\end{proof}

Combining our results so far, we have shown:
\begin{cor}\label{cor:covtalgpsh=mon}
  Let $\phi \colon \mathcal{Q} \to \mathcal{O}$ be a morphism of cartesian patterns. There is a natural equivalence
  \[ \Alg_{\mathcal{Q}/\mathcal{O}}(\PSh_{\mathcal{O}}(\blank)) \simeq \Mon_{\mathcal{Q} \times_{\mathcal{O}} (\blank)^{\op,\otimes}}(\mathcal{S})\]
  of functors $\Mon_{\mathcal{O}}(\CatI) \to \LCatI$.  \qed
\end{cor}

From \cref{obs:dayotimesfunctor} we see that on the left the functoriality is given for an $\mathcal{O}$-monoidal functor $F \colon \mathcal{C}^{\otimes} \to \mathcal{D}^{\otimes}$ by composition with the $\mathcal{O}$-monoidal functor $F_{!}\colon \PSh_{\mathcal{O}}(\mathcal{C}) \to \PSh_{\mathcal{O}}(\mathcal{D})$, while on the right it is given by left Kan extension along $F$ plus localization to monoids. Both these functors have right adjoints, given respectively by composition with the lax $\mathcal{O}$-monoidal right adjoint $F^{*} \colon \PSh(\mathcal{D}) \to \PSh(\mathcal{C})$ and by composition with
the functor $\mathcal{Q} \times_{\mathcal{O}}\mathcal{C}^{\op,\otimes} \to \mathcal{Q} \times_{\mathcal{O}}\mathcal{D}^{\op,\otimes}$ induced by $F$. (Since $F$ preserves all cocartesian morphisms over $\mathcal{O}$, it corresponds to a functor $\mathcal{C}_{\otimes} \to \mathcal{D}_{\otimes}$ between the dual cartesian fibrations.) Thus our two equivalent cocartesian fibrations over $\Mon_{\mathcal{O}}(\CatI)$ are also cartesian fibrations, and we can identify their contravariant unstraightenings:
\begin{cor}\label{cor:contravtalgpsh=mon}
  Let $\mathcal{O}$ be a cartesian pattern. There is a natural equivalence
  \[ \Alg_{\mathcal{Q}/\mathcal{O}}(\PSh_{\mathcal{O}}(\blank)) \simeq \Mon_{\mathcal{Q} \times_{\mathcal{O}} (\blank)^{\op,\otimes}}(\mathcal{S})\]
  of functors $\Mon_{\mathcal{O}}(\CatI)^{\op} \to \LCatI$.  \qed
\end{cor}

\begin{observation}\label{obs:monrfibcartmor}
  Let $\mathcal{Q}$ be a cartesian pattern. From the proof of \cref{cor:monrfibftr} we know that we can think of an object of $\Mon_{\mathcal{Q}}(\RFib)$ as a right fibration $\mathcal{E} \to \mathcal{C}_{\otimes}$ that corresponds to a $\mathcal{C}_{\otimes}^{\op}$-monoid in $\mathcal{S}$, where $\mathcal{C}_{\otimes} \to \mathcal{Q}^{\op}$ is the cartesian fibration for a $\mathcal{Q}$-monoid in $\CatI$. A morphism to another such object $\mathcal{F} \to \mathcal{D}_{\otimes}$ is then a commutative square
  \[
    \begin{tikzcd}
      \mathcal{E} \arrow{r} \arrow{d} & \mathcal{F} \arrow{d} \\
      \mathcal{C}_{\otimes} \arrow{r} & \mathcal{D}_{\otimes},
    \end{tikzcd}
  \]
  where the bottom horizontal functor preserves cartesian morphisms over $\mathcal{Q}$.
  This morphism is cartesian \IFF{} the square is a pullback.

  If now $\psi \colon \mathcal{Q}' \to \mathcal{Q}$ is a morphism of cartesian patterns, we get a commutative square
  \[
    \begin{tikzcd}
      \Mon_{\mathcal{Q}}(\RFib) \arrow{r}{\psi} \arrow{d} & \Mon_{\mathcal{Q}'}(\RFib) \arrow{d} \\
      \Mon_{\mathcal{Q}}(\Cat) \arrow{r}{\psi} & \Mon_{\mathcal{Q}'}(\RFib).
    \end{tikzcd}
  \]
  In our fibrational picture the horizontal functors are given by pullback along $\psi^{\op}$, and so preserve cartesian morphisms.
\end{observation}

\begin{observation}
  The fibration
  \[ \Mon_{\mathcal{O}}(\CatI) \times_{\Mon_{\mathcal{Q}}(\CatI)}
    \Mon_{\mathcal{Q}}(\RFib) \to \Mon_{\mathcal{O}}(\CatI)\] is
  evidently functorial in $\mathcal{Q}$ and $\mathcal{O}$, and by
  \cref{obs:monrfibcartmor} we get maps that preserve cartesian
  morphisms. The same goes for
  $\Alg'_{\mathcal{Q}/\mathcal{O}}(\Day^{\otimes}_{\mathcal{O}}) \to
  \Mon_{\mathcal{O}}(\CatI)$, so our contravariant identification of
  the corresponding functors is likewise natural.
\end{observation}

Applying this observation to algebroids, we conclude that for an enrichable pattern $\mathcal{O}$, the equivalence
\[ \Alg_{\mathcal{O}_{\bX}/\mathcal{O}}(\PSh_{\mathcal{O}}(\mathcal{C})) \simeq \Mon_{\mathcal{C}^{\op,\otimes}_{\bX}}(\mathcal{S})\]
is natural in both $\bX$ and $\mathcal{C}$. Combining this with the naturality of the equivalence between monoids and Segal spaces from \cref{propn:algdtomonnat}, we get:

\begin{cor}\label{cor:algdpshissegnat}
  The equivalence $\Algd_{\mathcal{O}}(\PSh_{\mathcal{O}}(\mathcal{C})) \simeq \Seg_{\mathcal{C}^{\op,\otimes}}(\mathcal{S})$ is natural in $\mathcal{C}^{\otimes} \in \Mon_{\mathcal{O}}(\CatI)$.
\end{cor}

\begin{remark}
  We expect that the naturality in $\mathcal{O}$-monoidal functors we have considered here can be extended to \emph{oplax} $\mathcal{O}$-monoidal functors. Spelling this out would require some $(\infty,2)$-categorical machinery that we do not wish to introduce here, however. Similarly, we will not discuss the compatibility between the naturality of \cref{cor:algdpshissegnat} and that of \cref{AlgdDayisSeg}.
\end{remark}

\subsection{Left Kan extensions for enrichable patterns}\label{sec:lke}

In the next section we will look at free algebroids and more
general left adjoints on algebroids to restrictions along simple maps
of enrichable patterns. In preparation for this, here we will discuss
left Kan extensions of Segal objects in spaces along maps of
enrichable patterns. In particular, we will give conditions for these
to decompose into well-behaved fibrewise left Kan extensions along the
decomposition from \cref{prop MonSeg}.

We start by looking at a general extendable morphism between
enrichable patterns for which $\mathcal{S}$ is admissible. In
\cref{propn:pbextdecomp} we will see that in this case the left Kan
extension does decompose fibrewise, but in a somewhat unsatisfactory
way.  Here we refer to \cite[Definitions 7.7 and 7.12]{patterns} for
the precise conditions needed for a morphism
$f \colon \mathcal{O} \to \mathcal{P}$ to be \emph{extendable} and for
a complete \icat{} $\mathcal{C}$ to be \emph{$f$-admissible}; the
point of these definitions is that if $f$ is extendable and
$\mathcal{C}$ is $f$-admissible, then left Kan extension along $f$
restricts to a functor
$f_{!}\colon \Seg_{\mathcal{O}}(\mathcal{C}) \to
\Seg_{\mathcal{P}}(\mathcal{C})$.

\begin{notation}
  Suppose $f \colon \mathcal{O} \to \mathcal{P}$ is a simple morphism
  of enrichable patterns such that
  $f_{0} \colon \mathcal{O}_{0} \to \mathcal{P}_{0}$ is extendable and
  $\mathcal{S}$ is $f_{0}$-admissible. Given an object
  $\bX \in \Seg_{\mathcal{O}_0}(\mathcal{S})$, we write
  $\overline{f}_\bX \colon \mathcal{O}_{\bX} \to
  \mathcal{P}_{f_{0,!}\bX}$ for the composite
  \[ \mathcal{O}_{\bX} \xto{\mathcal{O}_{\eta_{\bX}}}
    \mathcal{O}_{f_{0}^{*}f_{0,!}\bX} \xto{f_{f_{0,!}\bX}}
    \mathcal{P}_{f_{0,!}\bX},\] where the first functor is induced by
  the unit map $\eta_{\bX} \colon \bX \to f_{0}^{*}f_{0,!}\bX$.
\end{notation}

\begin{propn}\label{propn:pbextdecomp}
  Suppose $f \colon \mathcal{O} \to \mathcal{P}$ is a simple morphism
between enrichable patterns that is extendable, and assume also that
$\mathcal{S}$ is $f$-admissible. Then left Kan
  extension along $\overline{f}_{\bX}$ restricts to a functor
    \[
      \overline{f}_{\bX,!}\colon \Mon_{\mathcal{O}_{\bX}}(\mathcal{S}) \to
      \Mon_{\mathcal{P}_{f_{0,!}\bX}}(\mathcal{S}),
    \]
    which is left adjoint to the functor 
    \[ \overline{f}_{\bX}^{*} \colon \Mon_{\mathcal{P}_{f_{0,!}\bX}}(\mathcal{S}) \xto{f_{f_{0,!}\bX}^{*}} \Mon_{\mathcal{O}_{f_{0}^{*}f_{0,!}\bX}}(\mathcal{S}) \xto{\mathcal{O}_{\eta_{\bX}}^{*}} \Mon_{\mathcal{O}_{\bX}}(\mathcal{S})\]
    given by restriction along $\overline{f}_{\bX}$.
\end{propn}

Before we prove \cref{propn:pbextdecomp} it is convenient to state several lemmas that will also be useful later.

\begin{lemma}\label{lem:extrestrictto0}
  Suppose $f \colon \mathcal{O} \to \mathcal{P}$ is a morphism of enrichable patterns. Then:
  \begin{enumerate}[(i)]
  \item If $f$ is extendable, so is the restriction $f_0 \colon \mathcal{O}_{0} \to \mathcal{P}_{0}$.
  \item If $\mathcal{S}$ is $f$-admissible, then it is also
    $f_0$-admissible.
  \item If $f_{0}$ is extendable and $\mathcal{S}$ is $f_{0}$-admissible, then the mate transformation
    \[ f_{0,!}i_{0}^{\mathcal{O},*} \to i_{0}^{\mathcal{P},*}f_{!}\]
    is an equivalence of functors
    $\Seg_{\mathcal{O}}(\mathcal{S}) \to
    \Seg_{\mathcal{P}_{0}}(\mathcal{S})$.
  \end{enumerate}
\end{lemma}
\begin{proof}
  We first prove that $f_0$ is extendable when $f$ is. It is clear
  that $f_{0}$ inherits unique lifting of inert morphisms from $f$.
  For condition (2) in \cite[Definition 7.6]{patterns} we have to show
  that the canonical map
  \[ \mathcal{O}_{0/P}^{\act}\to \lim_{E\in
      \mathcal{P}^\el_{0,P/}}\mathcal{O}^{\act}_{0/E}\] is cofinal for
  $P \in \mathcal{P}_{0}$.  As in \cref{rem O_0 act} we can here
  identify $\mathcal{O}_{0/P}^{\act}$ with $\mathcal{O}_{/P}^{\act}$,
  and also $\mathcal{P}^{\el}_{0,P/}$ with $\mathcal{P}^{\el}_{P/}$ by
  \cref{rem O_0 int}, so this also follows from the corresponding
  condition for $f$.  For the same reason, condition (3) is also
  inherited from $f$, as is the admissibility of $\mathcal{S}$ for (ii).

  Identifying $\mathcal{O}^{\act}_{0,/P}$ with
  $\mathcal{O}_{/P}^{\act}$ for $P \in \mathcal{P}_0$ and the formula
  for left Kan extensions then show that the natural transformation in
  (iii) is an equivalence.
\end{proof}

\begin{lemma}\label{lem:fwextble}
  Suppose $f \colon \mathcal{O} \to \mathcal{P}$ is a simple morphism
  of enrichable patterns such that
  $f_{0} \colon \mathcal{O}_{0} \to \mathcal{P}_{0}$ is extendable and
  $\mathcal{S}$ is $f_{0}$-admissible. Then
  $\overline{f}^{\natural}_{\bX} \colon \mathcal{O}^{\natural}_{\bX} \to
  \mathcal{P}^{\natural}_{f_{0,!}\bX}$ is extendable. Moreover, if $\mathcal{S}$ is $f$-admissible then it is also $\overline{f}^{\natural}_{\bX}$-admissible.
\end{lemma}
\begin{proof}
  By construction we have a commutative diagram
  \[
    \begin{tikzcd}
      \mathcal{O}_\bX \arrow[d] \arrow[r, "\overline{f}_\bX"] & \mathcal{P}_{f_{0,!}\bX} \arrow[d] \\
      \mathcal{O} \arrow{r}{f} & \mathcal{P},
    \end{tikzcd}
  \]
  where $f$ is extendable and the vertical maps are Segal fibrations
  (and $\overline{f}_{\bX}$ preserves cocartesian morphisms since
  these are in fact left fibrations). Hence, \cite[Proposition
  9.5]{patterns} implies that the upper horizontal map
  $\overline{f}_\bX$ is extendable and that $\mathcal{S}$ is
  $\overline{f}_{\bX}$-admissible (using \cite[Corollary
  7.17]{patterns}).
\end{proof}

\begin{lemma}\label{lem:actslicewc}
  Suppose $f \colon \mathcal{O} \to \mathcal{P}$ is a simple morphism
  between enrichable patterns such that $f_{0}$ is extendable and
  $\mathcal{S}$ is $f_{0}$-admissible.  Then for every
  $\bX \in \Seg_{\mathcal{O}_{0}}(\mathcal{S})$ and
  $\overline{P} \in \mathcal{P}_{f_{0,!}\bX}$ lying over $\angled{0}$,
 the \icat{} $\mathcal{O}^{\act}_{0,\bX/\bar P}$ is weakly
  contractible (\ie{} its classifying space
  $\|\mathcal{O}^{\act}_{0,\bX/\bar P}\|$ is contractible).
\end{lemma}
\begin{proof}
  Since the inclusion of $\mathcal{O}^{\act}_{0,\bX/\bar P}$ 
  in $\mathcal{O}_{0,\bX/\bar{P}}$ is cofinal, we can also show that
  the latter \icat{} is weakly contractible. But in the commutative
  square
  \[
    \begin{tikzcd}
      \mathcal{O}_{0,\bX} \arrow{d} \arrow{r} & \mathcal{P}_{0,f_{0,!}\bX} \arrow{d} \\
      \mathcal{O}_{0} \arrow{r}{f_{0}} & \mathcal{P}_{0},
    \end{tikzcd}
  \]
  the right vertical map is the left fibration for the left Kan
  extension of the functor corresponding to the left vertical
  map. This square is therefore a component of the unit transformation
  for the left adjoint to pullback of left fibrations along $f_{0}$.
  But coinitial functors and left fibrations form a factorization
  system on $\CatI$, so this left adjoint is also given by factoring
  the composite $\mathcal{O}_{0,\bX} \to \mathcal{P}_{0}$ as a
  coinitial functor followed by a left fibration. The top horizontal
  functor in the square is therefore coinitial, so that the \icat{}
  $\mathcal{O}_{0,\bX/\bar{P}}$ is weakly contractible for all
  $\bar{P} \in \mathcal{P}_{0,f_{0,!}\bX}$, as required.
\end{proof}

\begin{proof}[Proof of \cref{propn:pbextdecomp}]
  We know from \cref{lem:fwextble} that left Kan extension along
  $\overline{f}_{\bX}$ restricts to a functor
  \[ \overline{f}_{\bX,!} \colon \Seg_{\mathcal{O}_{\bX}}(\mathcal{S})
    \to \Seg_{\mathcal{P}_{f_{0,!}\bX}}(\mathcal{S}),\] left adjoint
  to restriction along $\overline{f}_{\bX}$. It then suffices to show
  that this functor restricts further to monoids, \ie{} that for
  $M \in \Mon_{\mathcal{O}_{\bX}}(\mathcal{S})$ the left Kan extension
  $\overline{f}_{\bX,!}M$ is a
  $\mathcal{P}_{f_{0,!}\bX}$-monoid. By \cref{lem:monoidrestrpt}
  it suffices to show that $\overline{f}_{\bX,!}M|_{\mathcal{P}_{0,f_{0,!}\bX}}$ is constant at $*$. 

  Here we have for every $\bar{P} \in \mathcal{P}_{0,f_{0,!}\bX}$ a chain of equivalences
  \[f_{\bX,!}M(\bar{P}) \simeq
    \colim_{\mathcal{O}_{\bX/\bar{P}}^{\act}} M \simeq
    \colim_{\mathcal{O}^{\act}_{0,\bX/\bar{P}}} M \simeq
    \colim_{\mathcal{O}^{\act}_{0,\bX/\bar P}} * \simeq
    \|\mathcal{O}^{\act}_{0,\bX/\bar P}\|,\] where we used the
  equivalence
  $\mathcal{O}_{\bX/\bar P}^{\act} \simeq
  \mathcal{O}^{\act}_{0,\bX/\bar P}$ of \cref{rem O_0 act}.
  It thus
  suffices to show that
  $\|\mathcal{O}^{\act}_{0,\bX/\bar P}\|\simeq *$, which we saw in \cref{lem:actslicewc}. 
\end{proof}

\begin{remark}
  There are two problems with the conditions in \cref{propn:pbextdecomp}:
  \begin{itemize}
  \item Although we 
    showed that the left Kan extension along the map
    \[ \overline{f}^{\flat}_{\bX} \colon \mathcal{O}^{\flat}_{\bX} \to
      \mathcal{P}^{\flat}_{f_{0,!}\bX} \] preserves monoids, we do
    \emph{not} know that it is extendable.  This prevents us from
    applying results from \cite{freealg} to describe the
    corresponding left adjoint on algebras.
  \item On the other hand, the condition that $\mathcal{S}$ is
    $f$-admissible is rather strong, and in practice only holds if
    $\mathcal{P}$ is a cartesian pattern or the colimits we take in the
    left Kan extension are indexed by $\infty$-groupoids. It is
    therefore potentially useful to have other conditions under which
    a left Kan extension preserves Segal objects in spaces.
  \end{itemize}
  It turns out that if we know that the maps
  $\overline{f}^{\flat}_{\bX}$ are extendable, we can use this to
  circumvent the assumption of $f$-admissibility:
\end{remark}

\begin{propn}\label{propn:lkefibrewise}
  Suppose $f \colon \mathcal{O} \to \mathcal{P}$ is a simple map of enrichable patterns such that
  \begin{enumerate}[(i)]
  \item $f_{0} \colon \mathcal{O}_{0} \to \mathcal{P}_{0}$ is extendable and $\mathcal{S}$ is $f_{0}$-admissible.
  \item For every $\bX$ in $\Seg_{\mathcal{O}_{0}}(\mathcal{S})$, the morphism of cartesian patterns $\overline{f}_{\bX}^{\flat} \colon \mathcal{O}^{\flat}_{\bX} \to \mathcal{P}^{\flat}_{f_{0,!}\bX}$ is extendable.
  \end{enumerate}
  Then left Kan extension along $f$ restricts to a functor
  \[ f_{!} \colon \Seg_{\mathcal{O}}(\mathcal{S}) \to \Seg_{\mathcal{P}}(\mathcal{S}).\]
\end{propn}
\begin{proof}
  Suppose $F$ is a Segal $\mathcal{O}$-space, and let $\bX$ be its restriction to $\mathcal{O}_{0}$. Then by \cref{prop MonSeg} there is an $\mathcal{O}_{\bX}$-monoid $M$ such that $F \simeq p_{!} M$ where $p$ is the projection $\mathcal{O}_{\bX} \to \mathcal{O}$. We have a commutative square
  \[
    \begin{tikzcd}
      \mathcal{O}_{\bX} \arrow{r}{\overline{f}_{\bX}} \arrow{d}{p} & \mathcal{P}_{f_{0,!}\bX} \arrow{d}{q} \\
      \mathcal{O} \arrow{r}{f} & \mathcal{P},
    \end{tikzcd}
  \]
  where $q$ is the left fibration for
  $i_{0,*}^{\mathcal{P}}f_{0,!}\bX$, which is a Segal $\mathcal{P}$-space by our first
  assumption. We then have an equivalence
  \[ f_{!}F \simeq q_{!}\overline{f}_{\bX,!} M,\] where
  $\overline{f}_{\bX,!} M$ is a $\mathcal{P}_{f_{0,!}\bX}$-monoid by
  assumption. (Recall that for a morphism $f$ of cartesian patterns,
  $\mathcal{S}$ is always $f$-admissible since it is cartesian closed
  \cite[Corollary 7.18]{patterns}.)  This monoid therefore corresponds to a
  Segal $\mathcal{P}$-space under $q_{!}$, using \cref{prop MonSeg}
  again.
\end{proof}

To be able to apply \cref{propn:lkefibrewise} we need a condition on
$f$ that implies that the maps $\overline{f}^{\flat}_{\bX}$ are
extendable. It turns out that this works provided we can strengthen
the conclusion of \cref{lem:actslicewc}: we need the \icats{}
$\mathcal{O}^{\act}_{0,\bX/\bar P}$ to be \emph{contractible} rather
than merely weakly contractible. This is a rather strong assumption on
$f$, but it seems to be required in order to say anything about left
adjoints on algebroids, which we will discuss in the next section.

The following observation gives a useful starting point:
\begin{propn}\label{propn:getflatext}
  Suppose $f \colon \mathcal{O} \to \mathcal{P}$ is an
  extendable morphism of
  enrichable patterns such that the \icat{} $\mathcal{O}^{\act}_{0/P}$
  is contractible for all $P \in \mathcal{P}^{\el}_{0}$. Then the
  morphism of cartesian patterns
  $f^{\flat} \colon \mathcal{O}^{\flat} \to \mathcal{P}^{\flat}$ is
  extendable.
\end{propn}
\begin{proof}
  The morphism $f^{\flat}$ has unique lifting of inert morphisms since
  this is true for $f$  and the factorization systems are
  the same. Moreover, for $P \in \mathcal{P}$ the functor
  \[ \mathcal{O}^{\act}_{/P}\to \prod_{i} \mathcal{O}^{\act}_{/P_{i}}\]
  factors as
  \[ \mathcal{O}^{\act}_{/P}\to \lim_{E \in \mathcal{P}^{\el}_{P/}}
    \mathcal{O}^{\act}_{/E} \to  \prod_{i} \mathcal{O}^{\act}_{/P_{i}}.\] Here the first functor is cofinal
  since $f$ is extendable, so it suffices to show that the second functor is an
  equivalence. This holds because by assumption
  $\mathcal{O}^{\act}_{/E} \simeq \mathcal{O}^{\act}_{0/E}$ is contractible for all
  $E \in \mathcal{P}^{\el}_{0}$, so that the functor
  $\mathcal{P}^{\el}_{P/} \to \CatI$ taking $E$ to
  $\mathcal{O}^{\act}_{/E}$ is right Kan extended from
  $\mathcal{P}^{\flat,\el}_{P/}$; the limit of this functor is
  therefore the same as the limit over $\mathcal{P}^{\flat,\el}_{P/}$,
  which is what we want to prove. This completes the proof since the
  third condition for $f^{\flat}$ to be extendable is automatic for a
  morphism of cartesian patterns (see \cite[Remark 8.4]{freealg}).
\end{proof}

\begin{cor}\label{cor:fwextblecond}
  Suppose $f \colon \mathcal{O} \to \mathcal{P}$ is an extendable
  simple morphism of enrichable patterns such that $\mathcal{S}$ is
  $f_{0}$-admissible. If for any $\bX$ in
  $\Seg_{\mathcal{O}_{0}}(\mathcal{S})$ and $\overline{E}$ in
  $\mathcal{P}_{f_{0,!}\bX}$ over $E \in \mathcal{P}_{0}^{\el}$ we
  have that $\mathcal{O}_{0,\bX/\overline{E}}^{\act}$ is contractible,
  then $\overline{f}^{\flat}_{\bX}$ is extendable.
\end{cor}
\begin{proof}
  This follows from \cref{propn:getflatext} since we know that
  $\overline{f}^{\natural}_{\bX}$ is extendable by
  \cref{lem:fwextble}.
\end{proof}

\begin{cor}\label{cor:igpdgivesfwext}
  Suppose $f \colon \mathcal{O} \to \mathcal{P}$ is an extendable
  simple morphism of enrichable patterns. If the \icat{}
  $\mathcal{O}^{\act}_{0/E}$ is an \igpd{} for all
  $E \in \mathcal{P}^{\el}_{0}$, then $\mathcal{S}$ is $f_{0}$-admissible and $\overline{f}^{\flat}_{\bX}$ is
  extendable for every $\bX \in \Seg_{\mathcal{O}_{0}}(\mathcal{S})$.
\end{cor}
\begin{proof}
  We first observe that since the \icat{} $\mathcal{O}_{0/E}^{\act}$
  is an \igpd{} for all $E \in \mathcal{P}_{0}$ by assumption, we know
  from \cite[Corollary 7.20]{patterns} that $\mathcal{S}$ is
  $f_{0}$-admissible.
  
  Next, we want to apply \cref{cor:fwextblecond}. For this we need to
  show that for an object $\overline{E}$ in $\mathcal{P}_{f_{0,!}\bX}$
  lying over $E \in \mathcal{P}^{\el}_{0}$, the \icat{}
  $\mathcal{O}^{\act}_{0,\bX/\overline{E}}$ is contractible. We already know
  from \cref{lem:actslicewc} that this \icat{} is weakly contractible.
  To show that it is furthermore contractible, it then suffices to
  observe that it is actually an \igpd{}: it is the source of a left
  fibration over the \icat{} $\mathcal{O}^{\act}_{0/E}$, which is by
  assumption an \igpd{}.
\end{proof}

Putting this together with \cref{propn:lkefibrewise}, we get:
\begin{cor}
    Suppose $f \colon \mathcal{O} \to \mathcal{P}$ is an extendable
  simple morphism of enrichable patterns. If the \icat{} $\mathcal{O}^{\act}_{0/E}$ is an
  \igpd{} for all $E \in \mathcal{P}^{\el}_{0}$, then left Kan extension along $f$ restricts to a functor
  \[ f_{!} \colon \Seg_{\mathcal{O}}(\mathcal{S}) \to
    \Seg_{\mathcal{P}}(\mathcal{S}),\] left adjoint to restriction
  along $f$. \qed
\end{cor}

The following is a useful special case of \cref{cor:igpdgivesfwext}:
\begin{cor}
  Suppose $f \colon \mathcal{O} \to \mathcal{P}$ is an extendable
  simple morphism of enrichable patterns such that
  all maps in $\mathcal{O}_{0}$ are inert. Then for all
  $\bX \in \Seg_{\mathcal{O}_{0}}(\mathcal{S})$, the morphism
  \[ \overline{f}_{\bX}^{\flat} \colon \mathcal{O}^{\flat}_{\bX} \to \mathcal{O}^{\flat}_{f_{0,!}\bX}\]
  is extendable.
\end{cor}
\begin{proof}
  To see this, it suffices to observe that $\mathcal{O}_{0/P}^{\act}$ is an \igpd{} for all $P \in \mathcal{P}_{0}$: the \icat{} $\mathcal{O}_{0}^{\act}$ is one by assumption, and the projection $\mathcal{O}_{0/P}^{\act} \to \mathcal{O}_{0}^{\act}$ is a right fibration.
\end{proof}

As an important further specialization, we have:
\begin{cor}\label{lem:freealgextble}
  Suppose $\mathcal{O}$ is an enrichable pattern that is also
  extendable, and let $j$ denote the inclusion
  $\mathcal{O}^{\xint} \to \mathcal{O}$. Then for all
  $\bX \in \Seg_{\mathcal{O}^{\xint}_{0}}(\mathcal{S})$, the morphism
  \[ \overline{\jmath}_{\bX}^{\flat} \colon
    \mathcal{O}^{\xint,\flat}_{\bX} \to
    \mathcal{O}^{\flat}_{j_{0,!}\bX}\] is extendable. \qed
\end{cor}

\subsection{Free algebroids and other left adjoints}
Our goal in this section is to use the results from \cite{freealg}
together with our discussion of extendability in the preceding
section to give an explicit colimit formula for free algebroids
over an enrichable pattern that is also extendable, as well as for
certain other left adjoints. We start by recalling the following case of a result from \cite{freealg}:
\begin{propn}\label{propn:ifextbleitworks}
  Suppose $f \colon \mathcal{O} \to \mathcal{P}$ is a
  simple morphism between enrichable patterns such that $f_{0}$ is extendable, $\mathcal{S}$ is $f_{0}$-admissible, and
  $\overline{f}_{\bX}^{\flat} \colon \mathcal{O}_{\bX}^{\flat} \to
  \mathcal{P}_{f_{0,!}\bX}^{\flat}$ is extendable for some $\bX \in \Seg_{\mathcal{O}_{0}}(\mathcal{S})$. Then if $\mathcal{V}^{\otimes}$ is a presentably $\mathcal{P}$-monoidal \icat{}, we have an adjunction
  \[ \overline{f}_{\bX,!} \colon \Alg_{\mathcal{O}_{\bX}/\mathcal{P}}(\mathcal{V}) \rightleftarrows \Alg_{\mathcal{P}_{f_{0,!}\bX}/\mathcal{P}}(\mathcal{V}) \cocolon \overline{f}_{\bX}^{*},\]
  where the right adjoint is given by by restriction along $\overline{f}_{\bX}$ and the left adjoint is given at $A\in \Alg_{\mathcal{O}_\bX/\mathcal{P}}(\mathcal{V})$ and $\bar{P} \in \mathcal{P}_{f_{0,!}\bX}$ by
  \[\overline{f}_{\bX,!}A(\overline P)\simeq \colim_{(\overline O, \phi\colon f_\bX\overline O\actto \overline P)\in \mathcal{O}^\act_{\bX/\overline P}} \phi_!A(\overline O).\]
\end{propn}
\begin{proof}
  Under these assumptions, this is just a special case of \cite[Corollary 8.13]{freealg}.
\end{proof}

\begin{remark}
  The assumption that $\mathcal{V}$ is presentably
  $\mathcal{P}$-monoidal can be relaxed to requiring the
  $\mathcal{P}$-monoidal structure to be compatible with the colimits
  that appear (by embedding $\mathcal{V}$ in presheaves in a larger
  universe in the usual way), but we will keep it in order to use only
  results that are explicitly stated in \cite{freealg}.
\end{remark}

To extend \cref{propn:ifextbleitworks} to obtain a left adjoint
between \icats{} of algebroids, we use the following observation:
\begin{propn}\label{propn:leftadjtocartfib}
  Suppose given a commutative square
  \[
    \begin{tikzcd}
      \mathcal{E} \arrow{r}{\overline{R}} \arrow{d}{p} & \mathcal{E}'
      \arrow{d}{p'} \\
      \mathcal{B} \arrow{r}{R} & \mathcal{B}'
    \end{tikzcd}
  \]
  where
  \begin{enumerate}[(i)]
  \item $p$ and $p'$ are cartesian fibration, 
  \item $\overline{R}$ preserves cartesian morphisms,
  \item $R$ has a left adjoint $L \colon \mathcal{B}' \to
    \mathcal{B}$,
  \item for $x \in \mathcal{B}'$ there is a functor $\overline{L}_{x}
    \colon \mathcal{E}'_{x} \to \mathcal{E}_{Lx}$ left adjoint to the
    composite
    \[ \mathcal{E}_{Lx} \xto{\overline{R}_{Lx}} \mathcal{E}'_{RLx}
      \xto{u_{x}^{*}} \mathcal{E}'_{x} \]
    where the second functor is the cartesian transport along the unit
    map $u_{x} \colon x \to RLx$.
  \end{enumerate}
  Then $\overline{R}$ has a left adjoint $\overline{L}$ that lies over
  $L$ and is given
  fibrewise by $\overline{L}_{x}$.
\end{propn}
\begin{proof}
  For $e' \in \mathcal{E}'$ over $x \in \mathcal{B}'$
  the unit of the adjunction $\overline{L}_{x} \dashv
  u_{x}^{*}\overline{R}_{Lx}$ gives a map $e' \to
  u_{x}^{*}\overline{R}_{Lx}\overline{L}_{x}e'$  in
  $\mathcal{E}'_{x}$, \ie{} a map $\overline{u}_{e'} \colon e' \to
  \overline{R}_{Lx}\overline{L}_{x}e'$ in $\mathcal{E}'$ over
  $u_{x}$. From this we get for $e \in \mathcal{E}$ over $b \in \mathcal{B}$ a commutative diagram
  \[
    \begin{tikzcd}
    \Map_{\mathcal{E}}(\overline{L}_{x}e', e) \arrow{r} \arrow{d} &
    \Map_{\mathcal{E}'}(\overline{R}\overline{L}_{x}e', \overline{R}e) \arrow{d} \arrow{r}{\overline{u}_{e'}^{*}} & 
    \Map_{\mathcal{E}'}(e', \overline{R}e) \arrow{d} \\
    \Map_{\mathcal{B}}(Lx, b) \arrow{r} \arrow[bend right]{rr}{\sim} & 
    \Map_{\mathcal{B}'}(RLx, Rb) \arrow{r}{u_{x}^{*}} & \Map_{\mathcal{B}'}(x, Rb),
  \end{tikzcd}
    \]
  and we want to show that the composite in the top row is an
  equivalence. Since the composite in the bottom row is an
  equivalence, to see this it suffices to prove that the map on fibres
  over $\phi \colon Lx \to b$ is an equivalence (as this implies the
  composite square is cartesian). Here we have a commutative diagram
  \[
    \begin{tikzcd}
      \Map_{\mathcal{E}}(\overline{L}_{x}e',e)_{\phi} \arrow{d}{\sim}
      \arrow{r} & \Map_{\mathcal{E}'}(\overline{R}\overline{L}_{x}e',
      \overline{R}e)_{R\phi} \arrow{d}{\sim} \arrow{r}{\overline{u}_{e'}^{*}}
      & 
    \Map_{\mathcal{E}'}(e', \overline{R}e)_{R\phi u_{x}} \arrow{d}{\sim}
      \\
      \Map_{\mathcal{E}_{Lx}}(\overline{L}_{x}e', \phi^{*}e) \arrow{r} \arrow[equals]{d}
      & \Map_{\mathcal{E}'_{RLx}}(\overline{R}\overline{L}_{x}e',
      (R\phi)^{*}\overline{R}e) \arrow{r} \arrow{d}{\sim}
      & \Map_{\mathcal{E}'_{x}}(e', u_{x}^{*}(R\phi)^{*}\overline{R}e)\arrow{d}{\sim}
      \\
      \Map_{\mathcal{E}_{Lx}}(\overline{L}_{x}e', \phi^{*}e) \arrow{r}
      & \Map_{\mathcal{E}'_{RLx}}(\overline{R}_{Lx}\overline{L}_{x}e',
      \overline{R}_{Lx}\phi^{*}e) \arrow{r}
      & \Map_{\mathcal{E}'_{x}}(e', u_{x}^{*}(R\phi)^{*}\overline{R}e),
    \end{tikzcd}
  \]
  where we used that $(R \phi)^{*}\overline{R}e \simeq
  \overline{R}\phi^{*}e$ since $\overline{R}$ preserves cartesian
  morphisms. We can identify the composite map in the bottom row with
  the composite
  \[ \Map_{\mathcal{E}_{Lx}}(\overline{L}_{x}e', \phi^{*}e) \to
    \Map_{\mathcal{E}'_{x}}(u_{x}^{*}\overline{R}_{Lx}\overline{L}_{x}e',
    u_{x}^{*}\overline{R}_{Lx}\phi^{*}e) \to \Map_{\mathcal{E}'_{x}}(e',
    u_{x}^{*}\overline{R}_{Lx}\phi^{*}e) \]
  which is an equivalence since $\overline{L}_{x}$ is left adjoint to $u_{x}^{*}\overline{R}_{Lx}$.
\end{proof}

\begin{corollary}\label{cor:algdleftadj}
    Suppose $f \colon \mathcal{O} \to \mathcal{P}$ is a
    simple morphism between enrichable patterns such that
$f_{0}$ is extendable, $\mathcal{S}$ is $f_{0}$-admissible, and
  $\overline{f}_{\bX}^{\flat} \colon \mathcal{O}_{\bX}^{\flat} \to
  \mathcal{P}_{f_{0,!}\bX}^{\flat}$ is extendable for every $\bX \in \Seg_{\mathcal{O}_{0}}(\mathcal{S})$. If moreover  $\mathcal{V}^{\otimes}$ is a presentably $\mathcal{P}$-monoidal \icat{}, we then have an adjunction
  \[f_!\colon \Algd_{\mathcal{O}/\mathcal{P}}(\xV)\rightleftarrows
    \Algd_{\mathcal{P}}(\xV)\cocolon f^*,\] where $f^{*}$ is the
  functor given at $\bY \in \Seg_{\mathcal{P}_{0}}(\mathcal{S})$ by
  restriction along
  $f_{\bY} \colon \mathcal{O}_{f_{0}^{*}\bY} \to \mathcal{P}_{\bY}$,
  and the left adjoint $f_{!}$ is given at
  $\bX \in \Seg_{\mathcal{O}_{0}}(\mathcal{S})$ by the functor
  $\overline{f}_{\bX,!}$.
\end{corollary}
\begin{proof}
  Apply \cref{propn:leftadjtocartfib} to the commutative square
  \csquare{\Algd_{\mathcal{P}}(\xV)}{\Algd_{\mathcal{O}/\mathcal{P}}(\xV)}{\Seg_{\mathcal{P}_0}(\xS)}{\Seg_{\mathcal{O}_0}(\xS).}{f^*}{}{}{f^*_0}
  (where $f_{0}^{*}$ has a left adjoint since $f_{0}$ is extendable and $\mathcal{S}$ is $f_{0}$-admissible).
\end{proof}

Combining  \cref{cor:algdleftadj} with the results of the previous section, we get the following conditions for an adjunction on algebroids:
\begin{cor}
  Let $f \colon \mathcal{O} \to \mathcal{P}$ be an extendable simple morphism of enrichable patterns, and suppose one of the following (successively stronger) conditions holds:
  \begin{enumerate}[(1)]
  \item $\mathcal{S}$ is $f_{0}$-admissible and for every $\bX$ in
  $\Seg_{\mathcal{O}_{0}}(\mathcal{S})$ and $\overline{E}$ in
  $\mathcal{P}_{f_{0,!}\bX}$ over $E \in \mathcal{P}_{0}^{\el}$, the \icat{} $\mathcal{O}_{0,\bX/\overline{E}}^{\act}$ is contractible.
\item The \icat{} $\mathcal{O}^{\act}_{0/E}$ is an \igpd{} for all
  $E \in \mathcal{P}^{\el}_{0}$.
\item All maps in $\mathcal{O}_{0}$ are inert.
\end{enumerate}
If $\mathcal{V}^{\otimes}$ is a presentably $\mathcal{P}$-monoidal \icat{}, we then have an adjunction
  \[f_!\colon \Algd_{\mathcal{O}/\mathcal{P}}(\xV)\rightleftarrows
    \Algd_{\mathcal{P}}(\xV)\cocolon f^*,\] where $f^{*}$ is the
  functor given at $\bY \in \Seg_{\mathcal{P}_{0}}(\mathcal{S})$ by
  restriction along
  $f_{\bY} \colon \mathcal{O}_{f_{0}^{*}\bY} \to \mathcal{P}_{\bY}$,
  and the left adjoint $f_{!}$ is given at
  $\bX \in \Seg_{\mathcal{O}_{0}}(\mathcal{S})$ by the functor
  $\overline{f}_{\bX,!}$. \qed
\end{cor}

Let us single out the case of free algebras, which will be our concern for the rest of this section:
\begin{cor}\label{cor:freealg}
  Suppose $\mathcal{O}$ is an enrichable pattern that is also extendable, and let $j$ denote the inclusion $\mathcal{O}^{\xint} \to \mathcal{O}$. If  $\mathcal{V}^{\otimes}$ is a presentably $\mathcal{O}$-monoidal \icat{}, then we have an adjunction
  \[j_!\colon \Algd_{\mathcal{O}^{\xint}/\mathcal{O}}(\xV)\rightleftarrows
    \Algd_{\mathcal{O}}(\xV)\cocolon j^*,\] where $j^{*}$ is the
  functor given at $\bY \in \Seg_{\mathcal{O}_{0}}(\mathcal{S})$ by
  restriction along
  $j_{\bY} \colon \mathcal{O}^{\xint}_{j_{0}^{*}\bY} \to \mathcal{O}_{\bY}$,
  and the left adjoint $j_{!}$ is given at
  $\bX \in \Seg_{\mathcal{O}_{0}}(\mathcal{S})$ by the functor
  $\overline{j}_{\bX,!}$. \qed
\end{cor}

\begin{observation}\label{obs:freealgdformula}
  Explicitly, the free $\mathcal{O}$-algebroid on $A$ is given for
  $\tilde{E}$ in $\mathcal{O}_{\bX}^{\flat,\el}$ by
  \[ (j_{\bX,!}A)(\tilde{E}) \simeq
    \colim_{\tilde{O} \in
      (\mathcal{O}^{\xint}_{\bX})^{\act}_{/\tilde{E}}} \phi_{!}A(\tilde{O}). \]
    Note that if $E \in \mathcal{O}^{\flat,\el}$ is the image of $\tilde{E}$ then $(\mathcal{O}^{\xint}_{\bX})^{\act}_{/\tilde{E}}$ is an \igpd{} with a map to $\Act_{\mathcal{O}}(E)$, so we can also write this as an iterated colimit:
    \[ (j_{\bX,!}A)(\tilde{E}) \simeq \colim_{\phi \colon O \actto E \in \Act_{\mathcal{O}}(E)}
    \colim_{\tilde{O} \in (i_{0,*}\bX)(O) \times_{(i_{0,*}j_{0,!}\bX)(E)} \{\tilde{E}\}} \phi_{!}A(\tilde{O}). \]
\end{observation}

\begin{observation}
  Let $\mathcal{O}$ be an enrichable pattern and
  $\mathcal{V}^{\otimes}$ an $\mathcal{O}$-monoidal \icat{}. For
  $\bX \in \Seg_{\mathcal{O}_{0}}(\mathcal{S})$, the \icat{}
  $\Alg_{\mathcal{O}_{\bX}^{\xint}/\mathcal{O}_{\bX}}(\mathcal{V}^{\otimes})$
  is naturally equivalent to
  $\Fun_{/\mathcal{O}}(\mathcal{O}^{\el}_{\bX},
  \mathcal{V}^{\otimes})$ by \cite[Lemma 4.14]{freealg}. We can
  therefore identify
  $\Algd_{\mathcal{O}^{\xint}/\mathcal{O}}(\mathcal{V})$ with the
  cocartesian fibration for the corresponding functor.
\end{observation}

We now check that the free algebroid adjunction is monadic, and use this to get a criterion for a morphism between extendable enricable patterns to give an equivalence on \icats{} of algebroids.

\begin{propn}\label{propn:freealgdmonadic}
  Suppose $\mathcal{O}$ is an extendable enrichable pattern and $\mathcal{V}$ is a presentably $\mathcal{O}$-monoidal \icat{}. Then
  \[ j^{*} \colon \Algd_{\mathcal{O}}(\mathcal{V}) \to \Algd_{\mathcal{O}^{\xint}}(\mathcal{V})\]
  is a monadic right adjoint.
\end{propn}
\begin{proof}
  We have already established that $j^{*}$ is a right adjoint, so it
  suffices to show that it satisfies the conditions of the monadicity
  theorem. Since $j \colon \mathcal{O}^{\xint} \to \mathcal{O}$ is
  essentially surjective, it is clear that $j^{*}$ is conservative, so
  it remains to show that colimits of $j^{*}$-split simplicial
  diagrams exist and $j^{*}$ preserves them. Since $\mathcal{V}$ is by
  assumption presentably $\mathcal{O}$-monoidal, by
  \cref{prop:algdSloc} and \cref{rem Vk} we can describe
  $\Algd_{\mathcal{O}}(\mathcal{V})$ as
  $\Seg^{\mathbb{S}}_{\mathcal{C}^{\op,\otimes}}(\mathcal{S})$ for
  some small $\mathcal{O}$-monoidal \icat{} $\mathcal{C}$ and a set $\mathbb{S}$ compatible with the $\mathcal{O}$-monoidal structure. Note that both the Segal and $\mathbb{S}$-local conditions are detected on $(\mathcal{C}^{\op,\otimes})^{\xint}$, so that we have a pullback square
  \[
    \begin{tikzcd}
      \Seg^{\mathbb{S}}_{\mathcal{C}^{\op,\otimes}}(\mathcal{S}) \arrow{d}{j^{*}} \arrow[hookrightarrow]{r} & \Fun(\mathcal{C}^{\op,\otimes}, \mathcal{S}) \arrow{d}{j^{*}} \\
      \Seg^{\mathbb{S}}_{\mathcal{C}^{\op,\otimes,\xint}}(\mathcal{S})\arrow[hookrightarrow]{r} & \Fun(\mathcal{C}^{\op,\otimes,\xint}, \mathcal{S}).
    \end{tikzcd}
  \]
  Given a simplicial diagram in
  $\Seg^{\mathbb{S}}_{\mathcal{C}^{\op,\otimes}}(\mathcal{S})$, this
  has a colimit in $\Fun(\mathcal{C}^{\op,\otimes}, \mathcal{S})$ that
  is preserved by $j^{*}$. But if the original diagram is
  $j^{*}$-split, then (since split simplicial diagrams are always
  colimit diagrams) the colimit in
  $\Fun(\mathcal{C}^{\op,\otimes,\xint}, \mathcal{S})$ must lie in
  $\Seg^{\mathbb{S}}_{\mathcal{C}^{\op,\otimes,\xint}}(\mathcal{S})$. This
  means that the colimit we computed in
  $\Fun(\mathcal{C}^{\op,\otimes}, \mathcal{S})$ actually lies in the
  full subcategory
  $\Seg^{\mathbb{S}}_{\mathcal{C}^{\op,\otimes}}(\mathcal{S})$, and so our original simplicial diagram indeed has a colimit that is preserved by $j^{*}$.
\end{proof}

\begin{propn}
  Suppose $\mathcal{O}$ and $\mathcal{P}$ are extendable enrichable patterns and $\phi \colon \mathcal{O} \to \mathcal{P}$ is a simple morphism  of enrichable patterns such that:
  \begin{enumerate}[(1)]
  \item The restriction $\phi^{\xint} \colon \mathcal{O}^{\xint} \to \mathcal{P}^{\xint}$ is also simple.
  \item $\Act_{\mathcal{O}}(O) \to \Act_{\mathcal{P}}(\phi O)$ is an equivalence for all $O \in \mathcal{O}$ (or equivalently just $O \in \mathcal{O}^{\el}$),
  \item $\mathcal{O}^{\el}_{O/} \to \mathcal{P}^{\el}_{\phi O/}$ is coinitial for all $O \in \mathcal{O}$.
  \end{enumerate}
  If $\mathcal{V}$ is a presentably $\mathcal{P}$-monoidal \icat{} such that the restriction functor \[
    \Algd_{\mathcal{P}^{\xint}/\mathcal{P}}(\mathcal{V}) \to \Algd_{\mathcal{O}^{\xint}/\mathcal{P}}(\mathcal{V})\]
  is fully faithful, then the induced functor
  \[\phi^{*} \colon \Algd_{\mathcal{P}}(\mathcal{V}) \to \Algd_{\mathcal{O}}(\mathcal{V})\]
  is fully faithful, with image precisely those $\mathcal{O}$-algebroids whose restriction to $\mathcal{O}^{\xint}$ lies in the image of $\Algd_{\mathcal{P}^{\xint}/\mathcal{P}}(\mathcal{V})$.
\end{propn}
\begin{proof}
  We apply \cite[Proposition 5.3.5]{AnalMnd} to the commutative square
  \[
    \begin{tikzcd}
      \Algd_{\mathcal{P}}(\mathcal{V}) \arrow{r}{\phi^{*}} \arrow{d} & \Algd_{\mathcal{O}}(\mathcal{V}) \arrow{d} \\
      \Algd_{\mathcal{P}^{\xint}/\mathcal{P}}(\mathcal{V}) \arrow{r}{\phi^{\xint,*}} & \Algd_{\mathcal{O}^{\xint}/\mathcal{P}}(\mathcal{V}).
    \end{tikzcd}
  \]
  Here the bottom horizontal map is fully faithful by assumption, and
  the vertical maps are monadic right adjoints by
  \cref{propn:freealgdmonadic}. It therefore only remains to check
  that the mate transformation
  $j_{\mathcal{O},!}\phi^{\xint,*} \to \phi^{*}j_{\mathcal{P},!}$ is an
  equivalence. Using the description of the free algebroids from \cref{obs:freealgdformula}, we see that this will follow if we can show that for every $\bX \in \Seg_{\mathcal{P}_{0}^{\xint}}(\mathcal{S})$ and $\tilde{E} \in \mathcal{O}^{\flat,\el}_{j_{0,!}\phi_{0}^{*}\bX}$, the induced functor
  \[ (\mathcal{O}^{\xint}_{\phi_{0}^{*}\bX})^{\act}_{/\tilde{E}} \to (\mathcal{P}^{\xint}_{\bX})^{\act}_{/\phi(\tilde{E})} \]
  is cofinal (\ie{} an equivalence, since it is a map of \igpds{}). To this end, we first observe that our functor lives in a commutative square
  \[
    \begin{tikzcd}
      (\mathcal{O}^{\xint}_{\phi_{0}^{*}\bX})^{\act}_{/\tilde{E}} \arrow{r} \arrow{d} & (\mathcal{P}^{\xint}_{\bX})^{\act}_{/\phi(\tilde{E})} \arrow{d} \\
      \Act_{\mathcal{O}}(E) \arrow{r}{\sim} & \Act_{\mathcal{P}}(\phi(E)),
    \end{tikzcd}
  \]
  where the bottom horizontal map is an equivalence by assumption. The
  top horizontal map is then an equivalence \IFF{} this square is a
  pullback, and since this is a square of \igpds{}, to see this it suffices to check we get equivalences on all fibres. The map on fibres over an active map $\alpha \colon O \actto E$ is also a map on fibres in the square
  \[
    \begin{tikzcd}
      (i_{0,*}^{\mathcal{O}}\phi_{0}^{\xint,*}\bX)(O) \arrow{rr}{\sim} \arrow{d} & &  (\phi^{\xint,*}i_{0,*}^{\mathcal{P}}\bX)(O) \arrow{d} \\
      (i_{0,*}^{\mathcal{O}} j_{0,!}^{\mathcal{O}}\phi_{0}^{\xint,*}\bX)(E)  \arrow{r} &
      (i_{0,*}^{\mathcal{O}} \phi_{0}^{*} j_{0,!}^{\mathcal{P}}\bX)(E) \arrow{r}{\sim} & (\phi^{*}i_{0,*}^{\mathcal{P}} j_{0,!}^{\mathcal{P}}\bX)(E).
    \end{tikzcd}
  \]
  Here the top horizontal map is an equivalence since $\phi^{\xint}$ is simple, while in the bottom row the second map is an equivalence since $\phi$ is simple. The first map in the bottom row is also an equivalence, which follows by unwinding the formula for $j_{0,!}$ and noting that by assumption $\Act_{\mathcal{O}_{0}}(Z) \to \Act_{\mathcal{P}_{0}}(\phi Z)$ is an equivalence for all $Z \in \mathcal{O}_{0}$ by assumption (using \cref{rem O_0 act}). It follows that we get an equivalence on fibres in this square, which completes the proof.
\end{proof}

\begin{cor}
  Suppose $\mathcal{O}$ and $\mathcal{P}$ are extendable enrichable patterns and $\phi \colon \mathcal{O} \to \mathcal{P}$ is a simple morphism of enrichable patterns such that
  \begin{enumerate}[(1)]
  \item $\phi^{\xint}$ is also simple.
  \item $\Act_{\mathcal{O}}(O) \to \Act_{\mathcal{P}}(\phi O)$ is an equivalence for all $O \in \mathcal{O}$ (or equivalently just $O \in \mathcal{O}^{\el}$),
  \item $\mathcal{O}^{\el}_{O/} \to \mathcal{P}^{\el}_{\phi O/}$ is coinitial for all $O \in \mathcal{O}$,
  \item $\mathcal{O}^{\el} \to \mathcal{P}^{\el}$ is an equivalence.
  \end{enumerate}
  Then the induced functor
    \[\phi^{*} \colon \Algd_{\mathcal{P}}(\mathcal{V}) \to \Algd_{\mathcal{O}}(\mathcal{V})\]
    is an equivalence for every presentably $\mathcal{P}$-monoidal
    \icat{} $\mathcal{V}$. \qed
\end{cor}

\section{Examples of enrichment}\label{sec:ex}
  
\subsection{Enriched $\infty$-categories}\label{subsec enriched cat}

As our first example, we will see that enriched \icats{} in the sense of \cite{enriched} fit into the framework we introduced above.

\begin{defn}
  Let $\simp$ be the simplex category, with objects the ordered sets
  $[n] := \{0,\ldots,n\}$. This has a factorization system where a
  morphism $\phi \colon [n] \to [m]$ is \emph{inert} if $\phi$ is a
  subinterval inclusion, \ie{} $\phi(i) = \phi(0)+i$ for
  $i = 0,\ldots,n$, and \emph{active} if $\phi$ is boundary-preserving,
  \ie{} $\phi(0) = 0$ and $\phi(n) = m$. This gives an inert--active
  factorization system on $\Dop$, and we write $\simp^{\op,\natural}$
  for the algebraic pattern given by this, with the elementary objects
  being $[0]$ and $[1]$.
\end{defn}

\begin{remark}
  Segal $\simp^{\op,\natural}$-spaces are precisely Rezk's Segal spaces, first introduced in \cite{Rezk} as a model for \icats{}.
\end{remark}

\begin{observation}
The pattern $\simp^{\op,\natural}$ is enrichable when equipped
with the functor $|\blank|\colon \Dop \to \xF_{*}$ such that 
\begin{itemize}
	\item it takes an object $[n]$ to $|[n]| := \angled{n}$ and
	\item a morphism in $\Dop$ corresponding to $\phi\colon [n]\to [m]$ in $\simp$ is sent to the map $|\phi| \colon \angled{m} \to \angled{n}$ defined by $|\phi|(i) = j$ if
	$\phi(j-1)<i \leq \phi(j)$ for some $j$, and $|\phi|(i)= 0$ otherwise.
      \end{itemize}
      Here $\Dop_{0} = \{[0]\}$ and
$\Seg_{\Dop_{0}}(\mathcal{S}) \simeq \mathcal{S}$.
\end{observation}

For $X \in \mathcal{S}$ the Segal $\Dop$-fibration $\Dop_{X} \to \Dop$
corresponds to the functor $\Dop \to \mathcal{S}$ obtained by right
Kan extension from $\{[0]\}$; this takes $[n]$ to $X^{\times (n+1)}$,
with face maps given by projections and degeneracies given by
diagonals. We thus recover precisely the object denoted $\Dop_{X}$ in
\cite{enriched}, so that a $\Dop$-algebroid with underlying space $X$
in a $\Dop$-monoidal \icat{} is the same thing as a \emph{categorical
  algebra} with space of objects $X$ in the terminology of
\cite{enriched}.

To see why a $\Dop$-algebroid describes (the algebraic structure of) an enriched \icat{}, first note that we can think of the object $[n]$ in
$\simp^{\op}$ as the category
\[ \bullet \to \bullet \to \cdots \to \bullet\]
with $n+1$ objects. An object of $\Dop_{X}$ over $[n]$ can then be though of as this category
together with  a labelling of its objects by elements in the space $X$,
\ie
\begin{equation}
  \label{eq:enrcatob}
  {x_{0}} \to {x_{1}} \to \cdots \to
  {x_{n}}.
\end{equation}
If $\mathcal{C}$ is a $\Dop$-monoidal \icat{} then a $\Dop$-algebroid
$A$ in $\mathcal{C}$ assigns to
$(x,x') \in \Dop_{X,[1]} \simeq X \times X$ an object
$A(x,x') \in \mathcal{C}$, which describes the enriched morphisms from
$x$ to $x'$. The degeneracy $x \mapsto (x,x)$ gives an identity map
$\bbone \to A(x,x)$, and the inner face map $(x,y,z) \mapsto (x,z)$
induces a composition map $A(x,y) \otimes A(y,z) \to A(x,z)$. To a
general object \cref{eq:enrcatob} we assign the list
$(A(x_{0},x_{1}),\ldots,A(x_{n-1},x_{n}))$, and the morphisms of
$\Dop$ give a homotopy-coherently associative and unital composition
operation.

The pattern $\simp^{\op,\natural}$ is extendable (see \cite[Example 8.14]{patterns}). We therefore get a formula for free $\Dop$-algebroids from \cref{cor:freealg}:
\begin{propn}
  Let $\mathcal{V}$ be a presentably $\Dop$-monoidal \icat{}. For a space $X$, the free $\Dop$-algebroid on $\Phi \colon X \times X \to \mathcal{V}$ is given by
  \[ j_{!}\Phi(x,y) \simeq \coprod_{n=0}^{\infty} \colim_{(x_{0},\ldots,x_{n}) \in (X^{n+1})_{x,y}} \Phi(x_{0},x_{1}) \otimes \cdots \otimes \Phi(x_{n-1},x_{n}).\]
\end{propn}
Here the notation $(X^{n+1})_{x,y}$ means the fibre at $(x,y)$ of the first and last projections to $X$. If $n > 0$ this is equivalent to $X^{n-1}$, while for $n = 0$ it is the space $X(x,y)$ of paths from $x$ to $y$ in $X$; thus the contribution from $n = 0$ is $\colim_{X(x,y)} \bbone$.

\subsection{Enriched $\infty$-operads}\label{sec enropd}
Next, we will see that algebroids over $\bbO$ give one of the models for enriched \iopds{} we studied in \cite{ChuHaugseng}.

\begin{defn}
Let $\bbO$ be the dendroidal category of trees, \ie{} finite directed and
simply connected graphs, as introduced by Moerdijk--Weiss in \cite{MoerdijkWeiss}.
By \cite[Example 3.7]{patterns} its opposite category $\bbOop$ has an algebraic
pattern structure defined as follows:
\begin{itemize}
\item A morphism is inert or active if it is a subgraph inclusion or a boundary preserving map, respectively.
\item An object is elementary if it is either the plain edge $\eta$ or a corolla, i.e. a tree with one vertex.
\end{itemize}
\end{defn}

\begin{remark}
  Segal $\bbO^{\op,\natural}$-spaces are precisely the dendroidal Segal spaces introduced by Cisinski and Moerdijk \cite{CisinskiMoerdijk2} as a model for \iopds{}.
\end{remark}

\begin{observation}
By \cite[Definition 4.1.16]{ChuHaugseng} there is functor
$\name V_\bbO\colon \bbOop\to \mathbb F_*^\natural$ which takes a tree
to its set of vertices with a disjoint base point.  It
then follows from the definition that every inert map
$\name V_\bbO(T)\rightarrowtail \langle 1\rangle$ has a unique inert
lift $T\to C$, hence $\bbOop$ is an enrichable pattern. Moreover,        
as
$\bbOop_0=\{\eta\}$ we have $\Seg_{\bbOop_0}(\xS)\simeq\xS$. 
\end{observation}

For every $X \in \xS$, an object in $\bbOop_X$ can be regarded as a
tree whose edges are labeled by elements of the space $X$.  Thus
$\bbOop_{X} \to \bbOop$ in the sense of \cref{def OX} means the same
thing as in \cite{ChuHaugseng}. In particular, if $\mathcal{C}$ is a
symmetric monoidal \icat{} then by \cite[\S 4]{ChuHaugseng} an
$\bbOop$-algebroid $A$ in $\mathcal{C}$ lying over $X$ models a
$\mathcal{C}$-enriched $\infty$-operad with $X$ as its space of
objects: The algebroid $A$ takes a corolla in $\bbOop_{X}$ whose
leaves and root are labeled by $x_1, \ldots, x_n$ and $y$,
respectively, to an object in $\mathcal{C}$ which can be viewed as the
object of multimorphisms $\xMap(x_1,\ldots, x_n;y)$ for an \iopd{}
enriched in $\mathcal{C}$.  The identities and the operadic
compositions are then given by the images of active morphisms in
$\bbOop_X$ under $A$.

The pattern $\bbO^{\op,\natural}$ is extendable (the conditions are checked in \cite[\S 5.3]{AnalMnd}). We therefore get a formula for free $\bbOop$-algebroids from \cref{cor:freealg}.

\begin{variant}
  By \cite[Example 3.8]{patterns}, for every perfect operator category $\Phi$
  as defined by Barwick in \cite{bar} there is algebraic pattern
  $\simp^{\op,\natural}_\Phi$ whose objects are finite sequences of
  morphisms in $\Phi$. For $\Phi=*$ we  have
  $\simp^{\op, \natural}_*\simeq \simp^{\op, \natural}$. If $\Phi$ is
  the category $\xF$ of finite sets, then $\simp^{\op,\natural}_\xF$
  is a category of finite ``level trees'' with morphisms which
  respects the level structure. We also obtain a planar version by
  choosing $\Phi$ to be the category $\mathbb{O}$ of ordered sets. For
  every $\Phi$, the unique operator morphism $\Phi\to \xF$ induces a
  morphism of algebraic patterns
  $\simp^{\op,\natural}_\Phi\to \simp^{\op,\natural}_\xF$. It is
  easily checked that the composite of this functor with the functor
  $\simp^{\op,\natural}_\xF\to \xF_*^\natural$ sending a level tree
  to its set of vertices with a disjoint base point exhibits
  $\simp^{\op,\natural}_\Phi$ as an enrichable pattern. It then
  follows from \ref{subsec enriched cat} and \cite{ChuHaugseng} that
  for $\Phi$ being $*, \xF, \mathbb{O}$, the $\simp^\op_{\Phi}$-algebroids in a (symmetric) monoidal
  \icat{} $\mathcal{C}$ describe $\mathcal{C}$-enriched \icats{}, $\mathcal{C}$-enriched
  $\infty$-operads, and $\mathcal{C}$-enriched planar $\infty$-operads,
  respectively. Moreover, \cite[Proposition 4.3.2 and Theorem
  4.4.3]{ChuHaugseng} show that $\simp^\op_{\xF}$-algebroids in $\mathcal{C}$
  are equivalent to $\bbO^\op$-algebroids.
\end{variant}

\subsection{Enriched $\infty$-properads}
We now briefly observe that our definition of enriched structures also encompasses the enriched $\infty$-properads studied by Chu and Hackney in \cite{iprpd}.

\begin{defn}
Let $\bbGamma$ denote the category of acyclic connected finite directed graphs
defined by Hackney, Robertson, and Yau in \cite{HRYProperad}. As in \cite[Example 3.9]{patterns} its opposite category $\bbGamma^\op$ admits an algebraic pattern structure $\bbGamma^{\op,\natural}$ such that:
\begin{itemize}
	\item a morphism is inert or active if it is given by a subgraph inclusion or is boundary-preserving, respectively,
	\item an object is elementary if it has at most one vertex, \ie{} it is either the plain edge $\eta$ or a corolla $C_{p,q}$ with $p$ incoming and $q$ outgoing edges.
\end{itemize}
\end{defn}

\begin{observation}
  By \cite[Theorem A.1]{iprpd} and \cite[Proposition 2.2.23]{iprpd},
  there is a morphism of algebraic patterns
  $\name{V}_{\bbGamma^\op}\colon \bbGamma^\op\to \xF_*$ that takes a
  graph to its set of vertices with a disjoint base point. By
  construction, every inert map
  $\name{V}_{\bbGamma^\op}(G)\to \angled{1}$ has a unique lift in
  $\bbGamma^\op$, so we see that $\bbGamma^\op$ is enrichable. As in
  the previous examples we have $\bbGamma^{\op}_{0} \simeq \{\eta\}$
  and $\Seg_{\bbGamma^\op_0}(\xS)\simeq \xS$.
\end{observation}

An object of $\bbGamma^{\op}_{X}$ can then be identified with a graph
whose edges are labelled with points of the space $X$. We thus recover
the same indexing objects used in \cite{iprpd}, and a
$\bbGamma^\op$-algebroid $A$ in a symmetric monoidal \icat{}
$\mathcal{C}$ with underlying space $X$ describes a
$\mathcal{C}$-enriched $\infty$-properad with $X$ as its space of
objects.

\begin{variant}
  By restricting to certain full subcategories of $\bbGamma^\op$ we
  similarly obtain notions of enriched $\infty$-dioperads and enriched
  output $\infty$-properads as algebroids; see \cite{iprpd} for more
  details.
\end{variant}

\subsection{Cartesian products and enriched $n$-fold $\infty$-categories}\label{subsec:cartprod}
As our first new example of an enriched structure, we consider enriched $n$-fold \icats{}. Here by an \emph{$n$-fold \icat{}} we mean an \icat{} internal to $(n-1)$-fold \icats{}, \ie{} \icats{} for $n=1$, double \icats{} for $n = 2$, and so forth. These can be described using the cartesian product $\simp^{n,\op,\natural}$, and we start by discussing products of enrichable patterns in general, using the smash product on $\xFs$:

\begin{defn}
  We write
  \[ (\blank)\wedge(\blank)\colon \xFs \times \xFs \to \xFs \]
  for the smash product of pointed finite sets, given by
  \[ (S,s) \wedge (T,t) \cong \left(S \times T\right) / \left(\{s\}\times T \cup S \times \{t\}\right). \]
  In terms of our preferred skeleton, this can be formulated as:
  \begin{enumerate}
    \item $\langle m\rangle \wedge \langle n\rangle = \langle mn\rangle$,
    \item for morphisms $f\colon \langle m\rangle\to \langle m'\rangle, g\colon \langle n\rangle\to \langle n'\rangle,$ the map $f\wedge g \colon \angled{mn} \to \angled{m' n'}$ is given by
    \begin{equation*}
    f\wedge g(n(i-1)+j) = \begin{cases}
  \langle 0\rangle&\text{if $f(i)=\langle 0\rangle$ or $g(j)=\langle 0\rangle$}\\
    n'(f(i)-1) +g(j), &\text{otherwise.}
    \end{cases}
    \end{equation*}
  \end{enumerate}
\end{defn}

\begin{observation}
  Recall from \cite[Corollary 5.5]{patterns} that the \icat{}
  $\AlgPatt$ has cartesian products: if $\mathcal{O}$ and
  $\mathcal{P}$ are algebraic patterns, then their cartesian product
  is given by the product \icat{} $\mathcal{O} \times \mathcal{P}$
  equipped with the factorization system where the inert and active
  morphisms are those that project to inert and active morphisms in
  both $\mathcal{O}$ and $\mathcal{P}$, and the \icat{} of elementary
  objects is $\mathcal{O}^{\el} \times \mathcal{P}^{\el}$. Then it is
  easy to see that the smash product is a morphism of patterns
  \[ (\blank)\wedge(\blank) \colon \xFn \times \xFn \to \xFn.\]
  Moreover, as a morphism $\xFs^{\flat} \times \xFs^{\flat} \to \xFs^{\flat}$ the smash product is an iso-Segal morphism: it gives an isomorphism
  \[ (\xFs^{\flat,\el})_{\angled{n}/} \times (\xFs^{\flat,\el})_{\angled{m}/} \isoto (\xFs^{\flat,\el})_{\angled{nm}/} \]
  for all $n,m$.
\end{observation}

\begin{lemma}\label{lem:cartprodenrble}
  Suppose $\mathcal{O}$ and $\mathcal{P}$ are enrichable patterns. Then the map
  \[ |\blank|_{\mathcal{O}} \wedge |\blank|_{\mathcal{P}} \colon \mathcal{O} \times \mathcal{P} \to \xFn\]
  makes $\mathcal{O} \times \mathcal{P}$ an enrichable pattern, and $(\mathcal{O} \times \mathcal{P})^{\flat} \simeq \mathcal{O}^{\flat} \times \mathcal{P}^{\flat}$.
\end{lemma}
\begin{proof}
  We must show that $(\mathcal{O} \times \mathcal{P})^{\flat}$ is a
  cartesian pattern. It is immediate from the definition that
  $\angled{n} \wedge \angled {m} \cong \angled{1}$ implies
  $n = m = 1$, so the elementary objects of
  $\mathcal{O} \times \mathcal{P}$ over $\angled{1}$ are precisely
  those of $\mathcal{O}^{\flat} \times \mathcal{P}^{\flat}$; in particular, this product is precisely $(\mathcal{O} \times \mathcal{P})^{\flat}$.

  Given $O \in \mathcal{O}$ over $\angled{n}$ and $P \in \mathcal{P}$ over $\angled{m}$, we then have
  \[ (\mathcal{O} \times \mathcal{P})^{\flat,\el}_{O \times P/} \simeq \mathcal{O}^{\flat,\el}_{O/} \times \mathcal{P}^{\flat,\el}_{P/} \isoto
    (\xFs^{\flat,\el})_{\angled{n}/} \times (\xFs^{\flat,\el})_{\angled{m}/} \isoto (\xFs^{\flat,\el})_{\angled{nm}/},\]
  so that $(\mathcal{O} \times \mathcal{P})^{\flat}$ is indeed cartesian.
\end{proof}

\begin{remark}
  The smash product is a symmetric monoidal structure on $\xFs$, and
  it is easy to check that this lifts to algebraic patterns, making
  $\xFn$ a commutative algebra in $\AlgPatt$. We then have by
  \cite[Theorem 2.2.2.4]{ha} an induced symmetric monoidal structure
  on $\AlgPatt_{/\xFn}$ where the tensor product of objects
  $\mathcal{O},\mathcal{P} \to \xFn$ is given by the composite
  \[ \mathcal{O} \times \mathcal{P} \to \xFn \times \xFn \xto{\wedge}
    \xFn.\] By \cref{lem:cartprodenrble} we have shown that this
  symmetric monoidal structure restricts to one on the full
  subcategory of enrichable patterns. (In particular, the product of enrichable patterns is associative and symmetric, with unit the pattern $\{\angled{0}\}$.)
\end{remark}

We define $\simp^{n,\op,\natural}$ to be the cartesian product $(\simp^{\op,\natural})^{\times n}$, viewed as an enrichable pattern via the iterated smash product. We can think of an object $([m_1], \ldots, [m_n])$ in $\simp^{\op, n,\natural}$
as a diagram of $n$-dimensional cubes which has length $m_k$
in the $k$-th dimension; this object is elementary if
$0\leq m_k\leq 1$ for all $k$.

The category $\simp^{n,\op,\natural}_0$ is the full subcategory
spanned by objects $([m_1], \ldots, [m_n])$ such that $m_k=0$ for some
$k$, with elementary objects those where each $m_{i}$ is $0$ or $1$,
with $0$ occurring at least once. We can think of these elementary
objects as the vertices of the unit $n$-cube except for
$(1,\ldots,1)$, or more appropriately as its lower-dimensional faces
that contain the origin but not $(1,\ldots,1)$. For example, for $n = 2$ a functor $\Phi \colon \simp^{n,\op,\natural,\el}_{0} \to \mathcal{S}$ specifies a space $\Phi([0],[0])$ of objects and spaces $\Phi([1],[0])$ and $\Phi([0],[1])$ of two different types (``vertical and horizontal'') of 1-morphisms, while for $n = 3$ we get
\begin{itemize}
\item a space $\Phi([0],[0],[0])$ of objects,
\item 3 spaces $\Phi([1],[0],[0])$, $\Phi([0],[1],[0])$, $\Phi([0],[0],[1])$ of different types of 1-morphisms,
\item 3 spaces $\Phi([0],[1],[1])$, $\Phi([1],[0],[1])$, $\Phi([0],[1],[1])$ of different types of squares. 
\end{itemize}

By \cite[Example 3.4]{patterns},
Segal $\simp^{n,\op,\natural}$-spaces describe  $n$-fold
$\infty$-categories, \ie{} \icats{} internal to \ldots{} internal to
\icats{}, in the sense that
\[ \Seg_{\simp^{n,\op,\natural}}(\mathcal{S}) \simeq \Seg_{\simp^{\op,\natural}}(\Seg_{\simp^{n-1,\op,\natural}}(\mathcal{S})).\]
A Segal $\simp^{n,\op,\natural}_{0}$-space can similarly be viewed as a compatible collection of $n$ different $(n-1)$-fold \icats{} (where all have the same space of objects, etc.).

Given $\bX \in \Seg_{\simp^{n,\op,\natural}_{0}}(\mathcal{S})$, an
object of $\simp^{n,\op}_{\bX}$ can then be viewed as a diagram of
$n$-cubes each of whose $i$-dimensional faces for $0 \leq i < n$ has
been labelled by a point in the space obtained by evaluating $\bX$ at
the appropriate elementary object.

Suppose $\mathcal{C}$ is a $\Dnop$-monoidal \icat{} (which is
equivalent to an $E_{n}$-monoidal \icat{}, by the additivity theorem)
and $\bX\in \Seg_{\simp^{\op, n}_0}(\xS)$. Then a
$\simp^{\op, n}$-algebroid $A$ over $\bX$ in $\mathcal{C}$ models an
$n$-fold \icat{} whose $n$-cubes are enriched in $\mathcal{C}$. All
the lower dimensional structure in $A$ such as objects, morphisms,
$\ldots$, $n-1$-dimensional cubes as well as compositions of these
morphisms and compatibility conditions are determined by the object
$\bX$.

\begin{remark}
The pattern $\simp^{n,\op,\natural}$ is extendable by \cite[Example 3.3.22]{envelopes}, so we get a formula for free algebroids, that is free enriched $n$-fold \icats{}, from \cref{obs:freealgdformula}. More generally, if $\mathcal{O}$ and $\mathcal{P}$ are enrichable patterns that are \emph{soundly extendable} in the sense of \cite[Definition 3.3.16]{envelopes} such that $\mathcal{O}^{\el}_{O/}$ and $\mathcal{P}^{\el}_{P/}$ are weakly contractible for all $O \in \mathcal{O}$, $P \in \mathcal{P}$, then $\mathcal{O} \times \mathcal{P}$ is extendable by \cite[Lemma 3.3.21]{envelopes}.
\end{remark}

\subsection{Wreath products and enriched $(\infty, n)$-categories}\label{subsec:wreath}
As our second new example of an enriched structure, we consider enriched $(\infty,n)$-categories via the algebroid analogue of Rezk's $\bbTheta_{n}$-spaces. Here $\bbTheta_{n}$ is Joyal's $n$-disc category, which was identified by Berger \cite{Berger} with an inductive construction as the $n$-fold \emph{wreath product} of $\simp$. We therefore start with a brief discussion of wreath products of enrichable patterns in general; this is a minor variant of Lurie's construction of wreath products of \iopds{} in \cite[\S 2.4.4]{ha}.

We first recall the definition of the cocartesian \iopd{} from 
 \cite[Construction 2.4.3.1]{ha}:
 \begin{propn}\label{def Camalg}
   Let $\Gamma^{*}$ denote the full subcategory of $(\xFs)_{\angled{1}/}$ spanned by the active maps $\angled{1} \actto \angled{n}$. Then
   for any \icat{} $\mathcal{C}$, there exists an \icat{}
   $\mathcal{C}^{\amalg}$ over $\xFs$ characterized by the universal
   property that for any $\mathcal{K}$ over $\xFs$ there is a natural
   equivalence
   \[ \Map_{/\xFs}(\mathcal{K}, \mathcal{C}^{\amalg}) \simeq
     \Map(\mathcal{K} \times_{\xFs} \Gamma^{*}, \mathcal{C}),
   \]
   where the fibre product is over the forgetful functor
  $\Gamma^{*} \to \xFs$. \qed
\end{propn}
\begin{observation}
  An object of $\mathcal{C}^{\amalg}$ lying over $\angled{n}$ corresponds to a map from $\{\angled{n}\} \times_{\xF} \Gamma^{*}$, which we can identify with the discrete set of $n$ active maps $\angled{1} \to \angled{n}$. Thus an object over $\angled{n}$ can be described as a list $(c_{1},\ldots,c_{n})$ of objects $c_{i} \in \mathcal{C}$; we will denote this object as $\angled{n}(c_{1},\ldots,c_{n})$. Similarly, a
  morphism in $\mathcal{C}^{\amalg}$ lying over a given morphism $\phi \colon \angled{n} \to \angled{m}$ in $\xFs$ corresponds to a diagram in $\mathcal{C}$ of shape $[1] \times_{\xFs} \Gamma^{*}$. We can identify this category with the partially ordered set of pairs $(i,0)$ for $i = 1,\ldots,n$ and $(j,1)$ for $j = 1,\ldots,m$ where $(i,0) \leq (j,1)$ \IFF{} $\phi(i) = j$. In other words, a morphism $\angled{n}(c_{1},\ldots,c_{n})\to \angled{m}(c'_{1},\ldots, c'_{m})$ consists of a morphism $\phi \colon \angled{n} \to \angled{m}$ in $\xFs$ and a list of morphisms $c_{i} \to c'_{\phi(i)}$ for each $i$ where $\phi(i) \neq 0$; we will denote this morphism as $\phi(f_{i})$.
\end{observation}

\begin{example}
  It follows from the construction that $\mathbb{F}_*^{\amalg}$ is the
  category such that
	\begin{enumerate}
        \item an object is a finite sequence $(\angled{m_1},\ldots,\angled{m_k})$,
        \item a morphism $(\angled{m_1},\ldots,\angled{m_k})\to (\angled{n_1},\ldots,\angled{n_l})$ is determined by a morphism $\phi\colon \angled{k}\to \angled{l}$ in $\mathbb{F}_*$ and a set of morphisms $\{f_i\colon \angled{m_i}\to \angled{n_j} : \phi(i)=j\}$ in $\xF_{*}$.
	\end{enumerate}
\end{example}

\begin{propn}\label{propn:Camalgfact}
  Suppose $\mathcal{C}$ is an \icat{} equipped with a factorization system $(\mathcal{C}_{L}, \mathcal{C}_{R})$. Then $\mathcal{C}^{\amalg}$ has a factorization system $(\mathcal{C}^{\amalg}_{L}, \mathcal{C}^{\amalg}_{R})$ where a morphism $(\phi \colon \angled{n} \to \angled{m}, f_{i} \colon c_{i} \to c'_{\phi(i)})$ lies
  \begin{itemize}
  \item in $\mathcal{C}^{\amalg}_{L}$ if $\phi$ is inert and each $f_{i}$ lies in $\mathcal{C}_{L}$,
  \item in $\mathcal{C}^{\amalg}_{R}$ if $\phi$ is active and each $f_{i}$ lies in $\mathcal{C}_{R}$.
  \end{itemize}
\end{propn}
\begin{proof}
  We check the conditions of \cite[Definition 5.2.8.8]{ht}: If a
  morphism $\phi'(f'_{j})$ in $\mathcal{C}^{\amalg}$ is a retract of
  $\phi(f_{i})$ in $\mathcal{C}^{\amalg}_{L}$, then $\phi'$ is a
  retract of the inert map $\phi$ and so is again inert; moreover,
  each component $f'_{j}$ is a retract of some $f_{i}$ in
  $\mathcal{C}_{L}$ and so lies in $\mathcal{C}_{L}$. Similarly,
  $\mathcal{C}_{R}^{\amalg}$ is closed under retracts, which gives
  condition (1). To show (2), that $\mathcal{C}_{L}^{\amalg}$ is left
  orthogonal to $\mathcal{C}_{R}^{\amalg}$, we observe that a lift of
  a map in $\mathcal{C}_{L}^{\amalg}$ against one in
  $\mathcal{C}_{R}^{\amalg}$ is given by a lift of an inert map
  against an active map in $\xFs$ together with a set of lifts of
  maps in $\mathcal{C}_{L}$ against ones in
  $\mathcal{C}_{R}$. Finally, for (3) we observe that for a map
  $\phi(f_{i}) \colon \angled{n}(c_{i}) \to \angled{m}(c'_{j})$ we can
  pick an inert--active factorization of $\phi$ as
  $\angled{n} \xto{\alpha} \angled{k} \xto{\beta} \angled{m}$ and then
  factor $\phi(f_{i})$ as
  $\beta(f_{\beta^{-1}j}^{R})\circ \alpha(f_{i}^{L})$ where
  $f_{i}^{R} \circ f_{i}^{L}$ is the
  $(\mathcal{C}_{L},\mathcal{C}_{R})$-factorization of $f_{i}$.
\end{proof}

\begin{cor}
  Suppose $\mathcal{O}$ is an algebraic pattern. Then
  $\mathcal{O}^{\amalg}$ is an algebraic pattern over $\xFn$ when
  equipped with the factorization system induced as in
  \cref{propn:Camalgfact} by the inert--active factorization system on
  $\mathcal{O}$, and with the elementary objects given by
  $\angled{0}()$ and $\angled{1}(E)$ where $E \in
  \mathcal{O}^{\el}$. Moreover, any morphism of algebraic patterns
  $\mathcal{O} \to \mathcal{P}$ induces a morphism
  $\mathcal{O}^{\amalg} \to \mathcal{P}^{\amalg}$ of algebraic
  patterns over $\xFn$. \qed
\end{cor}

\begin{defn}\label{def wr}
  Let $\mathcal{O}$ be an algebraic pattern over $\xFn$, and let $\mathcal{P}$ be any algebraic pattern. Their \emph{wreath product} is the algebraic pattern
  \[ \mathcal{O} \wr \mathcal{P} := \mathcal{O} \times_{\xFn}\mathcal{P}^\amalg, \]
  defined using the projection $\mathcal{P}^{\amalg} \to \xFn$.
\end{defn}

We now want to make the wreath product of two enrichable patterns into
an enrichable pattern; for this we need to look at a different map from $\mathcal{P}^{\amalg}$ to
$\xFn$:

\begin{lemma}
  Let $\vee\colon \xFs^{\amalg}\to \xFs$ be the functor given on objects by
\[ \angled{n}(\angled{m_1},\ldots,\angled{m_n})\mapsto \bigvee_{i=1}^{n} \angled{m_{i}} \cong \angled{m_1+\ldots +m_n} \]
and on a morphism $\phi(\psi_{i}) \colon \angled{n}(\angled{m_{i}}) \to \angled{n'}(\angled{m'_{j}})$ by
\[ \vee(\phi(\psi_{i}))(t) =
  \begin{cases}
    \psi_{i}(t) \in \angled{m'_{\phi(i)}} \subseteq \vee \angled{n'}(\angled{m'_{j}}), & \phi(i) \neq 0, \\
    0, & \text{otherwise.}
  \end{cases}
\]
Then $\vee$ is a morphism of algebraic patterns
$\xFs^{\natural,\amalg} \to \xFn$, and this exhibits $\xFs^{\natural,\amalg}$ as an enrichable pattern.
\end{lemma}
\begin{proof}
  It is easy to see from the definition that $\vee$ preserves inert and active morphisms, and it preserves elementary objects since in $\xFs^{\natural,\amalg}$ these are by definition $\angled{0}()$ and $\angled{1}(\angled{1})$. Similarly, it is clear that
  \[ \xF^{\amalg,\flat,\el}_{\angled{n}(\angled{m_{1}},\ldots,\angled{m_{n}})} \to \xF^{\flat,\el}_{\angled{m_{1}+\cdots+m_{n}}/}\]
  is an isomorphism, since an inert map from $\angled{n}(\angled{m_{1}},\ldots,\angled{m_{n}}$ to $\angled{1}(\angled{1})$ consists of an inert map $\rho_{i} \colon \angled{n} \to \angled{1}$ together with an inert map $\rho_{j}\colon \angled{m_{i}} \to \angled{1}$.
\end{proof}

\begin{lemma}
  Suppose $\mathcal{O}$ and $\mathcal{P}$ are enrichable patterns. Then the composite
  \[ \mathcal{O} \wr \mathcal{P} \to \mathcal{P}^{\amalg} \to \xFs^{\natural,\amalg} \xto{\vee} \xFn\]
  makes $\mathcal{O} \wr \mathcal{P}$ an enrichable pattern.
\end{lemma}
\begin{proof}
  We must check that $(\mathcal{O} \wr \mathcal{P})^{\flat}$ is a cartesian pattern. The elementary objects of $(\mathcal{O} \wr \mathcal{P})^{\flat}$ are those of the form $E(E')$ where $E \in \mathcal{O}^{\flat,\el}$ and $E' \in \mathcal{P}^{\flat,\el}$; an inert map $O(P_{1},\ldots,P_{n}) \intto E(E')$ consists of an inert map $O \intto E'$ in $\mathcal{O}$ over some $\rho_{i}$ together with an inert map $P_{i} \intto E'$. More precisely, from the definition of the wreath product pattern we get equivalences
  \[ (\mathcal{O} \wr \mathcal{P})^{\flat,\el}_{O(P_{1},\ldots,P_{n})/} \simeq \mathcal{O}^{\flat,\el}_{O/} \times_{(\xFs^{\flat})^{\el}_{\angled{n}/}} \mathcal{P}^{\amalg,\flat,\el}_{\angled{n}(P_{1},\ldots,P_{n})},\]
  \[ \mathcal{P}^{\amalg,\flat,\el}_{\angled{n}(P_{1},\ldots,P_{n})} \simeq \prod_{i=1}^{n} \mathcal{P}^{\flat,\el}_{P_{i}/}.\]
  Since both $\mathcal{O}$ and $\mathcal{P}$ are enrichable patterns we see that the map
  \[ (\mathcal{O} \wr \mathcal{P})^{\flat} \to \xFs^{\amalg,\flat}\]
  is an iso-Segal morphism; since $\vee \colon \xFs^{\amalg,\flat} \to \xFs^{\flat}$ is an iso-Segal morphism it follows that $(\mathcal{O} \wr \mathcal{P})^{\flat}$ is a cartesian pattern, as required.
\end{proof}

\begin{observation}\label{obs:wr0}
  If $\mathcal{O}$ and $\mathcal{P}$ are enrichable patterns, then we have an equivalence  \[(\mathcal{O} \wr \mathcal{P})_{0} \simeq \mathcal{O} \wr (\mathcal{P}_{0})\] of algebraic patterns.
\end{observation}

\begin{observation}
  An object in $\mathcal{O}\wr \mathcal{P}$ consists of an object
  $O \in \mathcal{O}$ together with a list of objects
  $P_{1},\ldots,P_{n} \in \mathcal{P}$ where $|O| \cong \angled{n}$;
  we will denote this object by $O(P_{1},\ldots, P_{n})$. Similarly a
  morphism in $\mathcal{O}\wr \mathcal{P}$ from $O(P_{i})$ to
  $O'(P'_{j})$ is given by a map $f\colon O\to O'$ in $\mathcal{O}$
  together with a sequence of morphisms
  $\{g_i\colon P_j\to P'_{|f|(i)}\}$ in $\mathcal{P}$.
\end{observation}

\begin{remark}
  Suppose $\mathcal{O}, \mathcal{P}, \mathcal{Q}$ are enrichable
  patterns. Since the wreath product of enrichable patterns is
  enrichable, we obtain enrichable patterns
  $\mathcal{O} \wr (\mathcal{P} \wr \mathcal{Q})$ and
  $(\mathcal{O} \wr \mathcal{P}) \wr \mathcal{Q}$. It is not hard to
  show that these are naturally equivalent, \ie{} that the wreath
  product of enrichable patterns is associative. With a bit more
  effort we expect that one can make the wreath product a monoidal
  structure on the \icat{} of enrichable patterns, as done by Barwick
  for the wreath product of operator categories \cite[Proposition
  3.9]{bar}, but we do not pursue this here as we only care about a single instance of the construction.
\end{remark}

\begin{defn}
  The enrichable patterns $\Tnopn$ are defined inductively by $\Tqopn{1} := \Dopn$ and $\Tnopn := \Dopn \wr \Tqopn{n-1}$. 
\end{defn}

\begin{observation}
  From \cref{obs:wr0} we see that $(\Tnopn)_{0}$ is $\Tqopn{n-1}$.
\end{observation}

Segal $\Tnopn$-spaces are (ignoring completeness) Rezk's model for
$(\infty,n)$-categories \cite{RezkThetan}. We therefore regard
$\Tnopn$-algebroids as \emph{enriched $(\infty,n)$-categories}. Such
an algebroid is given by a Segal $\Tqopn{n-1}$-space $\bX$, which is
the \emph{underlying $(\infty,n-1)$-category} that specifies the
objects, morphisms, 2-morphisms, etc., up to $(n-1)$-morphisms, while
the $n$-morphisms are given by an enrichment in an $E_{n}$-monoidal
\icat{} $\mathcal{V}$. A bit more precisely, an object $I$ of $\Tn$
can be thought of as a ``pasting diagram'', that is an $n$-category
that specifies a shape that can be composed in any
$(\infty,n)$-category, and an object of $\TqXop{n}{\bX}$ over $I$ can be
described as a functor $\partial I \to \bX$, where $\partial I$ is the
underlying $(n-1)$-category of $I$; in other words, this is a
labelling of the objects, morphisms, etc. up to $(n-1)$-morphisms of
$I$ by such data from $\bX$.

Let us spell this out a bit more explicitly in the case $n = 2$.
An object in $\bbTheta_2^\op=\Dop\wr\Dop$ is of the form $[n]([m_1], \ldots, [m_n])$, which we think of as a (strict) 2-category
with objects $0,\ldots,n$ and with the poset $[m_{i}]$ as the category of morphisms from $i-1$ to $i$ (with the other morphism categories freely generated by these under composition). For instance, the object $[4]([0], [3], [1],[2])$ is given by the following 2-category:
\begin{equation}\label{dia Theta}
\begin{tikzcd}
  \bullet \ar[r] &
  \bullet \ar[r, bend right=85, ""{name=D, inner sep=1pt, above}] \ar[r, bend right=25, ""{name=C, inner sep=1pt, above}, ""{name=CC, inner sep=1pt, below}] \ar[r, bend left=25, ""{name=B, inner sep=1pt, above}, ""{name=BB, inner sep=1pt, below}] \ar[r, bend left=85, ""{name=AA, inner sep=1pt, below}]
  \ar[Rightarrow, from=AA, to=B]
  \ar[Rightarrow, from=BB, to=C]
    \ar[Rightarrow, from=CC, to=D]
  &
  \bullet \arrow[r, bend left=30, ""{name=EU, inner sep=1pt, below}] \arrow[r, bend right=30, ""{name=EL, inner sep=1pt, above}]
    \ar[Rightarrow, from=EU, to=EL]
    & \bullet \arrow[r, ""{name=FL, inner sep=1pt, above}, ""{name=GU, inner sep=1pt, below}] \arrow[r, bend right=45, ""{name=GL, inner sep=1pt, above}] \arrow[r, bend left=45, ""{name=FU, inner sep=1pt, below}]
    \ar[Rightarrow, from=FU, to=FL]
        \ar[Rightarrow, from=GU, to=GL]
    & \bullet.
\end{tikzcd}
\end{equation}
The pattern $(\Tqopn{2})_{0}$ is $\Dopn$, and the three elementary objects of
$\Tqopn{2}$ are $[0]()$, $[1]([0])$ and $[1]([1])$, corresponding to the 2-categories
\[
\begin{tikzcd}
\bullet, & \bullet \arrow[r] & \bullet, &  \bullet \arrow[r, bend left, ""{name=FU, inner sep=1pt, below}] \arrow[r, bend right, ""{name=FL, inner sep=1pt, above}] \arrow[Rightarrow, from=FU, to=FL] & \bullet,
\end{tikzcd}
\]
respectively, where the first two elementary objects lie over $0$ and can be identified with the two elementary objects in $\simp^{\op, \natural}$.

Given $\bX\in \Seg_{\Dopn}(\mathcal{S})$, meaning an \icat{} in the sense of a (not necessarily complete) Segal space \cite{Rezk}, an object of  $\TqXop{2}{\bX}$ can be regarded as a labelling of the objects and morphisms of an object of $\Tqop{2}$ by objects and morphisms in $\bX$; for the object illustrated above this is a diagram
\[\begin{tikzcd}
a_1 \arrow[r, "\phi_1"] & a_2 \arrow[r, bend right=20, "\phi_4"] \arrow[r, bend right=70, "\phi_5"]\arrow[r, bend left=20, "\phi_3"] \arrow[r, bend left=70, "\phi_2"]& a_3 \arrow[r, bend left, "\phi_6"] \arrow[r, bend right, "\phi_7"] & a_4 \arrow[r, "\phi_9"] \arrow[r, bend right, "\phi_{10}"] \arrow[r, bend left, "\phi_7"]& a_5,
\end{tikzcd}\]
where $a_i\in \bX([0])$ and $\phi_j\in \bX([1])$ (with appropriate sources and targets).
A $\Tqopn{2}$-algebroid $A$ in an $E_{2}$-monoidal \icat{} $\mathcal{V}$ then models an enriched $(\infty,2)$-category whose underlying \icat{} is given by some Segal space $\bX\in \Seg_{\Dopn}(\mathcal{S})$; in particular, $A$ takes a $2$-cell whose boundary is labelled as
\begin{tikzcd}
 a \arrow[r, bend left, "\phi"] \arrow[r, bend right, "\psi"] & b
\end{tikzcd}
to an object in $\mathcal{C}$, which gives the enriched 2-morphisms from $\phi$ to $\psi$. The horizontal and vertical compositions of $2$-cells are then determined by the structure of $\TqXop{2}{\bX}$.

\begin{remark}
The pattern $\Tnopn$ is extendable for all $n$ by \cite[Example 8.15]{patterns}, so we get a formula for its free algebroids, that is free enriched $(\infty,n)$-categories, from \cref{obs:freealgdformula}.
\end{remark}

\begin{remark}
  For any \icat{} $\mathcal{C}$, the projection $\mathcal{C}\times \Gamma^{*} \to  \mathcal{C}$ corresponds under the equivalence of Definition~\ref{def Camalg} to a functor $\mathcal{C} \times \xFs \to \mathcal{C}^{\amalg}$, which takes an object $(C,\angled{n})$ to $\angled{n}(C,\ldots,C)$. For enrichable patterns $\mathcal{O}, \mathcal{P}$, pulling back this map for $\mathcal{P}$ gives a functor
  \[ \delta \colon \mathcal{O} \times \mathcal{P} \to \mathcal{O} \wr \mathcal{P}, \]
  which is a morphism of enrichable patterns since the triangle
  \[
    \begin{tikzcd}
      \xFs \times \xFs \arrow{dr}[swap]{\wedge} \arrow{rr}{\delta} & & \xFs^{\amalg} \arrow{dl}{\vee} \\
       & \xFs
    \end{tikzcd}
  \]
  commutes. As a special case, we have a morphism of enrichable
  patterns $\delta_{n} \colon \Dnopn \to \Tnopn$ for every $n$; this
  is the functor that induces the eequivalence between Rezk's
  $\bbTheta_{n}$-spaces and Barwick's $n$-fold Segal spaces
  \cite{BSP,BergnerRezk,thetan}, which are Segal $\Dnopn$-spaces that
  satisfy an additional constancy condition. We expect that in the
  enriched setting we can similarly identify $\Tnopn$-algebroids with
  a full subcategory of $\Dnopn$-algebroids via $\delta_{n}$. This
  comparison should also lead to an equivalence between
  $\Tnopn$-algebroids and enriched $(\infty,n)$-categories defined
  inductively as \icats{} enriched in enriched
  $(\infty,n-1)$-categories, as in \cite{enriched}; we intend to
  return to these comparisons elsewhere.
\end{remark}

\begin{variant}\label{var:opcat}
  Both of the last two examples $\bbTheta_{n}^{\op,\natural}$ and
  $\simp^{n,\op,\natural}$ are special cases of the following general
  construction of enrichable patterns, due to Barwick: Suppose $\Phi$
  is a \emph{perfect operator category} as introduced in \cite{bar}, and
  let $\Lambda(\Phi)$ be its Leinster category, which is the Kleisli
  category of a certain monad on $\Phi$. By \cite[Example
  3.6]{patterns} there is an algebraic pattern structure
  $\Lambda(\Phi)^\natural$ which is compatible with the map
  $|\blank|\colon \Lambda(\Phi)^\natural\to
  \Lambda(\xF)^\natural\simeq \xF_*^\natural$ induced by the unique
  operator morphism $\Phi\to \xF$ (which takes $X$ to the cardinality of the
  finite set $\xHom_{\Phi}(*,X)$). An object
  $E\in\Lambda(\Phi)^\natural$ is elementary if it admits an inert map
  $*\to E$ in $\Lambda(\Phi)^\natural$ where $*$ is the terminal
  object in $\Phi$. The preservation of inert morphisms shows that $*$
  is the unique elementary object $\Lambda(\Phi)^\natural$ lying over
  $\angled{1}$. It then follows from the description of inert
  morphisms in $\Lambda(\Phi)^\natural$ (see \cite[Definition
  8.1]{bar}) that
  $\Lambda(\Phi)^{\flat,\el}_{X/}\simeq \xF_{*,|X|/}^{\flat,
    \el}$. Hence, for every perfect operator category $\Phi$ we have
  an enrichable pattern $\Lambda(\Phi)^\natural$ which then allows us
  to define $\Lambda(\Phi)$-algebroids. By choosing $\Phi$ to be the
  operator categories $\xF$ of finite sets, $\mathbb{O}$ of finite
  ordered sets, its cartesian power $\mathbb{O}^{n}$ and its wreath
  power $\mathbb{O}^{\wr n}$ we obtain the enrichable patterns
  $\xF^\natural_*$, $\simp^{\op,\natural}$, $\simp^{\op,n, \natural}$
  and $\Theta_n^{\op, \natural}$ and the associated algebroids
  discussed above.
\end{variant}

\subsection{Enriched modular $\infty$-operads}\label{subsec:modular}
A \emph{modular operad} is, roughly speaking, an operad-like structure
where operations can have multiple inputs and outputs, but there is a
duality on the set of objects that allows us to replace an input
object by its dual as an output object. They were first introduced by
Getzler and Kapranov~\cite{GetzlerKapranov} in order to describe the
``operadic'' structure on the homology of moduli spaces of stable
curves with marked points. Hackney, Robertson, and Yau
\cite{HRYModular} have defined an \icatl{} version of modular operads
as presheaves on a category $\bU$ of connected graphs that satisfy
Segal conditions. As our final example of an enriched
homotopy-coherent structure we will consider the corresponding notion
of enriched modular \iopds{}.

We begin by sketching the definition of the category $\bU$, for which
we follow the simplified (but equivalent) definition given by
Hackney~\cite{HackneyCats}; for more details we recommend the
exposition in \cite{HackneySeg}.

\begin{defn}
  An \emph{undirected graph} (with loose ends) $G$ is a diagram of
  finite sets
    \[
    \begin{tikzcd}
      A(G) \arrow[l, loop left, "\dagger"] & D(G) \arrow[l,
      hookrightarrow] \arrow[r, "t"] &  V(G)
    \end{tikzcd}
  \]
  where $\dagger$ is a fixed-point free involution and the map $D(G)
  \hookrightarrow A(G)$ is a monomorphism. Here $A(G)$ is the set of
  \emph{arcs} of $G$ and $V(G)$ is the set of \emph{vertices}. The
  \emph{boundary} of $G$ is the set $A(G)\setminus D(G)$ and the set
  $E(G)$ of \emph{edges} is the set of pairs $\{a,a^{\dagger}\}$,
  \ie{} of $\dagger$-orbits.
\end{defn}

\begin{defn}
  To each undirected graph $G$ we can associate a topological space
  that is its ``geometric realization'', and we say that $G$ is
  \emph{connected} if this topological space is connected; this
  condition can also be defined purely combinatorially
  \cite[Definition 2.10]{HackneySeg}.
\end{defn}

\begin{examples}\ 
  \begin{enumerate}[(i)]
  \item The \emph{undirected edge} $\eta$ is the undirected graph
    \[
    \begin{tikzcd}
      \textbf{2} \arrow[l, loop left, "\dagger"] & \emptyset \arrow[l,
      hookrightarrow] \arrow[r] &  \emptyset
    \end{tikzcd}
  \]
  with two arcs interchanged by the involution, both of which lie on
  the boundary, and with no vertices.
\item The \emph{$n$-star} $\star_{n}$ is the undirected graph
  \[
    \begin{tikzcd}
      \textbf{n} \amalg \textbf{n} \arrow[l, loop left, "\dagger"] & \textbf{n} \arrow[l,
      hookrightarrow] \arrow[r] &  \textbf{1}
    \end{tikzcd}    
  \]
  \end{enumerate}
\end{examples}

\begin{defn}
  Let $G,G'$ be connected undirected graphs.
  An \emph{embedding} of $G'$ in $G$ is a commutative diagram
  \[
    \begin{tikzcd}
      A(G') \arrow{d} & D(G') \arrow{d} \arrow[phantom]{dr}[very near start,description]{\lrcorner} \arrow[hookrightarrow]{l} \arrow{r} & V(G')\arrow[hookrightarrow]{d} \\
      A(G) & D(G) \arrow[hookrightarrow]{l} \arrow{r} & V(G)
    \end{tikzcd}
  \]
  such that the right square is a pullback, the map $A(G') \to A(G)$ is involutive, and the map
  $V(G') \to V(G)$ is injective. We write $\Emb(G)$ for the set of
  isomorphism classes of embeddings of graphs in $G$.
\end{defn}

\begin{observation}
  Let $G$ be a connected undirected graph. We can identify the set
  $E(G)$ of edges of $G$ with the set of embeddings of $\eta$ in $G$,
  and the set $V(G)$ of vertices with the set of embeddings of
  $\star_{n}$ for all $n$. In particular, we have inclusions
  $E(G),V(G) \hookrightarrow \Emb(G)$.
\end{observation}
\begin{defn}[{\cite[Definition 4.1]{HackneySeg}}]
  Let $G,G'$ be connected undirected graphs. A \emph{graph map} $\phi
  \colon G \to
  G'$ consists of
  \begin{itemize}
  \item an involutive function $\phi_{0} \colon A(G) \to A(G')$
    (\ie{} $\phi_{0} \dagger = \dagger \phi_{0}$),
  \item a function $\hat{\phi} \colon \Emb(G) \to \Emb(G')$
  \end{itemize}
  such that
  \begin{itemize}
  \item $\hat{\phi}$ sends edges to edges,
  \item $\hat{\phi}$ preserves unions and takes vertex disjoint pairs
    to vertex disjoint pairs (\cite[Definition 2.18]{HackneySeg}),
  \item $(\phi_{0},\hat{\phi})$ is boundary-compatible
    (\cite[Definition 2.19]{HackneySeg}).
  \end{itemize}
\end{defn}

\begin{defn}
  The category $\bU$ has connected undirected graphs as objects and
  graph maps as morphisms; these compose simply by composing their two
  component functions.
\end{defn}

\begin{defn}
  A graph map $\phi \colon G \to G'$ in $\bU$ is \emph{active} if
  $\hat{\phi}(\id_{G}) = \id_{G'}$ and \emph{inert} if the composite
  $V(G) \hookrightarrow \Emb(G) \xto{\hat{\phi}} \Emb(G')$ factors
  through $V(G')$. Let $\bU^{\act}, \bU^{\xint}$ denote the
  wide subcategories of $\bU$ containing the active and inert
  morphisms, respectively.
\end{defn}

\begin{observation}
  Inert maps $G' \to G$ can be identified with embeddings of $G'$ in $G$.
\end{observation}

\begin{thm}[{\cite[Theorem 2.15]{HRYModular}}]\label{graphfact}
  $(\bU^{\act}, \bU^{\xint})$ is a factorization system on $\bU$. \qed
\end{thm}

\begin{defn}
  We define an algebraic pattern $\bUopn$ on $\bU^{\op}$
  using the inert--active factorization system from \cref{graphfact},
  with the elementary objects being $\eta$ and $\star_{n}$ for all
  $n$.
\end{defn}

Our goal is now to show that the pattern $\bUopn$ is enrichable, for
which we must define a functor to $\xFs$:
\begin{construction}
  We define the functor $|\blank|$ from $\bUop$ to $\xFs$ on objects
  by assigning to a connected undirected graph $G$ the set $V(G)_{+}$
  of vertices of $G$, with an added base point $*$. For a graph map $\phi
  \colon G' \to G$ we define $|\phi| \colon V(G)_{+} \to V(G')_{+}$ by
  \[
    |\phi|(v) =
    \begin{cases}
      w \in V(G'), & \text{if } v \in V_{\hat{\phi}(w)} \subseteq V(G'), \\
      *, & \text{otherwise.}
    \end{cases}
  \]
  Checking that this is compatible with composition boils down to the
  observation that for $u$ a vertex of $G$, a vertex $w$ of
  $\hat{\phi}'\hat{\phi}(u)$ must lie in $\hat{\phi}'(v)$ for a unique
  vertex $v$ of $\hat{\phi}(u)$, so that we get
  \[ |\phi|(|\phi'|(w)) = |\phi|(v) = u = |\phi'\phi|(w). \]
\end{construction}

\begin{lemma}
  The functor $|\blank|$ is a morphism of algebraic patterns
  $\bUopn \to \xFn$, and exhibits $\bUopn$ as an enrichable pattern.
\end{lemma}
\begin{proof}
  If $\phi \colon G' \to G$ is an active morphism in $\bU$, so that
  $\hat{\phi}(\id_{G'}) = \id_{G}$, then every vertex of $G$ lies in
  $\hat{\phi}(v)$ for some vertex $v$ of $G'$. Thus $|\phi|$ is
  active. On the other hand, if $\phi$ is inert, then $\hat{\phi}(v)$
  is a vertex of $G$ for every $v \in V(G')$. This means there is at
  most one vertex $w$ of $G$ such that $|\phi|(w) = v$, so that
  $|\phi|$ is inert. Since we clearly also have $|\phi|(\eta) \cong
  \angled{0}$ and $|\phi|(\star_{n}) \cong \angled{1}$, this shows
  that $|\blank|$ is a morphism of algebraic pattern. It remains to
  check that $|\blank|$ induces for every $G \in \bU$ an equivalence
  \[ (\bU^{\op,\flat})^{\el}_{G/} \isoto
    \xF^{\flat,\el}_{*,V(G)_{+}/},\]
  but this just means that there is a unique embedding of a star in
  $G$ for every vertex.
\end{proof}

The edge $\eta$ is the only connected undirected graph with no
vertices, so $\bU^{\op}_{0}$ is the full subcategory on this
object. There is a single non-trivial endomorphism of $\eta$, which is
the order-2 automorphism that swaps the two arcs, so that
$\bU^{\op}_{0} \simeq BC_{2}$. Thus
\[\Seg_{\bU^{\op,\natural}_{0}}(\mathcal{S}) \simeq \Fun(BC_{2},
  \mathcal{S}) \]
is the \icat{} of $C_{2}$-spaces. For a $C_{2}$-space $\bX$, we can
describe the objects of $\bU^{\op}_{\bX}$ as connected undirected
graphs whose arcs are labelled by points of $\bX$, so that the
labelling of $a^{\dagger}$ is obtained from the labelling of $a$ via
the $C_{2}$-action on $\bX$. If $\mathcal{V}$ is a symmetric monoidal
\icat{}, then a $\bU^{\op,\natural}$-algebroid $A$ in $\mathcal{V}$
over the $C_{2}$-space $\bX$ describes a $\mathcal{V}$-enriched
modular $\infty$-operad whose $C_{2}$-space of objects is $\bX$.

\begin{variant}
  Hackney, Robertson, and Yau have defined a number of other
  categories of graphs and trees that can also be used to define
  \icatl{} analogues of other variants of operads
  \cite{HRYProperad,HRYCyclic}, which have recently
  been given a unified presentation by Hackney in
  \cite{HackneyCats,HackneySeg}. Using the latter description, these
  can be shown to be enrichable patterns by minor variants of our
  discussion for $\bU$, so we will not spell out any details
  here. Instead, we will just briefly mention some interesting cases and
  the corresponding homotopy-coherent structures, all of which can be enriched:
  \begin{itemize}
  \item The category $\bU$ has full subcategories
    \[ \bU_{\name{cyc}} \subseteq \bU_{0} \subseteq \bU\] where
    $\bU_{0}$ consists of undirected trees and $\bU_{\name{cyc}}$ of
    undirected trees with inhabited boundary. The enrichable pattern
    structure on $\bU^{\op}$ restricts to these full subcategories,
    which describe augmented cyclic \iopds{} and cyclic \iopd{},
    respectively. One can also consider the full subcategory of
    undirected linear graphs, which describes involutive \icats{}.
  \item There is a category $\bO$ of directed graphs, which can be
    described as the category of elements $\bU_{/\mathfrak{o}}$ where
    $\mathfrak{o}$ is the ``orientation presheaf'' on $\bU$. This has
    a similar enrichable pattern structure, for which Segal objects in
    $\mathcal{S}$ are wheeled $\infty$-properads.
  \item The category $\bO$ has full subcategories
    \[ \simp \subseteq \bbO \subseteq \bO_{0} \subseteq \bO,\]
    where the objects of $\bO_{0}$ are the simply connected
    directed graphs (\ie{} trees, but where vertices need not have a
    single output edge), the objects of $\bbO$ are the directed trees where
    vertices \emph{must} have a single output, and the objects of
    $\simp$ are the linear graphs. The enrichable pattern structure on
    $\bO^{\op}$ restricts to $\bO_{0}^{\op}$, whose Segal objects in
    $\mathcal{S}$ are $\infty$-dioperads (or ``symmetric
    poly-\icats{}''); this pattern structure restricts further to the
    standard patterns $\bbO^{\op,\natural}$ and $\Dopn$.
  \end{itemize}
\end{variant}

\providecommand{\bysame}{\leavevmode\hbox to3em{\hrulefill}\thinspace}
\providecommand{\MR}{\relax\ifhmode\unskip\space\fi MR }
\providecommand{\MRhref}[2]{%
  \href{http://www.ams.org/mathscinet-getitem?mr=#1}{#2}
}
\providecommand{\href}[2]{#2}

\end{document}